\documentclass[11pt]{amsart}

\usepackage[margin=1.5in]{geometry}
\usepackage{cite}
\usepackage[utf8]{inputenc}
\usepackage[T1]{fontenc}
\usepackage[english]{babel}
\usepackage{amsmath,amsfonts,amsthm, mathrsfs}
\usepackage[linktocpage=true]{hyperref}
\usepackage{tikz-cd}
\usepackage{enumitem}
\usepackage{caption}
\usepackage{subcaption}

\theoremstyle{plain}
\newtheorem{theorem}{Theorem}[section]
\newtheorem{corollary}[theorem]{Corollary}
\newtheorem{lemma}[theorem]{Lemma}
\newtheorem{proposition}[theorem]{Proposition}
\newtheorem{conjecture}[theorem]{Conjecture}

\theoremstyle{definition}
\newtheorem{definition}[theorem]{Definition}

\newtheorem{remark}[theorem]{Remark}
\numberwithin{equation}{section}
\numberwithin{figure}{section}
\numberwithin{table}{section}

\DeclareMathOperator{\PP}{\operatorname{\mathbb P}}
\newcommand{\R}[0]{\mathbb R}

\newcommand{\Z}[0]{\mathbb Z}

\newcommand{\N}[0]{\mathbb N}
\newcommand{\HH}[0]{\mathbb H}
\newcommand{\Ab}[0]{\mathbf A}
\newcommand{\Bb}[0]{\mathbf B}
\newcommand{\Cb}[0]{\mathbf C}
\newcommand{\Db}[0]{\mathbf D}
\newcommand{\id}[0]{\mathrm{id}}
\newcommand{\deq}[0]{\overset{(d)}{=}}

\DeclareMathOperator{\rot}{180^\circ}
\DeclareMathOperator{\rev}{rev}
\DeclareMathOperator{\wt}{wt}
\DeclareMathOperator{\sat}{SAT}
\DeclareMathOperator{\Ht}{Ht}
\DeclareMathOperator{\Li}{Li_2}
\DeclareMathOperator{\Beta}{Beta}

\title{Shift invariance of half space integrable models}

\author{Jimmy He}
\address{Department of Mathematics, MIT, Cambridge, MA  02139}
\email{jimmyhe@mit.edu}
\begin{document}
\maketitle
\begin{abstract}
We formulate and establish symmetries of certain integrable half space models, analogous to recent results on symmetries for models in a full space. Our starting point is the colored stochastic six vertex model in a half space, from which we obtain results on the asymmetric simple exclusion process, as well as for the beta polymer through a fusion procedure which may be of independent interest. As an application, we establish a distributional identity between the absorption time in a type $B$ analogue of the oriented swap process and last passage times in a half space, establishing the Baik--Ben Arous--P\'ech\'e phase transition for the absorption time. The proof uses Hecke algebras and integrability of the six vertex model through the Yang--Baxter and reflection equations.
\end{abstract}
\tableofcontents
\section{Introduction}
Many integrable stochastic models related to the full space KPZ universality class have been shown to exhibit remarkable symmetries. In particular, it has been shown that the joint distribution of certain collections of observables for these models are invariant under shifting of some, but not all, of these observables. These symmetries are non-trivial and in particular are not known to be a consequence of some bijection of the underlying model, except in some special cases.

While shadows of these symmetries appeared earlier (see e.g. work of Borodin and Wheeler \cite{BW18}), shift invariance in the generality considered in this paper was first formulated and established by Borodin, Gorin, and Wheeler \cite{BGW22}, who proved it for the stochastic six vertex model, and obtained results for other models including certain polymer models and the KPZ equation through various limits and specializations. Subsequently, these results were extended to other models and generalized \cite{D22,G21,K22, Kor22}. They have already found some applications, see \cite{BGR20, Z22}.

This paper initiates the study of these symmetries in a half space setting. While it is natural to hope that these symmetries continue to hold, even formulating them in the half space setting is a non-trivial task, since the presence of the boundary complicates the way in which the observables may be shifted. In particular, we restrict our study to observables which ``cross the boundary'' in a certain sense. It is unclear if this restriction is completely necessary, although our proof only seems to work in this setting.

Half space integrable models related to the KPZ universality class have been of interest due to the rich phase diagram that is expected to hold for limiting distributions based on the boundary strength. Unfortunately, the analysis of half space models is usually much more complicated, and although much is now known (see e.g. \cite{BC11, BBC20, BBCW18}), even the limiting one point distribution at large times was only recently established by Imamura, Mucciconi, and Sasamoto \cite{IMS22} (with earlier results in the critical case \cite{BBCW18} or at zero temperature \cite{BR01b, BBCS18}). For some additional recent works on half space models, see \cite{BFO20, BC22, CS18, P19, W20}, and also relevant work in the physics literature \cite{KD20, BD21, BKD22, BBC16}.

We begin by studying shift invariance for the colored half space stochastic six vertex model. We prove this by establishing a different symmetry, a version of flip invariance (introduced in the full space setting by Galashin \cite{G21}), from which shift invariance can be derived. The proof follows the same strategy as \cite{G21}, although additional difficulties arise due to the boundary. We then take a limit of the six vertex model to the ASEP in a half space as well as the half space beta polymer, or equivalently random walk in beta random environment, recently introduced by Barraquand and Rychnovsky \cite{BR22}. The limit to the beta polymer goes through a fusion procedure which we carry out and may be of independent interest. Finally, we use shift invariance for the TASEP to give a distributional identity between the absorption times in a type $B$ (or half space) analogue of the oriented swap process and certain point to line exponential last passage percolation times. The latter are known to undergo the Baik--Ben Arous--P\'ech\'e (BBP) phase transition, and so we derive the same asymptotics for the former.

We begin with possibly the simplest consequence of our shift invariance, for exponential last passage percolation (LPP) in a half space setting. Note that while the statement is only for a single last passage time, it is still non-trivial, unlike in the full space setting.

Let $X_{x,y}$ for $x,y\in \N_0$\footnote{We use $\N$ to denote the positive integers, and $\N_0$ to denote the non-negative integers} be a symmetric collection of independent exponential random variables, of rate $1$ for $x>y$, rate $\mu$ for $x=y$, and set $X_{y,x}=X_{x,y}$. For two points $(x,y)$ and $(x',y')$ with $x'\geq x$ and $y'\geq y$, let $L((x,y),(x',y'))=\max_\gamma \sum_{\gamma_i}X_{\gamma_i}$, where the maximum is over up-right paths from $(x,y)$ to $(x',y')$, and $\gamma_i$ are the vertices in $\gamma$. We call $L$ the \emph{point-to-point} half space LPP time.

\begin{corollary}
\label{cor: lpp shift}
For $x,y,z\in \N$ with $x\geq y+z$, we have
\begin{equation*}
    L((0,0),(x,y))\deq L((0,z),(x,y+z)).
\end{equation*}
\end{corollary}
\begin{remark}
While Corollary \ref{cor: lpp shift} is a simple statement, and could possibly be proven by elementary means, let us note that it is not completely trivial, in that the statement does not hold if the random variables are not exponential, but say Gaussian, as can be seen from simulations. This is unlike a diagonal shift, which is immediate from the definition of the model and would hold for any distribution on the $X_{x,y}$.

The simplest non-trivial case is the following. Let $A,B,C,D,E$ be independent exponential random variables of rate $1$, and $F,G$ be independent exponential random variables of rate $\mu$. Then we have a distributional equality between the last passage times
\begin{equation*}
    L\left(\:\begin{tabular}{|c|c|c|c|}
         \hline A&B&E&F\\
         \hline C&D&G&E\\
         \hline
    \end{tabular}\:\right)\deq L\left(\:\begin{tabular}{|c|c|c|c|}
         \hline A&E&F&B\\
         \hline C&G&E&D\\
         \hline
    \end{tabular}\:\right),
\end{equation*}
where the notation means the maximum of the sums over up-right paths in the rectangles.
\end{remark}

\subsection{Shift invariance for the half space six-vertex model}
\label{sec: 6vm intro}
We begin with the statement of shift invariance for the colored half space stochastic six vertex model. We will only give an informal description of the model, see Section \ref{sec: 6vm} for the formal definitions.

\begin{figure}
    \centering
    \includegraphics{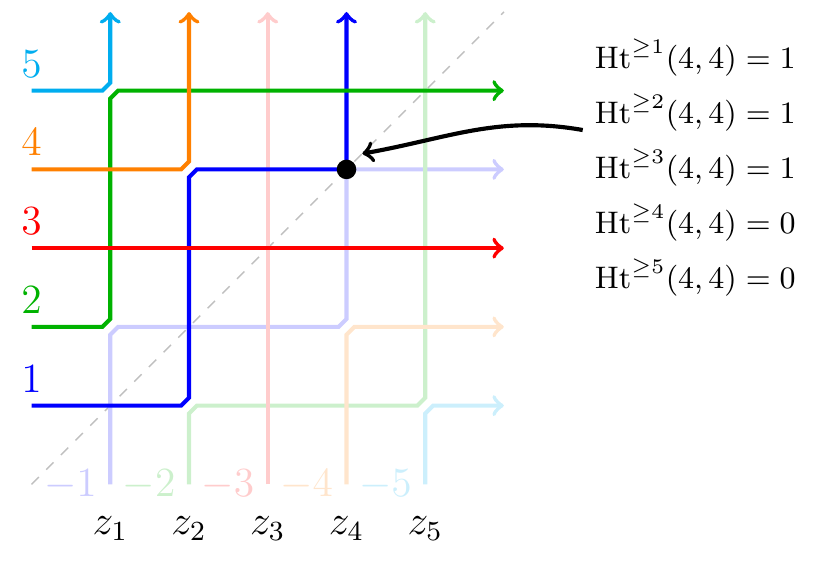}
    \caption{A configuration of the colored half space stochastic six vertex model. Positive and negative colored edges are represented by dark and light shades of the same color. Note that by symmetry, the positive lines determine the negative lines, and at the diagonal, a positive line always meets its negative counterpart. The rapidities associated to rows/columns (which are the same) are indicated, and the values of the height function at $(4,4)$ are given.}
    \label{fig:6vm}
\end{figure}

The model is a probability distribution on certain configurations of colored arrows (with colors in $\Z$) in the lattice $\N^2$ (see Figure \ref{fig:6vm}), which can be described with the following sampling procedure. We begin with an arrow of color $x$ entering the vertex $(1,x)$ from the right, and an arrow of color $-x$ entering the vertex $(x,1)$ from the bottom. We then sample the outcome of vertices $(x,y)$ with $x\geq y$ in a Markovian fashion, starting with vertices whose left and bottom arrows are determined, using certain integrable stochastic weights for the outcome. The outcome of vertices $(y,x)$ are determined by what occurs at $(x,y)$ in a symmetric manner. In particular, a color $i$ crosses an edge if and only if $-i$ crosses the corresponding edge given by reflection across $x=y$.

To define the weights, whose formulas are given in Section \ref{sec: 6vm}, we allow ourselves a parameter $z_i$ attached to the $i$th column/row, an asymmetry parameter $q$ for vertices $(x,y)$ with $x<y$, and both an asymmetry parameter $t$ and an intensity parameter $\nu$ for the vertices $(x,x)$ at the boundary. These weights are chosen so that the Yang--Baxter and reflection equations hold (see Propositions \ref{prop:YB} and \ref{prop:refl}). 

The observables we will consider are height functions. We let $\Ht^{\geq i}(x,y)$ denote the number of colored arrows passing at or under the vertex $(x,y)$, whose color is at least $i$ (referred to as a \emph{color cutoff}). When the dependence on the rapidities $\mathbf{z}=(z_1,z_2,\dotsc)$ is important, we write $\Ht^{\geq i}(x,y;\mathbf{z})$. See Figure \ref{fig:6vm} for an example. We say that a height function $\Ht^{\geq i}(x,y)$ \emph{crosses the diagonal} if $x\geq y$ and $i\geq 1$. We say that a collection of height functions is \emph{ordered} if all points $(x,y)$ are either weakly north-west or south-east of each other, i.e. the vertices $(x,y)$ lie on a down-right path. With this, we may now state a simplified version of shift invariance in this setting. See Figure \ref{fig:6vm shift} for an example in the case of two height functions.

\begin{figure}
    \centering
    \begin{subfigure}{0.45\textwidth}
    \includegraphics[scale=0.9]{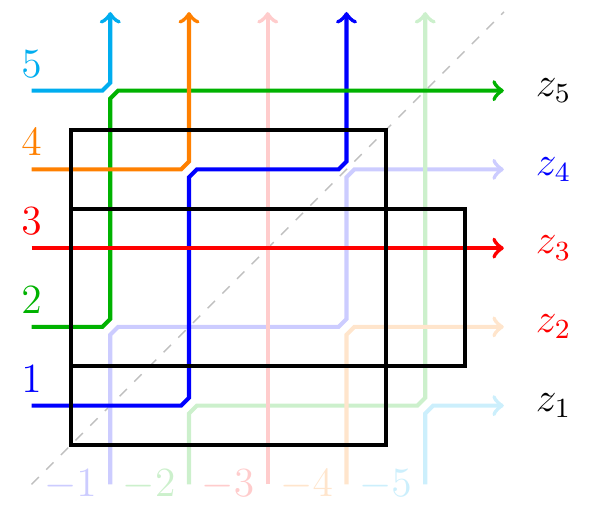}
    \caption{Unshifted height functions}
    \end{subfigure}
    \begin{subfigure}{0.45\textwidth}
    \includegraphics[scale=0.9]{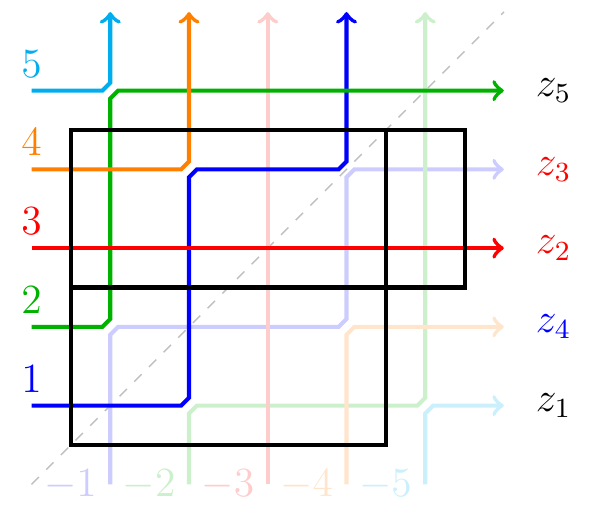}
    \caption{Shifted height functions}
    \end{subfigure}
    \caption{An example of shift invariance for the colored half space stochastic six-vertex model with two height functions. The height functions count the number of positive colors entering the rectangles from the left and exiting from the right. Theorem \ref{thm: 6vm shift simple} states that $\left(\Ht^{\geq 1}(4,4;\mathbf{z}),\Ht^{\geq 2}(5,3;\mathbf{z})\right)\deq\left(\Ht^{\geq 1}(4,4;\mathbf{z}'),\Ht^{\geq 3}(5,4;\mathbf{z}')\right)$.}
    \label{fig:6vm shift}
\end{figure}

\begin{theorem}
\label{thm: 6vm shift simple}
Let $i_k$, $i'_l$, $(x_k,y_k)$, and $(x_l',y_l')$ be two collections of color cutoffs and positions for height functions such that all height functions cross the diagonal, are ordered, $i_k\leq i_l'$ for all $k,l$, and $(x_k,y_k)$ is weakly northwest of $(x_l',y_l'+1)$ for all $k,l$. Then we have a joint distributional equality
\begin{equation*}
    \big\{\Ht^{\geq i_k}(x_k,y_k;\mathbf{z})\big\}_{k}\cup \big\{\Ht^{\geq i'_l}(x'_l,y'_l;\mathbf{z})\big\}_{l}\deq \big\{\Ht^{\geq i_k}(x_k;\mathbf{z}')\big\}_{k}\cup \big\{\Ht^{\geq i'_l+1}(x'_l,y'_l+1;\mathbf{z}')\big\}_{l},
\end{equation*}
where $\mathbf{z}'=(z_1',\dotsc)$ is obtained from $\mathbf{z}$ by setting $z_i'=z_i$ for $i< \min_l i_l'$ or $i>\max_l y_l'+1$, $z_{i+1}'=z_{i}$ if $\min_l i_l'\leq i\leq \max_l y_l'$, and $z_{\min_l i_l'}=z_{\max_l y_l'+1}$ (in other words, we move $z_{\max_l y_l'+1}$ below $z_{\min_l i_l'}$ and shift what is in between up).
\end{theorem}

\begin{remark}
For the six vertex model, we are able to prove a more general statement where the observables we consider do not have to all start $y$-axis, see Theorem \ref{thm: shift inv}. However, we still require that the height functions we study cross the diagonal in a suitable sense. This seems to be necessary (at least in the inhomogeneous case), see Remark \ref{rmk: bd necc}.
\end{remark}

\subsection{Shift invariance for the half space beta polymer}
\label{sec: intro beta}
The half space beta polymer, which is equivalent to a half space random walk in beta random environment, was recently introduced in \cite{BR22}. Here, we give a slightly different formulation which is more natural from the polymer point of view, but the connection to the random walk is not hard to see, and explained in Remark \ref{rmk: beta other form}.

\begin{figure}
    \centering
    \includegraphics[scale=0.7]{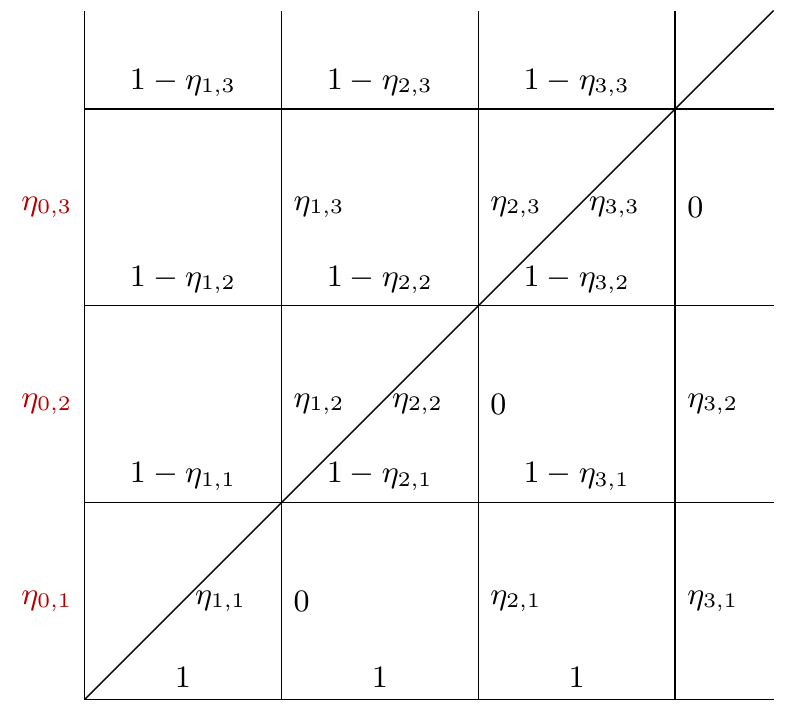}
    \caption{Half space beta polymer. Note that $\eta_{x,y}=1-\eta_{y,x}$ if $x\neq y$, so the random weights are symmetric except on the axes and at the boundary.}
    \label{fig:beta 2}
\end{figure}

We begin with positive parameters $\rho,\omega$, and $\sigma_x$ for $x\geq 1$. To each edge of $\N^2_0$, along with the additional edges of the form $(x,x)\to (x+1,x+1)$, we attach random edge weights in the following manner, see Figure \ref{fig:beta 2}.
\begin{itemize}
    \item For $x,y\geq 1$, let $\eta_{0,y}\sim\Beta(\rho,\sigma_y)$, let $\eta_{x,y}\sim \Beta(\sigma_x,\sigma_y)$ if $x<y$, $\eta_{x,x}\sim\Beta(\omega,\sigma_x)$, and $\eta_{y,x}=1-\eta_{x,y}$.
    \item To edges $(0,y-1)\to (0,y)$, we attach $\eta_{0,y}$. To edges $(x-1,0)\to (x,0)$, we attach $1$.
    \item To edges $(x-1,y)\to (x,y)$, we attach $\eta_{x,y}$, and to edges $(x,y-1)\to (x,y)$, we attach $1-\eta_{x,y}$ unless $x=y$, in which case we attach $0$. To edges $(x-1,x-1)\to(x,x)$, we attach $\eta_{x,x}$.
\end{itemize}
We let $Z^{\geq i}(x,y;\sigma)=\sum_{\gamma:(0,i)\to(x,y)} \wt(\gamma)$ denote the partition function starting at $(0,i)$ and ending at $(x,y)$, where the sum is over up right paths (with the edges $(x,x)\to(x+1,x+1)$ allowed to be taken), and $\wt(\gamma)$ is the product of all edge weights along $\gamma$.

By using a fusion procedure on the half space stochastic six vertex model, and then taking a continuous limit, which we identify with this beta polymer, we obtain the following shift invariance result. We use the same definitions as for height functions (see Section \ref{sec: 6vm intro}) for partition functions of crossing the diagonal and ordered.
\begin{theorem}
\label{thm: shift inv beta}
Let $i_k$, $i'_l$, $(x_k,y_k)$, and $(x_l',y_l')$ be two collections of color cutoffs and positions for partition functions such that all partition functions cross the diagonal, are ordered, $1\leq i_k\leq i_l'$ for all $k,l$, and $(x_k,y_k)$ is weakly northwest of $(x_l',y_l'+1)$ for all $k,l$. Then we have a joint distributional equality
\begin{equation*}
    \big\{Z^{\geq i_k}(x_k,y_k;\sigma)\big\}_{k}\cup \big\{Z^{\geq i'_l}(x'_l, y'_l;\sigma)\big\}_{l}\deq \big\{Z^{\geq i_k}(x_k,y_k;\sigma')\big\}_{k}\cup \big\{Z^{\geq i'_l+1}(x_l',y_l'+1;\sigma')\big\}_{l},
\end{equation*}
where $\sigma'=(\sigma_1',\dotsc)$ is obtained from $\sigma$ by setting $\sigma_i'=\sigma_i$ for $i< \min_l i_l'$ or $i>\max_l y_l'+1$, $\sigma_{i+1}'=\sigma_{i}$ if $\min_l i_l' \leq i\leq \max_l y_l'$, and $\sigma_{\min_l i_l'}=\sigma_{\max_l y_l'+1}$ (in other words, we move $\sigma_{\max_l y_l'+1}$ below $\sigma_{\min_l i_l'}$ and shift what is in between up).
\end{theorem}

\subsection{Shift invariance for the ASEP}
\label{sec: intro asep}

The \emph{colored half space asymmetric simple exclusion process} (ASEP) is a continuous time Markov process on the set of bijections from $\N\cup -\N$ to itself with the following description. At each position $x$, we place a particle of color $x$. Formally, configurations are given by functions $\eta$, with $\eta(x)$ being the color of the particle at position $x$, so we begin with the function $\eta(x)=x$.

The dynamics for the model (originally defined by Harris \cite{H77}) are given by assigning exponential clocks of rate $1$ and $q$ to each edge $(x,x+1)$ for $x\geq 1$, and exponential clocks of rate $\alpha$ and $\beta$ to the edge $(-1,1)$. When a clock of rate $1$ rings, we swap the particles at positions $x$ and $x+1$ if $\eta(x)<\eta(x+1)$, swapping particles at $-x$ and $-x-1$ as well, and when a clock of rate $q$ rings, we swap the particles at $x$ and $x+1$ if $\eta(x)>\eta(x+1)$ (and also $-x$ and $-x-1$). Similarly, if the clock of rate $\alpha$ rings, we swap $\eta(-1)$ and $\eta(1)$ if $\eta(-1)<\eta(1)$, and if the clock of rate $\beta$ rings, we swap those particles if $\eta(-1)>\eta(1)$. For any fixed time $t$, we can almost surely find a sequence $i_k\in\N$ where no clock associated to the edge $(i_k,i_k+1)$ rings up to time $t$, and given this, we can simply sample the process within the finite intervals $[i_k+1,i_{k+1}]$, so this process is well defined. When $q=\beta=0$, we call this the \emph{totally asymmetric simple exclusion process} (TASEP).

Note that it is equivalent to view the system as a coupled system of particle configurations, with particles of types $-1,0,1$, by projecting particles of color at least $i$ to $1$, at most $-i$ to $-1$, and the remaining as $0$. This is similar to what occurs in a full space, except that we are forced to introduce second class particles (the particles of color $0$) even in these projections (except when $i=1$). 

\begin{figure}
    \centering
\includegraphics[scale=0.5]{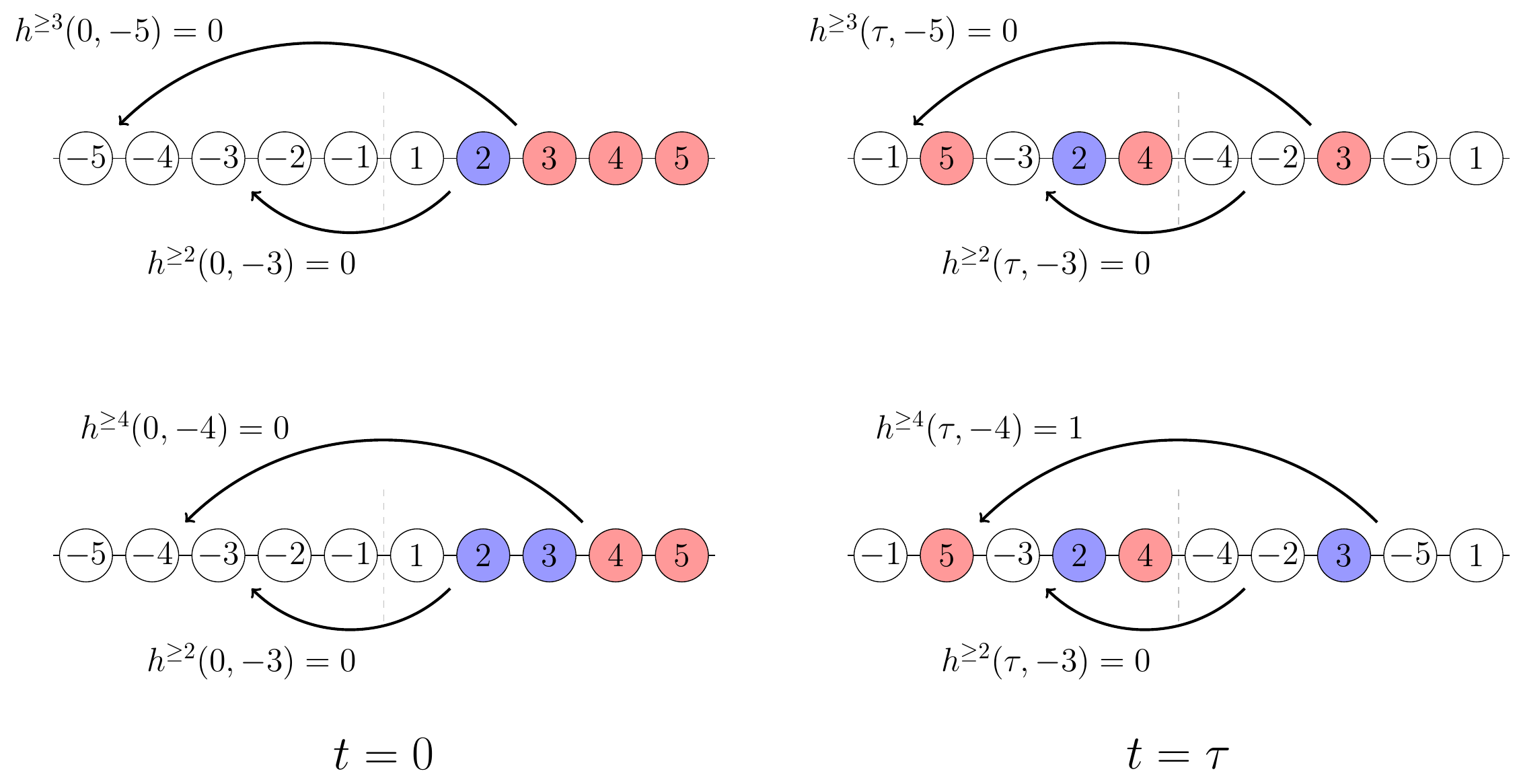}
    \caption{An example of shift invariance for the half space ASEP with two height functions. The top row shows, at times $0$ and $\tau$, the height functions $h^{\geq 3}(t,-5)$, which counts red particles at or to the left of $-5$, and $h^{\geq 2}(t,-3)$, which counts blur or red particles at or to the left of $-3$. In the bottom row, the height function $h^{\geq 3}(t,-5)$ is shifted to $h^{\geq 4}(t,-4)$. Shift invariance states that the joint distribution of the two height functions before and after the shift are equal, as long as the arrows cross $0$, and one starts at or before and ends at or after the other.}
    \label{fig:asep shift}
\end{figure}

We will now define the analogue of the height functions in this setting. Given a positive integer $i$, $x\in \N\cup -\N$, and $\tau\geq 0$, we let $h^{\geq i}(\tau,x)$ denote the number of particles of color at least $i$, at or to the left of position $x$ at time $\tau$. See Figure \ref{fig:asep shift} for some examples.

By taking a limit of the colored half space stochastic six vertex model, we obtain a version of shift invariance for the colored half space ASEP. See Figure \ref{fig:asep shift} for an example of shift invariance with two height functions.

\begin{theorem}
\label{thm: shift inv asep}
Fix a time $\tau\geq 0$ and collections of color cutoffs and positions $i_k$, $i_l'$, $x_k$, and $x_l'$, such that $1\leq i_k\leq  i'_l$, and $-1\geq  x_k>x'_l$. Then we have a joint distributional equality
\begin{equation*}
    \big\{h^{\geq i_k}(\tau, x_k)\big\}_{k}\cup\big\{h^{\geq i_l'}(\tau,x_l')\big\}_{l}\deq \big\{h^{\geq i_k}(\tau, x_k)\big\}_{k}\cup\big\{h^{\geq i_l'+1}(\tau,x_l'+1)\big\}_{l}.
\end{equation*}
\end{theorem}

\begin{remark}
Corollary \ref{cor: lpp shift} follows by taking $q=\beta=0$, and using the standard coupling between the TASEP and LPP. Unfortunately, a more general statement with more than one LPP time cannot be derived from Theorem \ref{thm: shift inv asep}, since it holds only for a single time $\tau$. 

Let us also remark that Theorem \ref{thm: shift inv asep} is non-trivial even for a single height function, whereas in the full space setting, the analogous statement follows from homogeneity of the random environment.
\end{remark}
\subsection{Asymptotics for the type \texorpdfstring{$B$}{B} oriented swap process}
\begin{figure}
    \centering
    \includegraphics[scale=0.7]{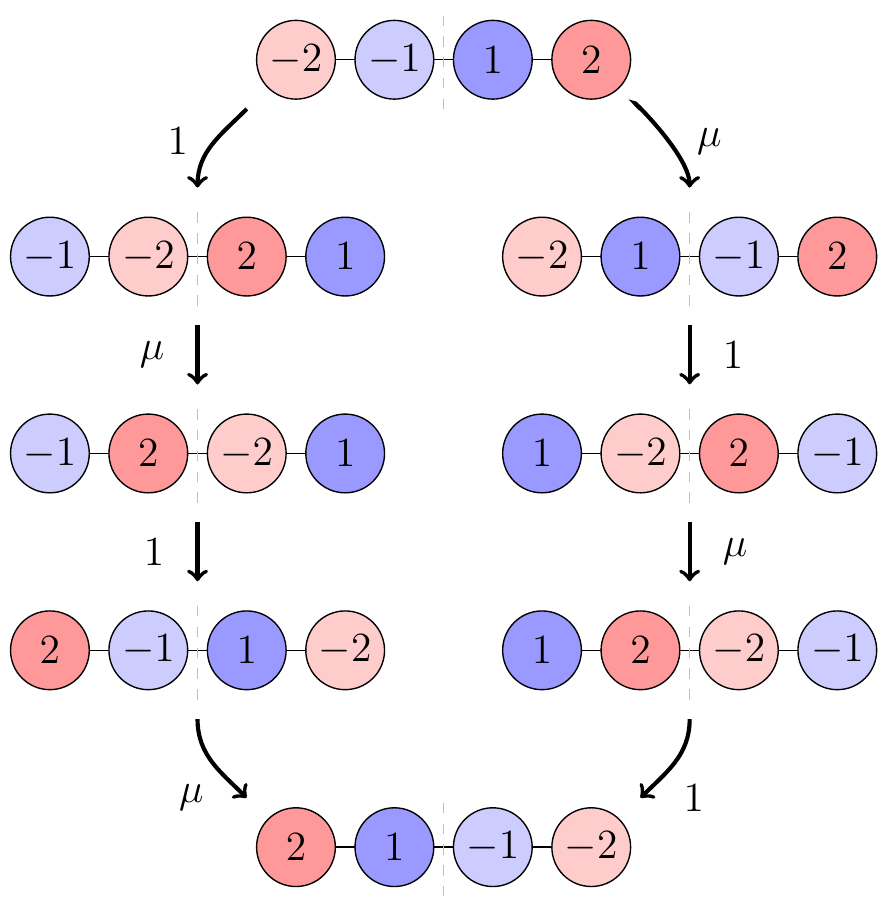}
    \caption{The $n=2$ type $B$ oriented swap process. The possible transitions along with their rates are shown.}
    \label{fig:osp}
\end{figure}

The \emph{type $B$ oriented swap process} is a continuous time absorbing Markov chain with the following description. We start with the numbers $-n,\dotsc, -1,1,\dotsc, n$ in order in a line. We put exponential clocks between positions $i,i+1$ and $-1,1$, of rates $1$ and $\mu$ respectively, and when a clock rings, we swap the numbers at those positions if they are increasing, along with the numbers at $-i,-i-1$ for the rate $1$ clocks. This is another name for the fully colored half space TASEP (restricted to a finite interval). 

Since the initial condition and dynamics are symmetric, the state is always symmetric in the sense that the numbers at positions $i$ and $-i$ always have the same absolute value. Clearly, this chain is absorbed when the order is completely reversed, and we let $T$ denote this time. We prove the following distributional identity between $T$ and a point-to-line half space LPP time. This is the type $B$ analogue of a result of Bufetov, Gorin and Romik on the (type $A$) oriented swap process \cite{BGR20}. Their result answered a question of Angel, Holroyd, and Romik \cite{AHR09}.
\begin{theorem}
Let $T$ denote the absorption time in the type $B$ oriented swap process, and let $L$ denote the point-to-point half space LPP times\footnote{See Corollary \ref{cor: lpp shift} for the definition.}. Then we have the equality of distributions
\begin{equation*}
    T\deq \max_{x+y=2n-2} L((0,0),(x,y)).
\end{equation*}
\end{theorem}

Since the limiting distribution for these LPP times are known (through an algebraic identity, it is equivalent to a full space point to point LPP time whose asymptotics are known \cite{BC11,BBP05,O08}), this immediately implies the following asymptotics for $T$. For simplicity, we keep $\mu$ fixed, but if $\mu$ varies with $n$, then one can also obtain the BBP transition near the critical point.
\begin{corollary}
Let $F_{GUE}$ and $F_{GOE}$ denote the Tracy--Widom distributions for the GUE and GOE matrix ensembles respectively. Let $\Phi$ denote the distribution function for a standard normal random variable. If $\mu>\frac{1}{2}$,
\begin{equation*}
    \PP\left(\frac{T-4n}{2^{\frac{4}{3}}n^{\frac{1}{3}}}\leq x\right)\to F_{GUE}(x),
\end{equation*}
if $\mu=\frac{1}{2}$,
\begin{equation*}
    \PP\left(\frac{T-4n}{2^{\frac{4}{3}}n^{\frac{1}{3}}}\leq x\right)\to F_{GOE}^2(x),
\end{equation*}
and if $\mu<\frac{1}{2}$,
\begin{equation*}
    \PP\left(\frac{T-(\mu^{-1}+(1-\mu)^{-1})n}{(\mu^{-2}+(1-\mu)^{-2})^{\frac{1}{2}}n^{\frac{1}{2}}}\leq x\right)\to \Phi(x).
\end{equation*}
\end{corollary}
We actually prove a generalization of these results to configurations where the number $0$ is allowed to appear as well, see Theorem \ref{thm: osp lpp same} and Corollary \ref{cor: osp asymp}. The proof uses shift invariance for the half space TASEP in a crucial way.

\begin{remark}
Let us note that the oriented swap process can be defined for any Coxeter group, and for finite Coxeter groups, it is natural to ask for the asymptotics for the absorbing time. For infinite Coxeter groups, this does not make sense but the random walk is still interesting, and in particular for affine Weyl groups, this was studied by Lam \cite{L15} (who viewed the walk equivalently as a non-backtracking walk on the alcoves of the corresponding affine hyperplane arrangement).
\end{remark}

\subsection{Related works and further directions}
In this paper, we introduce the notion of shift invariance for half space models, and give a few applications. We have not attempted to prove the most general or powerful versions of our results, and so we list a few directions where our work could be extended. We hope to address some of these in future work.

Galashin characterized the symmetries of the six vertex model in the full space setting, establishing a conjecture of Borodin, Gorin and Wheeler \cite{G21}. In particular, he showed that as long as a transformation preserves the marginal distributions of height functions, it preserves the joint distribution. It would be interesting to know if a similar result could be proven in the half space case.

In this work, we have completely focused our attention on height functions which cross the diagonal. One can ask whether a version of shift invariance holds for height functions that are close to but do not cross the diagonal. While it is easy to see that vertical or horizontal shifts cannot even preserve the distribution for a single height function (at least in the inhomogeneous model), a diagonal shift parallel to the boundary clearly does. It is unclear if an analogue of shift invariance would hold for these height functions. More generally, one could study an arbitrary collection of height functions, and ask what forms of shift invariance would hold.

While our work covers a broad spectrum of integrable models, it certainly does not cover all of them. A notable omission is the lack of integrable polymer or LPP models. Dauvergne was able to establish shift invariance (and more general symmetries) in integrable polymer and LPP models using the RSK correspondence \cite{D22}. It would be very interesting to know if these techniques could be used in the half space setting. 

Another problem left open is if a version of shift invariance holds for the KPZ equation in a half space, or even how to formulate it. While work of Parekh \cite{P19} showed that the ASEP in a half space with step initial condition converged to the KPZ equation with a Neumann boundary condition and narrow wedge initial data, this only covers the case when there are no $0$ particles present (treating $1$ as a hole and $-1$ as a particle). It would be interesting to understand if a KPZ equation limit could be obtained from the ASEP with macroscopically many $0$ particles. 

Another way to obtain a KPZ limit would be through polymers. Recent work of Barraquand and Corwin \cite{BC22} shows that the point to point partition functions of half space polymers should converge to a solution to the KPZ equation. In addition, the beta polymer that we consider does not seem to be general enough to obtain the half space KPZ limit, see \cite{BR22}. One would need to add an asymmetry parameter into the model first, and so it would be interesting to know whether shift invariance still holds for this model. The work of Korotkikh \cite{K22} shows shift invariance for the full space beta polymer with an extra family of parameters, and so one could hope that something similar can be done in the half space setting.

Finally, while our result on the absorption time of the type $B$ oriented swap process gives the limiting distribution, Zhang was able to prove an even more powerful distributional identity between the TASEP and LPP \cite{Z22}, which was derived through a version of shift invariance for the TASEP which allowed for multiple times. We conjecture that the analogous statement is true in the type $B$ setting (see Conjecture \ref{conj: osp}), and it seems plausible that the shift invariance for the half space TASEP continues to hold with multiple times.

\subsection{Organization}
In Section \ref{sec: prelim}, we review some preliminaries on Hecke algebras and vertex models in the type $B$ or half space setting, and show their connection. In Section \ref{sec: flip}, we establish a version of flip invariance (Theorem \ref{thm: flip}), and use it to prove shift invariance (Theorem \ref{thm: shift inv}). 

Sections \ref{sec: asep} and \ref{sec: osp} are completely independent of Sections \ref{sec: fusion}, \ref{sec: higher spin}, and \ref{sec: beta polymer}, and each collection may be read without reading the other.

In Section \ref{sec: asep}, we show that shift invariance holds for the ASEP on a half-line by taking a limit of the half space six-vertex model. In Section \ref{sec: osp}, we use shift invariance for the TASEP to study the distribution of the absorbing time in a type $B$ analogue of the oriented swap process, proving a type $B$ version of a result of Bufetov, Gorin, and Romik \cite{BGR20}.

In Section \ref{sec: fusion}, we explain how to fuse rows and columns of the half space six-vertex model to obtain a fused vertex model where multiple arrows may occupy a single edge, leaving the formal derivation for Appendix \ref{app: pf}. In Section \ref{sec: higher spin}, we degenerate and analytically continue the fused vertex model before taking a limit to obtain a continuous vertex model. Finally, in Section \ref{sec: beta polymer}, we show that the continuous vertex model generalizes the random walk in beta random environment on a half-line studied by Barraquand and Rychnovsky \cite{BR22}. In each section, we show that the models obtained satisfy shift invariance.

\subsection{Notation}
We use $\N$ to denote the positive integers, and $\N_0$ to denote the positive integers with $0$. If $a<b$ are integers, we let $[a,b]=\{a,\dotsc, b\}$. We will use $\pm [a,b]$ to denote the set $[-b,-a]\cup [a,b]$ (and not $[a,b]$ or $-[a,b]$). We will use $\pm \N$ to denote the set $\N\cup -\N$ (and not $\N$ or $-\N$). If $x\in\N$, we let $x^{(1)}$ and $x^{(2)}$ denote its first and second coordinates.

We will use $(a;q)_n=\prod_{i=0}^{n-1}(1-aq^i)$ to denote the $q$-Pochhammer symbol, and
\begin{equation*}
    {b\choose a}_q=\frac{(q;q)_b}{(q;q)_a(q;q)_{b-a}}
\end{equation*}
to denote the $q$-binomial coefficient.

We use $\deq$ to denote equality in distribution between two random variables.

\section{Hecke algebra and vertex models}
\label{sec: prelim}
In this section, we define the colored stochastic six-vertex model in the half space setting. We also review some needed background on type $B$ Hecke algebras and explain their connection to vertex models.

\subsection{Colored half space stochastic six-vertex model}
\label{sec: 6vm}
\subsubsection{Model definition}
We start with parameters $z_i$ for $i\in\N$, $q$, $t$ and $\nu$. We will sometimes use the notation $q_i=q$ if $i>0$ and $q_0=t$ to unify certain expressions. We call the $z_i$ \emph{rapidities}. Initially, we will assume that they are chosen so all weights defined below are probabilities, but later on we will view the weights as rational functions in these parameters, which will then be treated as formal variables. We let
\begin{equation}
\label{eq: weights}
    \mathbf{p}_{x,y}=\mathbf{p}_{|y-x|}(z_xz_y)=\begin{cases}\frac{z_xz_y-1}{z_xz_y-q}&\text{ if }x\neq y,
    \\\frac{z_x^2-1}{(z_x-\nu t)\left(z_x+\frac{1}{\nu}\right)}&\text{ if }x=y.
    \end{cases}
\end{equation}
We will refer to the argument of $\mathbf{p}_{|y-x|}$, the $z_xz_y$, as the \emph{spectral parameter}.

We consider certain configurations of colored arrows on the edges of $\N^2$. Each edge will contain exactly one arrow, and the boundary will contain an arrow of color $i$ entering vertex $(1,i)$ on the left, and an arrow of color $-i$ entering vertex $(i,1)$ on the bottom. We will refer to vertices $(x,y)$ with $x\neq y$ as \emph{bulk} vertices and vertices $(x,x)$ as \emph{boundary} vertices. Vertices $(1,y)$ and $(x,1)$ will be referred to as \emph{incoming} vertices.

\begin{figure}
    \centering
    \begin{tabular}{l|cc|cc}
    Configuration:&$\vcenter{\hbox{\includegraphics[scale=1]{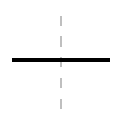}}}$ &$\vcenter{\hbox{\includegraphics[scale=1]{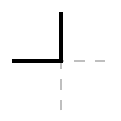}}}$&$\vcenter{\hbox{\includegraphics[scale=1]{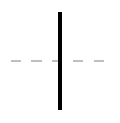}}}$&$\vcenter{\hbox{\includegraphics[scale=1]{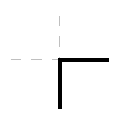}}}$ \\
    Probability:&$\mathbf{p}_{x,y}$&$1-\mathbf{p}_{x,y}$&$q_{y-x}\mathbf{p}_{x,y}$&$1-q_{y-x}\mathbf{p}_{x,y}$ 
    \end{tabular}
    \caption{Probabilities for sampling outgoing colors at a vertex $(x,y)$, with incoming colors $i<j$. The black line represents the larger color $j$ and the dashed line represents the smaller color $i$.}
    \label{fig:vtx wts}
\end{figure}

We now describe a Markovian sampling procedure to generate configurations. At a vertex $(x,y)$ with $x\leq y$ where both the left and bottom incoming edges already have arrows assigned, we sample outgoing arrows at the top and right according to the probabilities given in Figure \ref{fig:vtx wts}. If the vertex is not a boundary vertex (i.e. $x< y$), we then also symmetrically set the outcome in vertex $(y,x)$ by requiring that if an edge is occupied by the color $i$, then its reflection about $x=y$ is occupied by the color $-i$. Repeating this process inductively, we define a random configuration on $\N^2$. Note that as long as we choose the parameters correctly (say $q,t<1$, $z_x>1$, and $\nu\in (0,1)$), we obtain a genuine probability distribution on configurations.

We will let $\Ht^{\geq i}(x,y)$ denote the number of arrows of color at least $i$ leaving a vertex at or below $(x,y)$ to the right. Formally, we can define $\Ht^{\geq i}(x,y)$ using the recurrence $\Ht^{\geq i}(x,y)=\Ht^{\geq i}(x-1,y)-c^{\geq i}(x,y)$, where $c^{\geq i}(x,y)=1$ if a color of at least $i$ leaves vertex $(x,y)$ from the top, and $0$ otherwise. The boundary values are given by $\Ht^{\geq i}(0,y)=y-i+1$ if $i\leq y$, and $0$ otherwise.

We will eventually also wish to consider the model where there are multiple arrows of a given color. We will always consider only ordered incoming arrows, which means that arrows cannot decrease in color as we move from right to left or bottom to top along the incoming vertices. We use the same sampling procedure as above to generate configurations, with the caveat that if two arrows of the same color enter a vertex, they deterministically leave the top and right.

\subsubsection{Finite domains}
We will also wish to consider the vertex model restricted to a finite domain. Let $A\subseteq \Z^2$ be a subset symmetric with respect to reflection across the diagonal $x=y$, which is bounded by two up-left paths $(P,Q)$ (also symmetric under reflection across $x=y$), of length $2n$. We call such a set a \emph{symmetric skew domain}. Here, $P$ is to the bottom-left and $Q$ is to the top-right of $A$. We will number the edges crossing $P$ and $Q$ with $\pm [1,n]$, in increasing order moving up-left along the paths. We let $A^+\subseteq A$ denote the subset lying on or above the diagonal $x=y$.

\begin{remark}
We will generally conflate the domain $A$ with the up-left paths $P$ and $Q$. Strictly speaking, there is some ambiguity, because we wish to allow the degenerate case where $P$ and $Q$ coincide for some portion. Thus, it should be understood that whenever we refer to a symmetric skew domain $A$, we are implicitly assuming that we have chosen some up-left paths $(P,Q)$, although most of the time one can simply take the simplest choice where $P$ and $Q$ share no edges in common.
\end{remark}

We consider incoming colors $\pm [1,n]$ assigned to the edges crossing $P$. These will be given as an \emph{incoming color permutation} $\sigma\in B_n$, with $\sigma(i)$ being the color entering through edge $i$ of $P$. If no assignment is made, it is assumed that the colors are assigned in increasing order as we move up-left along $P$, which means $\sigma=\id$. 

A \emph{configuration} is an assignment of colored arrows to the edges in $A$ compatible with the incoming colors $\sigma$ (and satisfying all other constraints in the infinite model). The \emph{outgoing color permutation} $\pi\in B_n$ is defined so that $\pi(i)$ is the color crossing edge $i$ in $Q$. We let $P^{\sigma,\pi}_A(\mathbf{z})$ (or sometimes $P^{\sigma,\pi}_A$ if the dependence on $\mathbf{z}$ is not relevant) denote the probability of seeing an outgoing color permutation $\pi$ on the domain $A$ with incoming colors $\sigma$ (using the same Markovian sampling procedure as before). Explicitly, this is a sum over all configurations of the weights of each configuration. If $\sigma$ is omitted, we assume that $\sigma=\id$. 

\begin{remark}
We note that our conventions differ from \cite{G21}, whose color permutations are the inverse of ours. However, we also write our permutation composition as $(\pi\sigma)(x)=\pi(\sigma(x))$, and so all formulas essentially remain the same.
\end{remark}

\begin{remark}
Note that due to the symmetry, the model keeps track of more information than is required. In particular, we could either restrict to vertices $(x,y)$ for $x\leq y$, or only keep track of the positive colors, although we cannot do both simultaneously. However, we choose to keep all the data, because it can sometimes clarify certain things which may otherwise be unclear. For example, if we require that the color $1$ travels through a particular edge, we also simultaneously require that $-1$ travels through the reflected edge. By forgetting the negative colors, it is less clear that the reflected edge is not available for other colors to use. 
\end{remark}

We now define the notion of height functions in $A$. Let $A$ lie between $P$ and $Q$. A \emph{$(P,Q)$-cut} is a quadruple $(d,l,u,r)$ such that $l\leq r$, $d\leq u$, $(l,d),(r,u)\in A$, and $(l-1,d-1),(r+1,u+1)\not\in A$. Note that $C$ defines a rectangle with corners $(l,d)$ and $(r,u)$, and these points must lie on $P$ and $Q$ respectively. We will let $\Ht_A(C)$ denote the number of colors entering $C$ from the left and exiting $C$ from the right. Formally, if $\pi$ denotes the color permutation of a configuration, there are $i$ edges between $(l,d)$ and the top-left corner of $A$ on $P$, counting a left edge at $(l,d)$ but not a bottom edge, and $j$ edges between $(r,u)$ and the bottom-right corner of $A$ on $Q$, counting a right edge at $(r,u)$ but not a top edge, we let $\Ht_A(C)=\{k\leq j\mid \pi(k)\geq i\}$ (since $C$ is a $(P,Q)$-cut, this is equivalent to the number of colors crossing $C$). For a family $\mathcal{C}=\{C_j\}_{1\leq j\leq m}$ of $(P,Q)$-cuts, we let $\Ht_A(\mathcal{C})=(\Ht_A(C_1),\dotsc, \Ht_A(C_m))$ denote the vector of height functions for that family. If the dependence on $\mathbf{z}$ is relevant, we will write $\Ht_A(C;\mathbf{z})$ and $\Ht_A(\mathcal{C};\mathbf{z})$.

\subsubsection{Color projection}
An important feature of the six-vertex model is that adjacent colors may be merged to obtain a new six-vertex model. This is well-known in the full space setting, see e.g. Proposition 4.3 of \cite{BGW22}, and remains true in the half space setting. Of course since our colors are assumed to be symmetric, any projection of the colors should also be done in a symmetric way. For example, if the colors $1$ and $2$ are merged, then $-1$ and $-2$ should also be merged.

This property also means that more general ordered incoming arrows where multiple arrows of the same color are allowed can actually be viewed as a special case of the fully colored model. Additionally, we may also allow a color $0$, which can be viewed as an absence of an arrow, by merging $1$ and $-1$. The color $0$ is special because $0=-0$, which means that at the boundary, it will always meet itself.

\subsubsection{Integrability}
The reason for the choice of weights is that these weights are integrable. In the half space setting, this means that they satisfy the Yang--Baxter equation in the bulk and the reflection equation at the boundary. Integrability for half space models goes back to work of Cherednik \cite{C84} and Sklyanin \cite{S88}, and this integrability has been used in related half space six-vertex models, see e.g. \cite{BBCW18, K02, BW16, BWZ15, WZ16}.

In the following two propositions, we use a graphical notation where outside edges are assumed to have some fixed colors going in or out, and internal edges are summed over. The quantities near a vertex indicate the spectral parameter for that vertex. The equalities are equalities of the probability (or partition function) for the given outside configuration, and hold for any choice of what occurs with the outside edges. Both equations can be proven by checking a finite number of cases, and the Yang--Baxter equation and the uncolored versions of the reflection equation have been well understood, see e.g. \cite{ML19, VR93}.
\begin{proposition}[Yang--Baxter equation]
\label{prop:YB}
We have
\begin{equation*}
    \vcenter{\hbox{\includegraphics[scale=1]{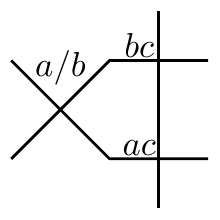}}}=\vcenter{\hbox{\includegraphics[scale=1]{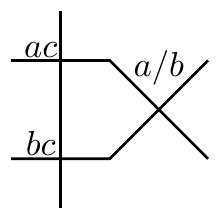}}}
\end{equation*}
where the notation means an equality of partition functions after fixing the colors of external edges and summing over all internal configurations, with spectral parameters indicated.
\end{proposition}

\begin{proposition}[Reflection equation]
\label{prop:refl}
We have
\begin{equation*}
    \vcenter{\hbox{\includegraphics[scale=1]{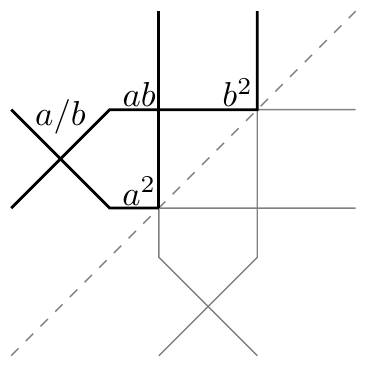}}}=\vcenter{\hbox{\includegraphics[scale=1]{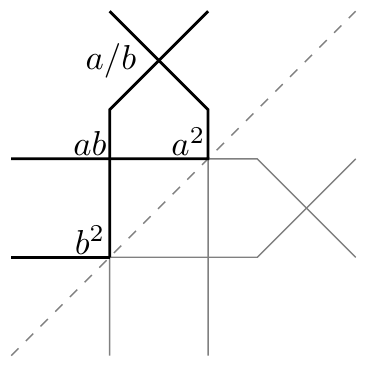}}}
\end{equation*}
where the notation means an equality of partition functions after fixing the colors of external edges and summing over all internal configurations, with spectral parameters indicated. Note that gray lines are determined by symmetry, and do not contribute to the weight of a configuration.
\end{proposition}

\subsection{Type \texorpdfstring{$B$}{B} Hecke algebras}
In this section, we will briefly review the necessary background and properties of the type $B$ Coxeter group and Hecke algebra. We refer the reader to \cite{H90} for more background on Coxeter groups and Hecke algebras in general.

We let $B_n$ denote the hyperoctahedral group, which is the group of signed permutations on $n$ letters (which we will just call permutations). That is, $B_n$ is the set of permutations $\pi$ of $\{-n,\dotsc,-1,1,\dotsc, n\}$ such that $\pi(-i)=-\pi(i)$. We will view this as a Coxeter group, with generators $s_0=(-1,1)$ and $s_k=(k,k+1)$ for $0<k<n$. We define the \emph{length} $l(\pi)$ to be the minimum number of generators needed to write $\pi$. 

Let $q,t,\nu$, and $z_i$ be formal variables, and for convenience we let $q_i=q$ if $i>0$, and $q_0=t$. We define the \emph{type $B$ Hecke algebra} to be the associative algebra $\mathcal{H}=\mathcal{H}_{q,t,\nu}(B_n,\mathbf{z})$ over $\mathbf{C}(q,t,\nu,\mathbf{z})$ defined by the basis $T_w$ for $w\in B_n$ with relations
\begin{alignat*}{2}
    T_uT_w&=T_{uw}&&\qquad \text{ if }l(uw)=l(u)+l(w),
    \\(T_i+q_i)(T_i-1)&=0,&&
\end{alignat*}
where $T_i=T_{s_i}$, with $s_i$ the Coxeter generators for $B_n$ defined above. We have $T_{id}=1$.

Given $[a,b]=\{a,a+1,\dotsc, b\}$, we let $B_{[a,b]}$ be the \emph{parabolic subgroup} of $B_n$ generated by $s_i$ for $i\in [a,b-1]$, and let $\mathcal{H}_{[a,b]}$ be the \emph{parabolic subalgebra} generated by $T_{s_i}$ for $i\in [a,b-1]$. Note that $B_{[0,b]}\cong B_b$ and $B_{[a,b]}\cong S_{b-a}$ if $a>0$. Note that $B_{[a,b]}$ always acts on $[a,b]$ (and by symmetry $-[a,b]$) if $a>0$, and on $\pm [1,b]$ if $a=0$. We will also write $B_{-[a,b]}=B_{[a,b]}$ and $\mathcal{H}_{-[a,b]}=\mathcal{H}_{[a,b]}$ for convenience.

Given $p\in \mathbf{C}(q,\mathbf{z})$, we let
\begin{equation*}
    R_k(p)=pT_k+(1-p).
\end{equation*}
The above rules imply that
\begin{align}
    \label{eq: R_k right} T_\pi R_k(p)&=\begin{cases}pT_{\pi s_k}+(1-p)T_{\pi}&\text{ if }l(\pi s_k)>l(\pi),
    \\pq_kT_{\pi s_k}+(1-pq_k)T_{\pi}&\text{ if }l(\pi s_k)<l(\pi).
    \end{cases}
    \\\label{eq: R_k left} R_k(p)T_\pi&=\begin{cases}pT_{s_k\pi}+(1-p)T_{\pi}&\text{ if }l(s_k \pi)>l(\pi),
    \\pq_kT_{s_k\pi}+(1-pq_k)T_{\pi}&\text{ if }l(s_k\pi)<l(\pi).
    \end{cases}
\end{align}

The elements $R_k(p)$ also satisfy a version of the Yang--Baxter and reflection equations. In fact, the following equations are equivalent to the ones for the six-vertex model by Proposition \ref{prop: hecke 6vm same}. We write
\begin{equation*}
    \mathbf{p}_i(z)=\frac{z-1}{z-q}
\end{equation*}
if $i>0$, and
\begin{equation*}
    \mathbf{p}_0(z)=\frac{z-1}{(z-\nu t)(z+\frac{1}{\nu})}.
\end{equation*}
The Yang--Baxter equation is
\begin{equation*}
\tag{Yang--Baxter equation}
\begin{split}
    &R_{k+1}(\mathbf{p}_{k+1}(a/b))R_{k}(\mathbf{p}_{k}(ac))R_{k+1}(\mathbf{p}_{k+1}(bc))
    \\=&R_{k}(\mathbf{p}_{k}(bc))R_{k+1}(\mathbf{p}_{k+1}(ac))R_{k}(\mathbf{p}_{k}(a/b)),
\end{split}
\end{equation*}
valid for $k>0$, and the reflection equation is
\begin{equation*}
\tag{Reflection equation}
\begin{split}
    &R_1(\mathbf{p}_1(a/b))R_{0}(\mathbf{p}_{0}(a^2))R_1(\mathbf{p}_1(ab))R_0(\mathbf{p}_0(b^2))
    \\=&R_{0}(\mathbf{p}_{0}(b^2))R_{1}(\mathbf{p}_{1}(ab))R_{0}(\mathbf{p}_{0}(a^2))R_1(\mathbf{p}_1(a/b)).
\end{split}
\end{equation*}

Given a symmetric skew domain $A$, we let
\begin{equation*}
    Y_{A}=\prod _{(x,y)\in A^+}R_{y-x}(\mathbf{p}_{x,y}),
\end{equation*}
where the product is taken in the up-right order, so the factor indexed by $(x,y)$ occurs before the factor indexed by $(x',y')$ if $x\leq x'$ and $y\leq y'$, and $(x,y)\neq (x',y')$. This is well-defined due to the Yang--Baxter and reflection equations.

\begin{remark}
The $Y_A$ that we have defined are closely related to the \emph{Yang--Baxter basis} for the type $B$ Hecke algebra. In particular, if $q=t$ and $\nu=1$ or $\nu=t^{-1/2}$, then the $Y_A$ (for certain $A$) exactly coincide with certain elements of the type $B$ or type $C$ Yang--Baxter basis, see \cite{NN15} for example. Thus, the $Y_A$ can be viewed as a deformation.
\end{remark}

\subsection{A Hecke algebra interpretation of the six-vertex model}
We now establish a connection between the six-vertex model and the Hecke algebra.
\begin{proposition}
\label{prop: hecke 6vm same}
Let $A$ denote a symmetric skew domain, and fix an incoming color permutation $\sigma\in B_n$. Then
\begin{equation*}
    T_\sigma Y_A=\sum_{\pi\in B_n} P^{\sigma,\pi}_A T_\pi.
\end{equation*}
\end{proposition}
\begin{proof}
We induct on $|A^+|$. The base case $|A^+|=0$ is trivial. Now for a symmetric skew domain $A$, there is at least one vertex $(x,y)$ with no vertices in $A$ above or to the right of it. This vertex must lie on $Q$, so let $i$ and $i+1$ denote the outgoing edges on $Q$ to the right and the top. Such a vertex (and its reflection) can be removed, giving a new symmetric skew domain $A'$, with $|(A')^+|=|A^+|-1$. We then note that
\begin{equation*}
    P^{\sigma,\pi}_{A}=q_{y-x}\mathbf{p}_{x,y}P_{\sigma,\pi s_{y-x}}^{A'}+(1-\mathbf{p}_{x,y})P^{\sigma,\pi}_{A'}
\end{equation*}
if $l(\pi s_{y-x})>l(\pi)$, and
\begin{equation*}
    P^{\sigma,\pi}_{A}=\mathbf{p}_{x,y}P_{\sigma,\pi s_{y-x}}^{A}+(1-q_{y-x}\mathbf{p}_{x,y})P^{\sigma,\pi}_{A}
\end{equation*}
if $l(\pi s_{y-x})<l(\pi)$. But
\begin{equation*}
\begin{split}
    T_\sigma Y_A&=T_\sigma Y_{A'} R_{y-x}(\mathbf{p}_{x,y})
    \\&=\sum_{\pi\in B_n}P^{\sigma,\pi}_{A'} T_\pi R_{y-x}(\mathbf{p}_{x,y})
\end{split}
\end{equation*}
and we can compute this using \eqref{eq: R_k right}. If $l(\pi s_{y-x})>l(\pi)$, then the coefficient of $T_\pi$ will be
\begin{equation*}
    q_{y-x}\mathbf{p}_{x,y}P^{\sigma,\pi}_{A'}+(1-\mathbf{p}_{x,y})P_{\sigma,\pi s_{y-x}}^{A'}=P^{\sigma,\pi}_A,
\end{equation*}
and if $l(\pi s_{y-x})<l(\pi)$, then the coefficient will be
\begin{equation*}
    \mathbf{p}_{x,y}P^{\sigma,\pi}_{A'}+(1-q_{y-x}\mathbf{p}_{x,y})P_{\sigma,\pi s_{y-x}}^{A'}=P^{\sigma,\pi}_A,
\end{equation*}
which establishes the claim.
\end{proof}

Because of this result, we will let $P^\pi_{Y}$ denote the coefficient of $T_\pi$ in $Y$, and we have $P^\pi_{T_\sigma Y_A}=P^{\sigma,\pi}_A$.

\subsection{Boundary conditions}
We now turn to the types of boundary conditions we will consider. Unlike in the full space setting, it does not make sense to consider horizontal or vertical conditions, since for example any restriction on the horizontal positive arrows forces a restriction on the vertical negative arrows. In addition, we will require additional constraints for the argument to work.

We will thus call our boundary conditions central and peripheral. We fix $M,N,K\in\N$, and let $n=M+N-K$. We will let $P_{in}=[N-K+1,n]$, $P_{out}=-[M-K+1,n]$, $C_{in}=\pm[1,N-K]$, and $C_{out}=\pm [1,M-K]$. See Section \ref{sec: symm rect} for the motivation and meaning behind these choices, but for now we just view them as parameters.

\begin{definition}[Boundary conditions]
A \emph{central boundary condition} is given by a set $\mathfrak{C}=\{(i_1,o_1),\dotsc, (i_k,o_k)\}$ where the $i_j$ and $o_j$ are distinct, and $i_j\in C_{in}$ and $o_j\in C_{out}$, along with two functions $c_{in}:C_{in}\to\{P_{out},-P_{out},C_{out}\}$ and $c_{out}:C_{out}\to\{P_{in},-P_{in},C_{in}\}$ satisfying $c_{in}(-i)=-c_{in}(i)$ and $c_{out}(-o)=-c_{out}(o)$ (where $-C_{in}=C_{in}$ and $-C_{out}=C_{out}$). A \emph{peripheral boundary condition} is given by a set $\mathfrak{P}=\{(i_1,o_1),\dotsc, (i_k,o_k)\}$ where the $i_j$ and $o_j$ are distinct, and $i_j\in P_{in}$ and $o_j\in P_{out}$. We let $i(\mathfrak{C})$ and $o(\mathfrak{C})$ denote the projections of $\mathfrak{C}$ to the first and second coordinates, and similarly for $\mathfrak{P}$. We will use the notation $\mathfrak{BC}=(\mathfrak{C},c_{in},c_{out},\mathfrak{P})$, and call this a \emph{boundary condition}. 

A permutation $\pi\in B_n$ \emph{satisfies} the boundary conditions $\mathfrak{BC}=(\mathfrak{C},c_{in},c_{out},\mathfrak{P})$, if $\pi(o)=i$ for all $(i,o)\in \mathfrak{C},\mathfrak{P}$, $\pi^{-1}(i)\not\in C_{out}$ if $i\not\in i(\mathfrak{C})$, $\pi^{-1}(i)\not\in P_{out}$ if $i\not\in i(\mathfrak{P})$, $\pi^{-1}(i)\in c_{in}(i)$ for all $i\in C_{in}$, and $\pi(o)\in c_{out}(o)$ for all $o\in C_{out}$. For boundary conditions $\mathfrak{BC}$, define $\sat(\mathfrak{BC})$ to be the set of all $\pi\in B_n$ satisfying those boundary conditions.
 
We say that a boundary condition $\mathfrak{BC}$ is consistent if, $(i,o)\in \mathfrak{C}$ if and only if $(-i,-o)\in\mathfrak{C}$, $c_{in}(i)=C_{out}$ for all $i\in i(\mathfrak{C})$, and $c_{out}(o)=C_{in}$ for all $o\in o(\mathfrak{C})$. We will always assume that the boundary conditions are chosen so that $\sat(\mathfrak{BC})$ is non-empty, which means that they must be consistent.
\end{definition}

\begin{remark}
Let us explain informally the information that the boundary condition tracks. Viewing the signed permutation visually as connecting the numbers $\pm[1,n]$ to themselves, $\mathfrak{C}$ gives the connections between $C_{in}$ and $C_{out}$, $c_{in}$ gives the destination (between $P_{out},-P_{out},C_{out}$) of the numbers in $C_{in}$, $c_{out}$ gives the origin of the numbers in $C_{out}$, and $\mathfrak{P}$ gives the connections between $P_{in}$ and $P_{out}$ (and by symmetry $-P_{in}$ and $-P_{out}$). The conditions for consistency are exactly the ones which ensure that the conditions are symmetric (respecting the fact that $\pi(-i)=-\pi(i)$), and do not contradict each other.

The parameters in the boundary conditions will eventually relate to certain finite domains for the six-vertex model. For now, we view the definition purely formally, but see Figure \ref{fig:sym rect} for a visual guide of what the parameters represent.
\end{remark}

We note the following fact that the total number of connections between any two of the intervals is determined by $\mathfrak{BC}$.
\begin{lemma}
Consider boundary condition $\mathfrak{BC}$. For all $\pi\in\sat(\mathfrak{BC})$, the counts $|\{o\in O\mid \pi(o)\in I\}|$ for $I=P_{in},-P_{in},C_{in}$ and $O=P_{out},-P_{out},C_{out}$ are constant, i.e. the number of connections from $I$ to $O$ is a function of $\mathfrak{BC}$.
\end{lemma}
\begin{proof}
This is clear if $(I,O)=(P_{in},P_{out})$, $I=C_{in}$, or $O=C_{out}$, since $\mathfrak{P}$, $c_{in}$ and $c_{out}$ immediately give the number of connections. The only remaining case is $(I,O)=(P_{in},-P_{out})$ (and $(-P_{in},P_{out})$ by symmetry), but this can be derived by the number of connections between $P_{in}$ and either $C_{out}$ or $P_{out}$, since the total number must be $|P_{in}|$.
\end{proof}

If $Y\in \mathcal{H}$, we let $P_Y^{\mathfrak{BC}}=\sum _{\pi\in\sat(\mathfrak{BC})}P_Y^{\pi}$, and for a symmetric skew domain $A$, we let $P_A^{\mathfrak{BC}}=P_{Y_A}^{\mathfrak{BC}}$. We note that $P_A^{\mathfrak{BC}}$ has the interpretation of being the probability of observing the boundary conditions $\mathfrak{BC}$ in a random configuration on $A$.

Given a set $S=\{(i_1,o_1),\dotsc, (i_k,o_k)\}$ with $i_j,o_j\in \pm [1,n]$, we let $\sigma\in B_n$ act by $\sigma \cdot S=\{(\sigma(i_1),o_1),\dotsc, (\sigma(i_k),o_k)\}$ and $S\cdot \sigma=\{(i_1,\sigma(o_1)),\dotsc, (i_k,\sigma(o_k))\}$. We define certain actions on the set of boundary conditions. For $\sigma$ in either $B_{C_{in}}$ or $B_{P_{in}}$, we write $\sigma\cdot \mathfrak{BC}=(\sigma\cdot \mathfrak{C},c_{in}\circ \sigma,c_{out},\sigma\cdot \mathfrak{P})$. For $\sigma$ in either $B_{C_{out}}$ or $B_{P_{out}}$, we write $\mathfrak{BC}\cdot \sigma=(\mathfrak{C}\cdot \sigma,c_{in},c_{out}\circ \sigma , \mathfrak{P}\cdot \sigma)$.

\begin{lemma}
\label{lem: bd cond trichot}
Let $i\in [N-K+1,n-1]$, and let $o\in [M-K+1,n-1]$. Then for all boundary conditions $\mathfrak{BC}$:
\begin{enumerate}
    \item Either $l(\pi)<l(s_i \pi)$ for all $\pi\in \sat(\mathfrak{BC})$, $l(\pi)>l(s_i\pi)$ for all $\pi\in\sat(\mathfrak{BC})$, or $\mathfrak{BC}=s_i\cdot \mathfrak{BC}$.
    \item Either $l(\pi)<l(\pi s_o)$ for all $\pi\in \sat(\mathfrak{BC})$, $l(\pi)>l(\pi s_o)$ for all $\pi\in \sat(\mathfrak{BC})$, or $\mathfrak{BC}=\mathfrak{BC}\cdot s_o$.
\end{enumerate}
Let $i\in[0,N-K-1]$, and let $o\in[0,M-K-1]$. Then for all boundary conditions $\mathfrak{BC}$:
\begin{enumerate}
    \item Either $l(\pi)<l(s_i \pi)$ for all $\pi\in \sat(\mathfrak{BC})$, $l(\pi)>l(s_i\pi)$ for all $\pi\in\sat(\mathfrak{BC})$, or $\mathfrak{BC}=s_i\cdot \mathfrak{BC}$.
    \item Either $l(\pi)<l(\pi s_o)$ for all $\pi\in \sat(\mathfrak{BC})$, $l(\pi)>l(\pi s_o)$ for all $\pi\in \sat(\mathfrak{BC})$, or $\mathfrak{BC}=\mathfrak{BC}\cdot s_o$.
\end{enumerate}
\end{lemma}
\begin{proof}
We will assume that $\sat(\mathfrak{BC})$ is non-empty, as otherwise the result is vacuously true. In particular, we assume that it is possible to satisfy the boundary conditions.

Since the proofs of statements (1) and (2) in each case are very similar to each other (and actually equivalent after inverting $\pi$, which exchanges the incoming and outgoing data), we will prove (1) for each case. Suppose that $i\in [N-K+1,n-1]$, and let $\pi\in\sat(\mathfrak{BC})$. Consider $\pi^{-1}(i)$ and $\pi^{-1}(i+1)$, which determine which of the three possible conclusions hold. If both of $i,i+1\in i(\mathfrak{P}$), then we can determine which of $\pi^{-1}(i)$ and $\pi^{-1}(i+1)$ is larger. If say $i+1\in i(\mathfrak{P})$, we can still determine this, because we must have $\pi^{-1}(i)\not\in P_{out}$, so it must be smaller than $\pi^{-1}(i+1)\in P_{out}$, and similarly if $i\in i(\mathfrak{P})$. In all these cases, either $l(\pi)<l(s_i\pi)$, or $l(\pi)>l(s_i\pi)$. Otherwise, $s_i\cdot \mathfrak{P}= \mathfrak{P}$, and $s_i\cdot \mathfrak{C}=\mathfrak{C}$ and $c_{in}\circ s_i=c_{in}$ since $i\in B_{P_{in}}$. Thus, $s_i\cdot \mathfrak{BC}=\mathfrak{BC}$.

We now assume that $i\in [0,N-K-1]$. As above if $i,i+1\in i(\mathfrak{C})$, we can determine which of $\pi^{-1}(i),\pi^{-1}(i+1)$ is larger. If $c_{in}(i)\neq c_{in}(i+1)$, then we can still determine this. Finally, in the case that $c_{in}(i)=c_{in}(i+1)=P_{out},-P_{out}$, we have $s_i\cdot \mathfrak{P}=\mathfrak{P}$, $c_{in}\circ s_i=c_{in}$, and $s_i\cdot \mathfrak{C}=\mathfrak{C}$, and so in this case we have $s_i\cdot \mathfrak{BC}=\mathfrak{BC}$.

\end{proof}

We will write $s_i\cdot \mathfrak{BC}<\mathfrak{BC}$ and $s_i\cdot \mathfrak{BC}>\mathfrak{BC}$ if either $l(s_i \pi)<l(\pi)$ or $l(s_i\pi)>l(\pi)$ for all $\pi\in\sat(\mathfrak{BC})$, and similarly for $\mathfrak{BC}\cdot s_o$. By Lemma \ref{lem: bd cond trichot}, for $i\in [0,N-K-1],[N-K+1,n-1]$ and $o\in [0,M-K-1],[M-K+1,n-1]$, if neither of the two inequalities hold, then $s_i\cdot\mathfrak{BC}=\mathfrak{BC}$ and similarly $\mathfrak{BC}\cdot s_o=\mathfrak{BC}$. Moreover, let us remark that if $i\in [0,N-K-1]$, then which of the three cases we are in depends only on the central boundary condition $\mathfrak{C}$, $c_{in}$, and $c_{out}$, and similarly if $o\in [0,M-K-1]$. If $i\in [N-K,n-1]$ or $o\in [M-K,n-1]$, then it depends only on $\mathfrak{P}$.

\begin{proposition}
\label{prop: hecke local/nonlocal}
Let $Y\in\mathcal{H}$, fix boundary conditions $\mathfrak{BC}=(\mathfrak{C},c_{in},c_{out},\mathfrak{P})$. Then if $s_i\in B_{C_{in}}$ or $B_{P_{in}}$,
\begin{equation}
\label{eq: hecke loc?}
    P_{R_i(p)Y}^{\mathfrak{BC}}=\begin{cases}P_Y^{\mathfrak{BC}}&\text{if }s_i\cdot \mathfrak{BC}=\mathfrak{BC},
    \\pP_Y^{s_i\cdot\mathfrak{BC}}+(1-q_ip)P_Y^{\mathfrak{BC}}&\text{if }s_i\cdot \mathfrak{BC}<\mathfrak{BC},
    \\q_ipP_Y^{s_i\cdot\mathfrak{BC}}+(1-p)P_Y^{\mathfrak{BC}}&\text{if }s_i\cdot \mathfrak{BC}>\mathfrak{BC},
    \end{cases}
\end{equation}
and if $s_o\in B_{C_{out}}$ or $B_{P_{out}}$,
\begin{equation}
\label{eq: hecke nonloc?}
    P_{YR_o(p)}^{\mathfrak{BC}}=\begin{cases}P_Y^{\mathfrak{BC}}&\text{if }\mathfrak{BC}\cdot s_o=\mathfrak{BC},
    \\pP_Y^{\mathfrak{BC}\cdot s_o}+(1-q_op)P_Y^{\mathfrak{BC}}&\text{if } \mathfrak{BC}\cdot s_o<\mathfrak{BC},
    \\q_opP_Y^{\mathfrak{BC}\cdot s_o}+(1-p)P_Y^{\mathfrak{BC}}&\text{if } \mathfrak{BC}\cdot s_o>\mathfrak{BC}.
    \end{cases}
\end{equation}
\end{proposition}
\begin{proof}
We prove \eqref{eq: hecke loc?}, \eqref{eq: hecke nonloc?} being similar. Let $S=\sat(\mathfrak{BC})$, and write $Y=\sum_{\pi\in S}c_\pi T_\pi+\sum_{\pi\in S}c_{s_i\pi} T_{s_i\pi}+S'$, where $S'$ contains only terms $T_\pi$ with neither $\pi,s_i\pi\in S$. Now $P_{R_i(p)T_\pi}^{\mathfrak{BC}}=0$ if $\pi,s_i\pi\not\in\sat(\mathfrak{BC})$, so we may assume $S'=0$.

If $s_i\cdot \mathfrak{BC}<\mathfrak{BC}$, then for all $\pi\in S$, $l(s_i\pi)<l(\pi)$, and so \eqref{eq: R_k left} implies that
\begin{equation*}
\begin{split}
    R_i(p)Y&=\sum_{\pi\in S}c_\pi (pq_iT_{s_i\pi}+(1-pq_i)T_\pi)+\sum_{\pi\in S}c_{s_i\pi}(pT_{\pi}+(1-p)T_{s_i\pi})
    \\&=\sum_{\pi\in S}((1-pq_i)c_\pi + pc_{s_i\pi})T_\pi+\sum_{\pi\in S}(pq_ic_\pi+(1-p)c_{s_i\pi})T_{s_i\pi}.
\end{split}
\end{equation*}
In particular, $P_{R_i(p)T_\pi}^{\mathfrak{BC}}=\sum_{\pi\in S}((1-pq_i)c_\pi + pc_{s_i\pi})=pP_Y^{s_i\cdot \mathfrak{BC}}+(1-q_ip)P_Y^{\mathfrak{BC}}$. Similarly, if $s_i\cdot\mathfrak{BC}>\mathfrak{BC}$, then
\begin{equation*}
\begin{split}
    R_i(p)Y&=\sum_{\pi\in S}c_\pi (pT_{s_i\pi}+(1-p)T_\pi)+\sum_{\pi\in S}c_{s_i\pi}(q_ipT_{\pi}+(1-q_ip)T_{s_i\pi})
    \\&=\sum_{\pi\in S}((1-p)c_\pi + q_ipc_{s_i\pi})T_\pi+\sum_{\pi\in S}(pc_\pi+(1-q_ip)c_{s_i\pi})T_{s_i\pi},
\end{split}
\end{equation*}
and so $P_{R_i(p)T_\pi}^{\mathfrak{BC}}=\sum_{\pi\in S}((1-p)c_\pi + q_ipc_{s_i\pi})=q_ipP_Y^{s_i\cdot\mathfrak{BC}}+(1-p)P_Y^{\mathfrak{BC}}$.

If $s_i\cdot \mathfrak{BC}=\mathfrak{BC}$, we can essentially repeat the same proof except we replace $S$ with the subset containing the shorter of $\pi$ and $s_i\pi$ for each $\pi$, and note that at the end we end up with sums over both $S$ and $s_i S$, and means the coefficients of each $c_\pi$ will sum to $1$.
\end{proof}

We record as a corollary the statement for the partition function of the six-vertex model. We let $A$ be a symmetric skew domain, with $|P|=|Q|=2n$, and we use $P_A^{\mathfrak{BC}}=\sum_{\pi\in\sat(\mathfrak{BC})}P_A^\pi$ to denote the partition function for configurations satisfying $\mathfrak{BC}$. Then as $P_{Y_A}^\pi=P_A^\pi$ for all $\pi$, we have $P_A^{\mathfrak{BC}}=P_{Y_A}^{\mathfrak{BC}}$. Notationally, if it is important to keep track of the dependence on $\mathbf{z}$, we will write $P_A^{\mathfrak{BC}}(\mathbf{z})$ and so on. We will also later consider the case when the incoming colors are not all ordered, so we will use $P_A^{\sigma,\mathfrak{BC}}=\sum_{\pi\in \sat(\mathfrak{BC})}P_A^{\sigma,\pi}$. 

We will let $R_i(p)\cup A$ and $A\cup R_o(p)$ denote the domain $A$ with a vertex connecting edges $i$ and $i+1$ in $P$ or $o$ and $o+1$ in $Q$, of parameter $p$. There are two special cases which we will consider. The first is when the edges $i$ and $i+1$ are both either vertical or horizontal, in which case we will add a \emph{Yang--Baxter vertex}. This should be thought of as a vertex tilted at a $45^\circ$ angle, and these will be used with the Yang--Baxter and reflection equations. The other case is when one edge is vertical and the other horizontal, in which case we are simply adding a regular vertex, growing the domain $A$.

\begin{corollary}
\label{cor: local/nonlocal}
Fix boundary conditions $\mathfrak{BC}$. Then if $s_i\in B_{C_{in}}$ or $B_{P_{in}}$,
\begin{equation}
\label{eq: loc?}
    P_{R_i(p)\cup A}^{\mathfrak{BC}}=\begin{cases}P_A^{\mathfrak{BC}}&\text{if }s_i\cdot \mathfrak{BC}=\mathfrak{BC},
    \\pP_A^{s_i\cdot\mathfrak{BC}}+(1-qp)P_A^{\mathfrak{BC}}&\text{if }s_i\cdot \mathfrak{BC}<\mathfrak{BC},
    \\qpP_A^{s_i\cdot\mathfrak{BC}}+(1-p)P_A^{\mathfrak{BC}}&\text{if }s_i\cdot \mathfrak{BC}>\mathfrak{BC},
    \end{cases}
\end{equation}
and if $s_o\in B_{C_{out}}$ or $B_{P_{out}}$,
\begin{equation}
\label{eq: nonloc?}
    P_{A\cup R_o(p)}^{\mathfrak{BC}}=\begin{cases}P_A^{\mathfrak{BC}}&\text{if }\mathfrak{BC}\cdot s_o=\mathfrak{BC},
    \\pP_A^{\mathfrak{BC}\cdot s_o}+(1-qp)P_A^{\mathfrak{BC}}&\text{if } \mathfrak{BC}\cdot s_o<\mathfrak{BC},
    \\qpP_A^{\mathfrak{BC}\cdot s_o}+(1-p)P_A^{\mathfrak{BC}}&\text{if } \mathfrak{BC}\cdot s_o>\mathfrak{BC}.
    \end{cases}
\end{equation}
\end{corollary}
\begin{proof}
This follows immediately by taking $Y=Y_A$ in Proposition \ref{prop: hecke local/nonlocal}.
\end{proof}
\begin{remark}
We note that these relations depend only on the choices of parameters $N,M,K$, and have nothing to do with the structure of the domain $A$. As we will see later however, we will wish to choose these parameters in a manner compatible with $A$ in order to obtain useful consequences.
\end{remark}

\section{Flip and shift invariance for the six vertex model}
\label{sec: flip}
In this section, we establish flip and shift invariance for the half space stochastic six-vertex model. We begin with establishing a much simpler symmetry, namely color-position symmetry, which follows immediately from the Hecke algebra viewpoint.

\subsection{Color-position symmetry}
Fix a symmetric skew domain $A\subseteq [L_1,L_2]^2$, and let $\rev_{[L_1,L_2]}(A)$ denote the reflection of $A$ about the antidiagonal of the square $[L_1,L_2]^2$. For $\mathbf{z}=(z_1,\dotsc,z_n)$, we let
\begin{equation*}
    \rev_{[a,b]}(\mathbf{z})=(z_1,\dotsc, z_{a-1},z_b,z_{b-1},\dotsc,z_a,z_{b+1},\dotsc, z_n).
\end{equation*}
We will also let $\rev_{[a,b]}$ denote the involution of $\mathcal{H}$ induced by the corresponding permutation of $\mathbf{z}$.

The following theorem does not appear to have been written down anywhere, but very similar statements have appeared for both the full space six-vertex model \cite{G21}, as well as interacting particle systems in a half space \cite{B20, K20}. The proof is identical in all cases, and uses the anti-involution of $\mathcal{H}$ taking $T_{\pi}$ to $T_{\pi^{-1}}$.
\begin{theorem}[Color-position symmetry]
We have
\begin{equation*}
    P_A^{\pi}(\mathbf{z})=P_{\rev_{[L_1,L_2]}(A)}^{\pi^{-1}}(\rev_{[L_1,L_2]}(\mathbf{z})).
\end{equation*}
\end{theorem}
\begin{proof}
Consider the anti-involution of $\mathcal{H}$ taking $T_\pi$ to $T_{\pi^{-1}}$ (see e.g. \cite[Chapter 7, Exercise 1]{H90}). This takes $Y_A$ to $\rev_{[L_1,L_2]}(Y_{\rev_{[L_1,L_2]}})$ (which is just $Y_A$ with the factors in reverse order), from which the result follows.
\end{proof}

\subsection{Symmetric rectangles}
\label{sec: symm rect}

We will work with special domains $A_{M,N,K}$, which should be thought of as the analogues of rectangles in the half space setting.

\begin{definition}
Given $M,N,K\in \N$, we let $n=M+N-K$,
\begin{equation*}
    A^+_{M,N,K}=\{(x,y)\in\N^2\mid x\leq N-1, N-K\leq y\leq n-1, x\leq y\},
\end{equation*}
and we let $A_{M,N,K}$ denote the symmetrization of $A_{M,N,K}^+$. We call these domains \emph{symmetric rectangles}. We let $P$ and $Q$ denote the bottom left and top right boundaries of $A_{M,N,K}$, each of length $2n$.
\end{definition}

See Figure \ref{fig:sym rect} for an example of a symmetric rectangle. Note that $M,N,K$ are $1$ more than the lengths of the corresponding line segments, and should really be thought of as the number of edges or vertices for that segment.

\begin{figure}
    \centering
    \includegraphics[scale=0.7]{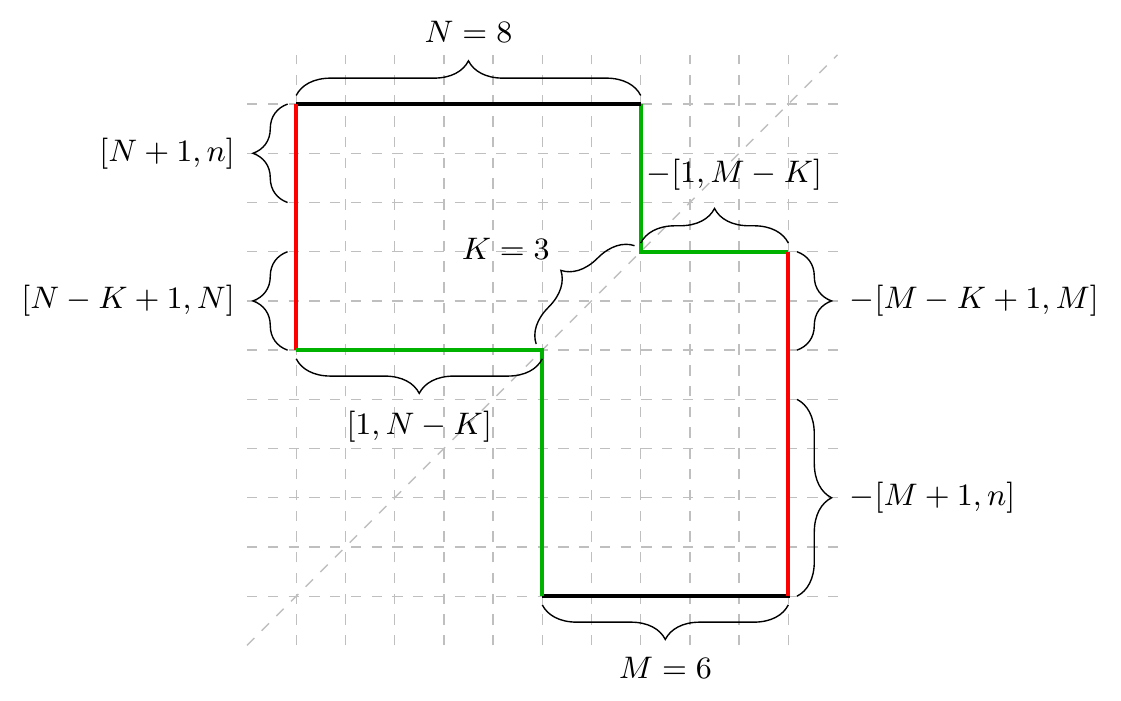}
    \caption{A symmetric rectangle with $M=6$, $N=8$, and $K=3$. Important intervals for the incoming and outgoing colors are also shown. Central boundary conditions correspond to specifying arrows between the green regions. Peripheral boundary conditions correspond to specifying arrows between the red regions (and by symmetry the black regions).}
    \label{fig:sym rect}
\end{figure}

\begin{remark}
These domains are exactly the smallest symmetric skew domains containing a given rectangle intersecting the diagonal. In fact, except when the boundary cuts a corner, each symmetric rectangle is generated by one of two non-symmetric rectangles, one of size $(M-1)\times (N-1)$ with a corner cut off by the diagonal, and one of size $(K-1)\times (n-1)$ cut in the middle by the diagonal.
\end{remark}

For now, we will work in a fixed symmetric rectangle $A=A_{M,N,K}$. We note that $Y_{A}\in B_n$, and so we will assume (for now) that we are working in $\mathcal{H}=\mathcal{H}_{q,t,\nu}(B_n,\mathbf{z})$. We will also fix the incoming colors $\sigma=\id$.

\subsection{Flip invariance}
We are now ready to state and prove flip invariance. We assume that the incoming colors are ordered, i.e. $\sigma=\id$, unless explicitly stated otherwise.

For a peripheral boundary condition $\mathfrak{P}=\{(i_1,o_1),\dotsc, (i_k,o_k)\}$, we let $\rot(\mathfrak{P})=\{(N-M-o_1,N-M-i_1),\dotsc,(N-M-o_k,N-M-i_k)\}$. Note that if $N\neq M$, it is possible that this will not correspond to a proper boundary condition, but as we will see, in all cases we consider this will not be the case (or for any specific $\mathfrak{BC}$, we can extend $P$ and $Q$ until this operation becomes well-defined). We let $\rot(\mathfrak{BC})=(\mathfrak{C},c_{in},c_{out},\rot(\mathfrak{P}))$ (and in particular we do not change $\mathfrak{C}$, $c_{in}$, or $c_{out}$). We note that $\rot(\mathfrak{BC})$ is consistent if $\mathfrak{BC}$ is. Recall that for $\mathbf{z}=(z_1,\dotsc,z_n)$, we let
\begin{equation*}
    \rev_{[a,b]}(\mathbf{z})=(z_1,\dotsc, z_{a-1},z_b,z_{b-1},\dotsc,z_a,z_{b+1},\dotsc, z_n).
\end{equation*}
We will also use the notation
\begin{equation*}
    \rev_{[a,b]^c}(\mathbf{z})=(z_n,\dotsc,z_{b+1},z_a,\dotsc, z_b,z_{a-1},\dotsc,z_1).
\end{equation*}

\begin{theorem}
\label{thm: flip}
We have
\begin{equation*}
    P_A^{\mathfrak{BC}}(\mathbf{z})=P_A^{\rot(\mathfrak{BC})}(\rev_{[N-K+1,N]}(\mathbf{z})).
\end{equation*}
\end{theorem}
\begin{proof}
We first fix $\mathfrak{C}$, $c_{in}$, and $c_{out}$, as well as $|\mathfrak{P}|=k$. Note that unless $\mathfrak{P}$ only connects the subinterval $[N-K+1,N]$ to $-[M-K+1,M]$, both sides will be $0$, so we may restrict to this case (and in particular, $\rot(\mathfrak{BC})$ makes sense as a boundary condition). Define a partial order on the set of possible peripheral boundary conditions $\mathfrak{P}$ by defining $\mathfrak{P}\preceq \mathfrak{P}'$ if $o(\mathfrak{P})_j\leq  o(\mathfrak{P}')_j$ for all $j$, where $o(\mathfrak{P})_j$ denotes the $j$th smallest element in $o(\mathfrak{P})$. We induct on this partial order.

\begin{figure}
    \centering
    \begin{subfigure}{0.45\textwidth}
    \includegraphics[scale=1]{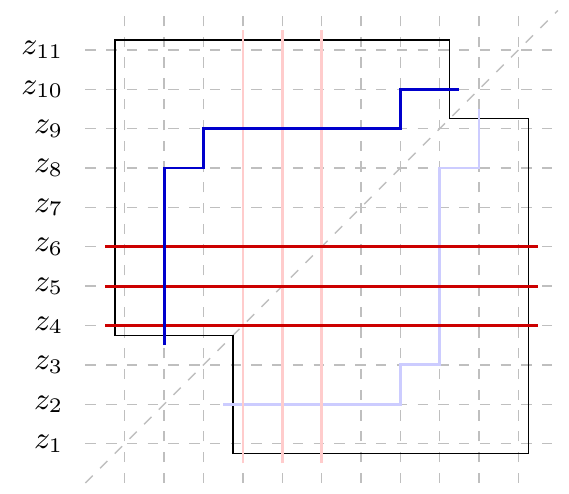}
    \caption{Minimal boundary condition.}
    \label{fig:base case unflip}
    \end{subfigure}
    \begin{subfigure}{0.45\textwidth}
    \includegraphics[scale=1]{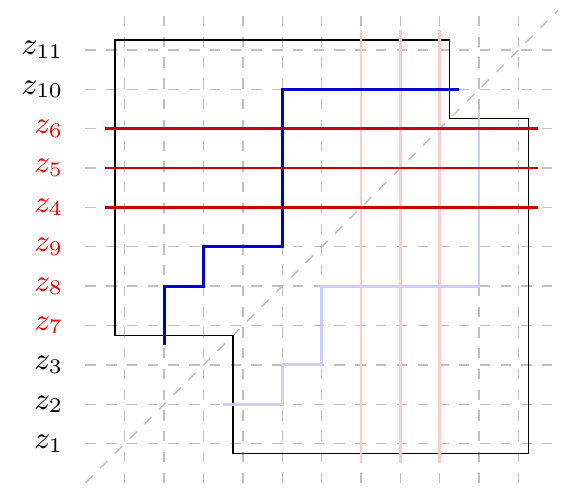}
    \caption{Maximal boundary condition.}
    \label{fig:base case flip}
    \end{subfigure}
    \label{fig: base case}
    \caption{The base case for the induction argument. The red arrows are forced to cross by $\mathfrak{P}$, and freeze the region they pass. An example of how the bijection between configurations for the two boundary conditions is shown (with only one peripheral arrow drawn in blue).}
\end{figure}

The base case is the smallest element, which corresponds to the packed boundary condition $\mathfrak{BC}_0$ with $\mathfrak{P}_0=\{(N-K+1,-M),\dotsc, (N-K+k,-M+k-1)\}$. In this case, the arrows specified by $\mathfrak{P}_0$ are forced to move deterministically, freezing the region they pass. We note that there is a weight preserving bijection between the configurations in $P_A^{\mathfrak{BC}_0}(\mathbf{z})=P_A^{\rot(\mathfrak{BC}_0)}(\mathbf{z}')$, where $\mathbf{z}'$ matches $\mathbf{z}$ outside the interval $[N-K+1,N]$, and swaps the $z_x$ in the intervals $[N-K+1,N-K+k]$ and $[N-K+k+1,N]$, see Figure \ref{fig:base case flip}. Thus, we must show that we can reverse the rapidities $z_x$ for $x$ within each interval $[N-K+1,N-K+k]$, and $[N-K+k+1,N]$, after we've performed this swap. 

The interval $[N-K+1,N-K+k]$ is easy, because it corresponds to the frozen region, whose weights are fixed because the colors entering the frozen region always have absolute value less than the colors frozen by $\mathfrak{C}_0$. To see that we can swap the $z_x$ for $x\in [N-K+k+1,N]$, we use a standard Yang--Baxter argument. In particular, it suffices to show that we can swap the rapidities in adjacent rows. To do so, we introduce a Yang--Baxter vertex on the right between rows $x$ and $x+1$ (and also columns $x$ and $x+1$ by symmetry), which does not change the partition function by Corollary \ref{cor: local/nonlocal}. Using the Yang--Baxter and reflection equations, we may move the vertex to the left, where it still does not change the partition function due to Corollary \ref{cor: local/nonlocal}. However, this has exchanged the rapidities of rows $x$ and $x+1$. Repeating this allows us to reverse the rapidities, and so $P_A^{\mathfrak{BC}_0}(\mathbf{z})=P_A^{\rot(\mathfrak{BC}_0)}(\rev_{[M-K+1,M]}(\mathbf{z}))$.

We now proceed with the inductive step. Suppose we have some arbitrary peripheral boundary condition $\mathfrak{P}$ with $|\mathfrak{P}|=k$. We may assume that $\mathfrak{P}\neq \mathfrak{P}_0$, and so there is a least (negative) $o\geq -M$ such that $o\in o(\mathfrak{P})$ and $o-1\not\in o(\mathfrak{P})$. Let $x$ and $x+1$ denote the corresponding rows, and place a Yang--Baxter vertex between these rows. The Yang--Baxter and reflection equations let us move it to the other side, which gives us an equation
\begin{equation}
\label{eq: unflip rec}
aP_A^{\mathfrak{BC}}(\mathbf{z})+bP_A^{\mathfrak{BC}\cdot s_o}(\mathbf{z})=cP_A^{\mathfrak{BC}\cdot s_o}(\mathbf{z}')+dP_A^{s_i\cdot \mathfrak{BC}\cdot s_o}(\mathbf{z}'),
\end{equation}
where $i=N-M-o$, $\mathbf{z}'$ is just $\mathbf{z}$ with $z_x$ and $z_{x+1}$ swapped, and $a,b,c,d$ are coefficients determined from $\mathfrak{BC}$ by Corollary \ref{cor: local/nonlocal}.

We now do the same for $\rot(\mathfrak{P})$. Let $y$ and $y+1$ denote the rows corresponding to the largest $i$ such that $i+1\not\in i(\rot(\mathfrak{P}))$ (which is just the rotation of $o$). We place a Yang--Baxter vertex between rows $y$ and $y+1$, and move it to the other side. This gives us the equation
\begin{equation}
\label{eq: flip rec}
\begin{split}
&a'P_A^{\rot(\mathfrak{BC})}(\rev_{[N-K+1,N]}(\mathbf{z}))+b'P_A^{\rot(\mathfrak{BC}\cdot s_o)}(\rev_{[N-K+1,N]}(\mathbf{z}))
\\=&c'P_A^{\rot(\mathfrak{BC}\cdot s_o)}(\rev_{[N-K+1,N]}(\mathbf{z}'))+d'P_A^{\rot(s_i\cdot \mathfrak{BC}\cdot s_o)}(\rev_{[N-K+1,N]}(\mathbf{z}')),
\end{split}
\end{equation}
where again $a',b',c',d'$ are coefficients determined from $\rot(\mathfrak{BC})$ by Corollary \ref{cor: local/nonlocal}.

Note that $\mathfrak{BC}\cdot s_o,s_i\cdot \mathfrak{BC}\cdot s_o\prec \mathfrak{BC}$, and so by the inductive hypothesis, the partitions functions appearing in \eqref{eq: unflip rec} and \eqref{eq: flip rec} are equal. Furthermore, the coefficients $a,b,c,d$ and $a',b',c',d'$ are equal as well by Corollary \ref{cor: local/nonlocal}, and the fact that $\rot$ commutes with the actions of $s_i$ and $s_o$ (up to some change of the $i$ and $o$). Thus, we have $P_A^{\mathfrak{BC}}(\mathbf{z})=P_A^{\rot(\mathfrak{BC})}(\rev_{[M-K+1,M]}(\mathbf{z}))$.
\end{proof}

We actually need a slightly more general result, requiring a generalization of our original setup.

First, let us note that there is no harm in extending $P$ and $Q$ together, so we may assume that $P$ and $Q$ extend far away from $A_{M,N,K}$. By shifting our setup if necessary, we can thus always assume we are working in a large square $[1,L]^2$ where $P,Q$ start and end at $(1,L)$ and $(L,1)$ respectively. This in particular means $\rot(\mathfrak{BC})$ is a valid boundary condition. This also means we work in the Hecke algebra $\mathcal{H}_{[0,L]}$. We will not focus on this issue and just assume that $P,Q$ can be extended if needed. We can now shift and reindex everything so that $A_{M,N,K}$ still has the same coordinates as above. This means that we will potentially have negative coordinates and negatively indexed rapidities, but beyond this nothing changes. 

Recall that we use $\rev_{[a,b]}$ to denote the induced involution of $\mathcal{H}$ induced by the permutation on $\mathbf{z}$. We define an anti-homomorphism between $\mathcal{H}_{P_0}$ and $\mathcal{H}_{P_{in}}$ by $(T_{k_1}\dotsm T_{k_m})^*=T_{N-M-k_m}\dotsm T_{N-M-k_1}$ (for $k_i\geq N+1$ or $k_i\leq -M-1$). As with $\rot(\mathfrak{P})$, technically speaking, this is not exactly well-defined, so we need to extend the smaller of $P_{in}$ and $P_{out}$ until they're of the same size. However, we may always pick $P$ and $Q$ (and thus $P_{in}$ and $P_{out}$) large enough so this is not an issue.
\begin{theorem}
\label{thm: gen flip}
Fix boundary conditions $\mathfrak{BC}$, and for $X=P_{in},P_{out},C_{in},C_{out}$, let $Y_{X}\in \mathcal{H}_X$. Then
\begin{equation*}
    P_{Y_{P_{in}}Y_{C_{in}}Y_AY_{C_{out}}Y_{P_{out}}}^{\mathfrak{BC}}=P_{Y_{P_{out}}^*Y_{C_{in}}\rev_{[N-K+1,N]}(Y_A)Y_{C_{out}}Y_{P_{in}}^*}^{\rot(\mathfrak{BC})},
\end{equation*}
where we extend the Hecke algebra we work in if necessary (by extending $P$ and $Q$) to interpret $Y_{P_{in}}^*$ and $Y_{P_{out}}^*$.
\end{theorem}
\begin{proof}
The statement is linear in the $Y_X$, and so it suffices to prove it when they are all of the form $T_\pi$ for some $\pi$. We then proceed by induction on the total length of the permutations. The base case is when $Y_X=\id$ for all $X$, which is given by Theorem \ref{thm: flip}.

We now assume that it holds for $Y_X=T_{\pi_X}$. Working with $X=C_{in}$ first, we replace the permutation $\pi_{C_{in}}$ with $s_i\pi_{C_{in}}$ for $s_i\in \mathcal{H}_{C_{in}}$ where $l(s_i\pi_{C_{in}})>l(\pi_{C_{in}})$, and note that we can write $T_{s_i\pi_{C_{in}}}=T_iT_{\pi_{C_{in}}}$. Since $T_i$ commutes with $\mathcal{H}_{P_{in}}$, we can move $T_i$ to the left, and use \eqref{eq: loc?} with $p=1$ to express both sides as the same linear combination of expressions which are equal by the inductive hypothesis. A similar argument works for $X=P_{in},C_{out},P_{out}$ (with the definition of $Y^*$ exactly corresponding to the transformation $\rot$).
\end{proof}

\subsection{Invariance for height functions}
We are now ready to give the applications of flip invariance to invariance for height functions. As in the previous section, we continue to assume that the incoming colors satisfy $\sigma=\id$. We first make some definitions regarding crossings which are needed to formulate the various invariance statements.

Given two $(P,Q)$-cuts $C=(d,l,u,r)$ and $C'=(d',l',u',r')$, we write $C\preceq C'$ if $d\leq d'$, $u'\leq u$, $l'\leq l$, and $r\leq r'$. We say that $C$ and $C'$ \emph{cross} if either $C\preceq C'$ or $C'\preceq C$.

We say that a $(P,Q)$-cut $C=(d,l,u,r)$ \emph{crosses the diagonal} if $l\leq d$ and $r\geq u$. Visually, such cuts correspond to rectangles crossing the diagonal $x=y$. Given $C$ crossing the diagonal, we define $A_C$ to be the smallest symmetric skew domain containing $C$. Up to a diagonal shift, this will be a symmetric rectangle $A_{M,N,K}$. For now, we fix a $C$, and shift our coordinates so $A_{M,N,K}$ has the same coordinates as before.

\begin{figure}
    \centering
    \includegraphics[scale=0.7]{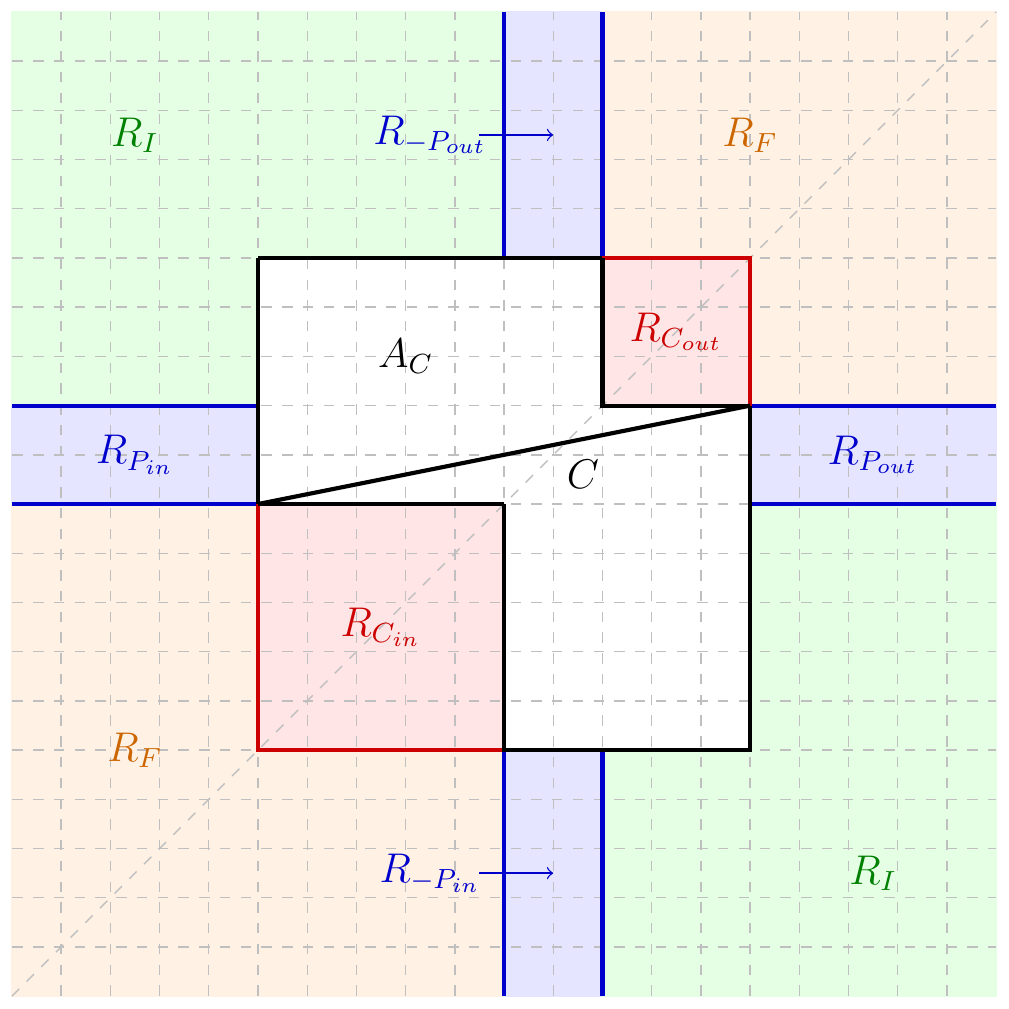}
    \caption{The region $A_C$ defined by the $(P,Q)$-cut $C$ is shown in black and white. For this choice of $C$, the regions $R_{C_{in}}$ and $R_{C_{out}}$ are shown in red, the regions $R_{P_{in}}$, $R_{-P_{in}}$, $R_{P_{out}}$, and $R_{-P_{out}}$ are shown in blue, $R_F$ is shown in orange, and $R_I$ is shown in green. Bold lines of a given color indicate that the line belongs to the corresponding region.}
    \label{fig:regions}
\end{figure}

Given the fixed $C$ and corresponding $M,N,K$ with $A_C=A_{M,N,K}$, we consider the following regions in the plane. We let $R_{P_{in}}$, $R_{C_{in}}$, $R_{P_{out}}$, and $R_{C_{out}}$ be defined by
\begin{align*}
    R_{P_{in}}&=\{(x,y)\mid  N-K\leq y\leq N-1, x< 0\}
    \\R_{C_{in}}&=\{(x,y)\mid 0\leq x,y< N-K\} 
    \\R_{P_{out}}&=\{(x,y)\mid n-M+K-1\leq y\leq N-1, x> n-1\}  
    \\ R_{C_{out}}&=\{(x,y)\mid n-M+K-1< x,y\leq n-1\}.
\end{align*}
We let $R_{-P_{in}}$ and $R_{-P_{out}}$ denote the reflections of $R_{P_{in}}$ and $R_{P_{out}}$ through the diagonal $x=y$, and we let $R_F$ denote the union of the remaining regions to the southwest and northeast, and $R_{I}$ denote the union of the remaining regions to the northwest and southeast ($F$ stands for forbidden and $I$ stands for irrelevant). See Figure \ref{fig:regions} for an example. 

If $A$ is a symmetric skew domain and $R_X$ is one of above regions, we let $A_X=A\cap R_X$. We first show that if $C$ is a $(P,Q)$-cut, then $A$ cannot intersect $R_F$.

\begin{lemma}
If $A$ is a symmetric skew domain and $C$ is a $(P,Q)$-cut crossing the diagonal, then $A\cap R_F=\emptyset$.
\end{lemma}
\begin{proof}
Since $C$ is a $(P,Q)$-cut, the points $(l-1,d-1)$ and $(r+1,u+1)$ cannot lie in $A$, and as $A$ is a symmetric skew domain, this means no part of $R_F$ can lie in $A$.
\end{proof}

Next, we establish the key technical lemma that the height functions $\Ht_A(C)$ are in some sense independent of what occurs outside of $C$.
\begin{lemma}
\label{lem: key ind lemma}
If $A$ is a symmetric skew domain and $\mathcal{C}=\{C_i\}$ is a collection of $(P,Q)$-cuts, then $\Ht_A(\mathcal{C})$ is independent of the randomness outside of the rectangles defined by the $C_i$ (and by symmetry their reflections along the line $x=y$). More specifically, for any deterministic choice of assignments of where arrows go outside of $\bigcup _i C_i\cup C'_i$, where $C'_i$ is the reflection of $C_i$ across $x=y$ (which occurs with positive probability), $\Ht_A(\mathcal{C})$ is a function of the outgoing colors in $\bigcup _i C_i\cup C'_i$, and has the same distribution as without this conditioning.
\end{lemma}
\begin{proof}
Since the $C_i$ are $(P,Q)$-cuts, all colors entering must be ordered, in the sense that colors entering the bottom are smaller than colors entering the left (and similarly for $C_i'$). We also note that the height function is entirely determined by what happens in $C_i$ (which by symmetry also affects what occurs in $C_i'$, but the point is we may ignore this), and only cares about positive colors, so we may forget the negative colors. 

Thus, we may project to a model where we only keep track of the arrows within $R=\bigcup C_i$, and the fact that arrows entering from the left are bigger than arrows entering from the bottom. When we sample within $R$, we only keep track of whether the particular arrow entered from the bottom or left of any particular $C_i$, and only while that arrow is in $C_i$. Since an arrow cannot leave and then re-enter $C_i$ (it may enter $C_i'$, but this will not matter), the fact that we forgot whether it entered from the left or bottom will not affect how it interacts with other arrows. 

Formally, we may view this as color projection for the model, except that the colors we keep track of change depending on the location of an edge. We divide $P$ into intervals based on where the $C_i$ intersect $P$, and project all colors in an interval to the same color. This is valid because the incoming colors are ordered so we only project adjacent colors. If an edge is to the north or east of all $C_i$ intersecting at a point on $P$, we also project the two intervals neighboring that point to a single color. This is valid because the set of such edges is itself a skew domain, and so once we've projected, we will never cross an edge where we have not projected.

Note that this is not equivalent to a full space model, since there are still symmetry constraints near the boundary, but $\Ht_A(\mathcal{C})$ can be entirely determined from the data encoded in this projection. A method of sampling this model would be to begin as usual with the Markovian sampling procedure at a vertex $(x,y)$, and continue as long as $(y,x)$ is not also in $R$. If $(y,x)\in R$, when we determine what occurs in $(x,y)$, we also by symmetry determine what occurs in $(y,x)$.

This projection is the same for all choices of where arrows go outside of $\bigcup _i C_i\cup C'_i$, and so the distribution of $\Ht_A(\mathcal{C})$ is independent of this conditioning.
\end{proof}
This independence immediately implies the following two useful corollaries.

\begin{lemma}
If $A$ is a symmetric skew domain, and $C$ is a $(P,Q)$-cut crossing the diagonal, and $\mathcal{C}=\{C_i\}$ is a collection of $(P,Q)$-cuts crossing $C$ and the diagonal, then $\Ht_A(\mathcal{C})$ is independent of the randomness in $R_I$, in the sense that we can condition on where the incoming arrows go at the outgoing boundary without changing the distribution of $\Ht_A(\mathcal{C})$.
\end{lemma}
\begin{proof}
This follows immediately from Lemma \ref{lem: key ind lemma}, as $R_I$ lies in the complement of $C_i$ and its reflection across $x=y$ for all $i$ (in fact, $C_i$ is contained in either $R_{P_{in}}\cup A_C\cup R_{P_{out}}$ or $R_{C_i}\cup A_C\cup R_{C_{out}}$ depending on how it crosses $C$).
\end{proof}

\begin{lemma}
\label{lem: irrev sq}
Let $A, A'$ be two symmetric skew domains, and let $\mathcal{C}=\{C_i\}$ denote a family of $(P,Q)$-cuts which are also $(P',Q')$-cuts. Then
\begin{equation*}
    \Ht_A(\mathcal{C})\deq \Ht_{A'}(\mathcal{C}).
\end{equation*}
\end{lemma}
\begin{proof}
We show both are equal to $\Ht_{A''}(\mathcal{C})$, where $A''$ is the smallest symmetric skew domain containing the cuts $\mathcal{C}$. This follows from Lemma \ref{lem: key ind lemma}, as any vertices not in $A''$ are in the complement of $C_i$ and its reflection across $x=y$ for all $i$, and the configurations are related by a deterministic assignment of what happens outside of $A''$.
\end{proof}

\begin{figure}
    \centering
    \begin{subfigure}{0.45\textwidth}
    \includegraphics[scale=0.5]{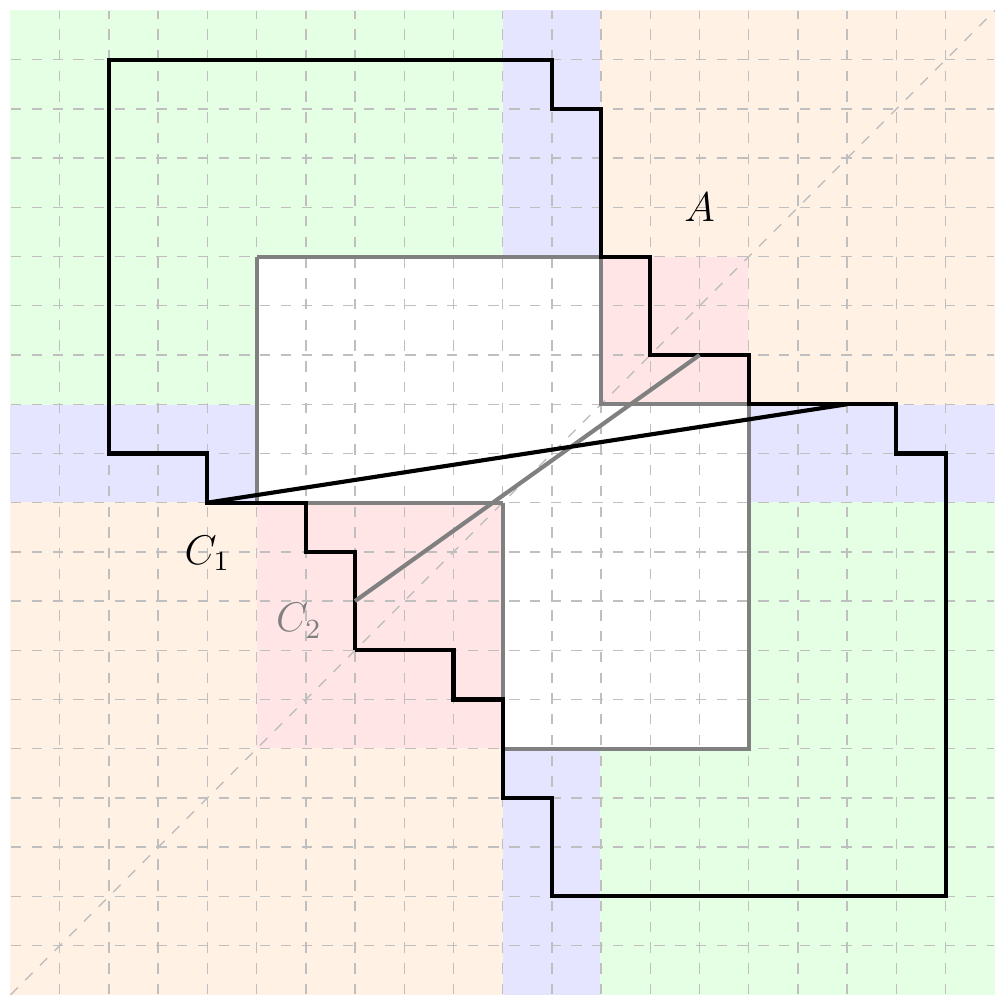}
    \caption{Unrotated domain and cuts.}
    \label{fig: rot ex no rot}
    \end{subfigure}
    \begin{subfigure}{0.45\textwidth}
    \includegraphics[scale=0.5]{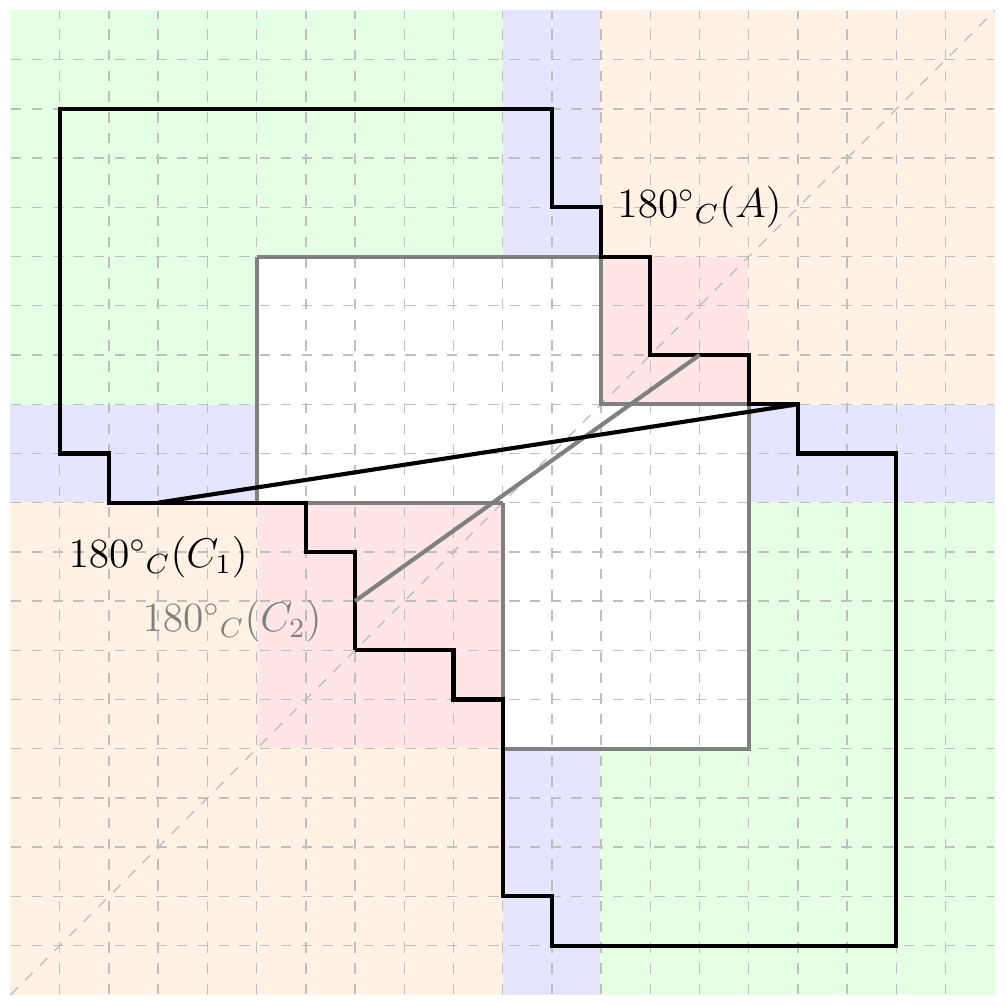}
    \caption{Rotated domain and cuts.}
    \label{fig: rot ex rot}
    \end{subfigure}
    \caption{An example showing how a symmetric skew domain $A$ and $(P,Q)$-cuts $C_1,C_2$ are transformed by $\rot_C$. Here, $C$ is the $(P,Q)$-cut generating the gray symmetric rectangle. The colors correspond to the regions $R_X$ as in Figure \ref{fig:regions}, and the parts of $A$ lying in $R_X$ are the $A_X$ (for $X=C_{in},C_{out},P_{in},-P_{in},P_{out},-P_{out},F,I$). Note the change in the green region to accommodate the difference sizes of $A_{P_{in}}$ and $A_{P_{out}}$.}
    \label{fig: rot ex}
\end{figure}

Given a symmetric skew domain $A$ and a $(P,Q)$-cut $C$ crossing the diagonal, we use $\rot_C(A)$, $\rev_C(\mathbf{z})$, and $\rot_C(\mathfrak{BC})$ to denote the corresponding operations $\rot$ and $\rev$ with respect to this setup (since we will wish to consider multiple choices for $(P,Q)$-cuts). In particular, $\rot_C(A)$ is the smallest symmetric skew domain containing the region defined by taking $A\setminus A_I$, and swapping $A_{P_{in}}$ and $A_{P_{out}}$ through a rotation (and symmetrically $A_{-P_{in}}$ and $A_{-P_{out}}$). If $A_C=A_{M,N,K}$, we have $\rev_C=\rev_{[N-K+1,N]}\circ \rev_{[0,n]^c}$. For any other $(P,Q)$-cut $C'=(d',l',u',r')$ crossing $C$, we let $\rot_C(C')=C'$ if $C'\preceq C$, and we let $\rot_C(C')$ denote the $180^\circ$ rotation of $C'$ around $C$ if $C\preceq C'$. See Figure \ref{fig: rot ex} for an example.

We note the following lemma, which states that the height functions at certain points on $P$ and $Q$ are functions of the boundary data $\mathfrak{BC}$.

\begin{lemma}
\label{lem: rev cut}
Let $A$ be a symmetric skew domain, let $C$ denote a $(P,Q)$-cut, and let $C'$ denote a $(P,Q)$-cut crossing $C$. Then $\Ht_A(C')$ as a function of $\mathfrak{BC}$ is equal to $\Ht_{\rot_C(A)}(\rot_C(C'))$ as a function of $\rot_C(\mathfrak{BC})$.
\end{lemma}
\begin{proof}
To see that $\Ht_A(C')$ is a function of $\mathfrak{BC}$, we note that either $C'\preceq C$, in which case we can determine $\Ht_A(C')$ from $\mathfrak{C}$, $c_{in}$, $c_{out}$, and the number of colors starting in $P_{in}$ and ending in $-P_{out}$ (which can be determined from $|\mathfrak{P}|$ and $c_{out}$), or $C\preceq C'$, in which case we can determine $\Ht_A(C')$ from $\mathfrak{P}$. To check the equality as functions of $\mathfrak{BC}$, we again consider the two cases. If $C'\preceq C$, then $\mathfrak{C}$, $c_{in}$, $c_{out}$, and $|\mathfrak{P}|$ stay the same even after applying $\rot_C$, and so $\Ht_A(C')=\Ht_{\rot_C(A)}(\rot_C(C'))$. Otherwise, if $C\preceq C'$, then we can check that $\Ht_A(C')=\Ht_{\rot_C(A)}(\rot_C(C'))$, since any $(i,o)\in \mathfrak{P}$ contributes to $\Ht_A(C')$ if and only if its rotation in $\rot_C(\mathfrak{P})$ contributes to $\Ht_{\rot_C(A)}(\rot_C(C'))$.
\end{proof}

\begin{theorem}
\label{thm: ht flip}
Let $A$ be a symmetric skew domain. Let $\mathcal{C}=\{C_i\}$ be a collection of $(P,Q)$-cuts, and let $C=(d,l,u,r)$ be a $(P,Q)$-cut crossing all $C_i$. Then
\begin{equation*}
    \Ht_A(\mathcal{C};\mathbf{z})\deq\Ht_{\rot_{C}(A)}(\rot_C(\mathcal{C});\rev_C(\mathbf{z})).
\end{equation*}
\end{theorem}
\begin{proof}
Since $C$ is a $(P,Q)$-cut, we can write $A$ as $A=A_C\cup A_{P_{in}}\cup A_{C_{in}}\cup A_{P_{out}}\cup A_{C_{out}}\cup A_{I}$, and the distribution of the height function is independent of what occurs in $A_I$. This means we can condition on what occurs, and so we can replace the factors of $R_{y-x}(\mathbf{p}_{x,y})$ in $Y_{A_I}$ with $1$ (we're conditioning on the event that what occurs in $A_I$ corresponds to $\id$), and so we are studying the element $Y=Y_{C_{in}}Y_{C_{out}}Y_{A_C}Y_{C_{out}}Y_{P_{out}}$, written in the manner required by Theorem \ref{thm: gen flip}. Similarly, we have that the height function on $\rot_C(A)$ is independent of what occurs in $A_I$, and so can be studied through the element $Y'=Y_{P_{out}}^*Y_{C_{in}}Y_{A_{C}}Y_{C_{out}}Y_{P_{in}}^*$ (where $Y_{P_{out}}^*$ corresponds exactly to the model in $R_{P_{in}}$ after applying $\rot$ and $\rev_C$, and similarly for $Y_{P_{in}}^*$).

Theorem \ref{thm: gen flip} tells us that the probability of observing certain boundary conditions $\mathfrak{BC}$ in $Y$ and $\rot_C(\mathfrak{BC})$ in $Y'$ is the same, and by Lemma \ref{lem: rev cut}, the height functions across cuts $C_i$ and $\rot_C(C_i)$ are functions of the boundary condition in a compatible manner.
\end{proof}

If $C=(d,l,u,r)$ denotes a rectangle, we will write $C+(x,y)=(d+y,l+x,u+y,r+x)$.
\begin{theorem}[Shift invariance]
\label{thm: shift inv}
Let $A$ and $A'$ denote two symmetric skew domains (with boundaries $(P,Q)$ and $(P',Q')$), and let $\mathcal{C}=\{C_j\}$ and $\mathcal{C}'=\{C_j'\}$ denote two collections of $(P,Q)$-cuts crossing the diagonal, such that $\mathcal{C}$ and $\mathcal{C}'+(0,1)$ are $(P',Q')$ cuts, and $C_j'+(0,1)\preceq C_k$ for all $j$ and $k$. Then
\begin{equation*}
    \Ht_A(\mathcal{C}\cup \mathcal{C}';\mathbf{z})=\Ht_{A'}(\mathcal{C}\cup (\mathcal{C'}+(0,1));\mathbf{z}'),
\end{equation*}
where $\mathbf{z}'$ is obtained from $\mathbf{z}$ by taking the rapidity $z_i$ corresponding to the first row above all the $C_j'$, and moving it to the first row below all the $C_j'+(0,1)$ (and shifting everything in between). 

Similarly, if $\mathcal{C}$ and $\mathcal{C}'+(1,0)$ are $(P',Q')$-cuts, and $C_j'+(1,0)\preceq C_k$ for all $j,k$, then
\begin{equation*}
    \Ht_A(\mathcal{C}\cup \mathcal{C}';\mathbf{z})=\Ht_{A'}(\mathcal{C}\cup (\mathcal{C'}+(1,0));\mathbf{z}'),
\end{equation*}
where $\mathbf{z}'$ is obtained from $\mathbf{z}$ by taking the rapidity $z_i$ corresponding to the first column to the right of all the $C_j'$, and moving it to the first column to the left of all the $C_j'+(1,0)$.
\end{theorem}
\begin{proof}
Consider $C=(d,l,u,r)$, where
\begin{align*}
    d&=\min\{d'\mid (d',l',u',r')\in \mathcal{C}'\},
    \\l&=\max \{l'\mid (d',l',u',r')\in \mathcal{C}'\},
    \\u&=\max\{u'+1\mid (d',l',u',r')\in \mathcal{C}'\},
    \\r&=\min \{r'\mid (d',l',u',r')\in \mathcal{C}'\}.
\end{align*}
By construction, $C_j',C_j'+(0,1)\preceq C\preceq C_k$. Replacing $A$ with the smallest symmetric skew domain containing $\mathcal{C}$, $\mathcal{C}'$, and $C$, we have that the $C_j,C_j'$ are still $(P,Q)$-cuts, and by Lemma \ref{lem: irrev sq}, this does not change the distribution of the height functions.

We apply flip invariance (Theorem \ref{thm: ht flip}) twice. First, apply it to $A$ with the cut $C$, noting that $\rot_C(\mathcal{C})=\mathcal{C}$, to obtain
\begin{equation*}
    \Ht_A(\mathcal{C}\cup \mathcal{C}';\mathbf{z})\deq \Ht_{\rot_C(A)}(\mathcal{C}\cup\rot_C(\mathcal{C}');\rev_C(\mathbf{z})).
\end{equation*}
We then apply flip invariance with the cut $C'$, which is $C$ with the bottom row removed (this is a $(P,Q)$-cut by our choice of $C$ and $A$), and note that $\rot_{C'}\circ\rot_C$ exactly corresponds to a shift of the $C_j'$, and leave the remaining $C_i$ alone. Similarly, $\mathbf{z}$ transforms correctly. Moreover, the $\mathcal{C}$ and $\mathcal{C}'+(0,1)$ are cuts for the transformed domain $\rot_{C'}\circ\rot_C(A)$, and the distribution is unchanged by changing this to $A'$ by Lemma \ref{lem: irrev sq}.

For the second part, the proof is identical, except in the definition of $C$, we add $1$ to $r$ instead of $u$, and we apply flip invariance to the cut $C$ and the cut $C'$ where we remove the leftmost column.
\end{proof}

\begin{remark}
By applying Theorem \ref{thm: shift inv} repeatedly (and possibly changing $A$), we obtain that the joint distribution of height functions for any collection of $(P,Q)$-cuts crossing the diagonal and each other is invariant under any shifting of the cuts, as long as they remain crossing the diagonal and each other.
\end{remark}

\begin{remark}
\label{rmk: bd necc}
Note that the condition that the cuts cross the diagonal seems to be necessary, as otherwise even for a single cut, $\Ht_A(C)$ and $\Ht_A(C+(0,1))$ or $\Ht_A(C+(1,0))$ would depend on a subset of the rapidities with a different cardinality. However, this would not be the case for a diagonal shift, and it is clear that for a single height function, $\Ht_A(C)\deq \Ht_{A+(1,1)}(C+(1,1))$ (with the rapidities shifted as well). It is unclear whether a form of shift invariance would hold for height functions near but not crossing the diagonal, and we leave this open.
\end{remark}

Theorem \ref{thm: 6vm shift simple} is an immediate consequence of Theorem \ref{thm: shift inv}.

\begin{proof}[Proof of Theorem \ref{thm: 6vm shift simple}]
For each choice of color cutoff and position, we obtain a rectangle $C_k$ or $C_j'$. Let $A$ be the smallest symmetric skew domain containing all the rectangles, and $A'$ the same but for $C_k$ and $C_j'+(0,1)$. The assumptions guarantee that the $C_k$ and $C_j'$ are all $(P,Q)$-cuts and the $C_k$ and $C_j'+(0,1)$ are all $(P',Q')$-cuts, and that $C_j'+(0,1)\preceq C_k$ for all $k,l$. Thus, we can apply Theorem \ref{thm: shift inv}. Finally, we can extend $A$ and $A'$ to the quadrant by Lemma \ref{lem: irrev sq}, giving the desired conclusion.
\end{proof}

\section{Colored half space ASEP}
\label{sec: asep}
In this section, we explain how to take a limit of the six vertex model to obtain the ASEP, and prove Theorem \ref{thm: shift inv asep}. See Section \ref{sec: intro asep} for the relevant definitions.

We fix parameters $q\in (0,1)$ (the same as for the six-vertex model), and $\beta>\alpha\geq 0$, and a fixed time $\tau>0$. We take a small parameter $\varepsilon>0$, and set
\begin{equation}
\label{eq: spec for asep lim}
z_i=1+\frac{(1-q)\varepsilon}{2},\qquad t=\frac{\alpha}{\beta},\qquad N=\varepsilon^{-1}\tau, 
\end{equation}
and let $\nu\in (-\infty,-1)$ solve
\begin{equation}
\label{eq: nu eq}
    \frac{1-q}{(1-\nu t)(1+1/\nu)}=\beta,
\end{equation}
which has a unique solution since the left hand side is a strictly increasing function going from $0$ to $\infty$. With these choices, we have (up to first order in $\varepsilon$) $\mathbf{p}_{x,y}\approx \varepsilon$ for $x\neq y$, and $\mathbf{p}_{x,x}\approx \beta \varepsilon$. Thus, for bulk vertices, we have
\begin{equation*}
    \PP\left(\vcenter{\hbox{\includegraphics[scale=0.5]{figures/vtx1.pdf}}}\right)\approx \varepsilon,\qquad \PP\left(\vcenter{\hbox{\includegraphics[scale=0.5]{figures/vtx2.pdf}}}\right)\approx 1-\varepsilon,\qquad \PP\left(\vcenter{\hbox{\includegraphics[scale=0.5]{figures/vtx3.pdf}}}\right)\approx q\varepsilon,\qquad \PP\left(\vcenter{\hbox{\includegraphics[scale=0.5]{figures/vtx4.pdf}}}\right)\approx 1-q\varepsilon,
\end{equation*}
and for boundary vertices, we have
\begin{equation*}
    \PP\left(\vcenter{\hbox{\includegraphics[scale=0.5]{figures/vtx1.pdf}}}\right)\approx \beta\varepsilon,\qquad \PP\left(\vcenter{\hbox{\includegraphics[scale=0.5]{figures/vtx2.pdf}}}\right)\approx 1-\beta\varepsilon,\qquad \PP\left(\vcenter{\hbox{\includegraphics[scale=0.5]{figures/vtx3.pdf}}}\right)\approx \alpha\varepsilon,\qquad \PP\left(\vcenter{\hbox{\includegraphics[scale=0.5]{figures/vtx4.pdf}}}\right)\approx 1-\alpha\varepsilon.
\end{equation*}

Let us first see heuristically that with these choices, as $\varepsilon\to 0$, the six-vertex model should converge to the ASEP. Note that as $\varepsilon\to 0$, the arrows in the six-vertex model will essentially always wiggle, alternating turns and almost never going straight. We can thus view each diagonal of the six vertex model as a Bernoulli process, where successes occur when a vertex does not have a turn. While these processes are updated sequentially, and the outcome of one can affect the other, as $\varepsilon\to 0$, this sequential updating will not matter. Thus, we would expect that these Bernoulli processes converge to Poisson processes, which are exactly the Poisson clocks in the ASEP.

This convergence has been shown in the full space setting (for uncolored models) in \cite{A17}, and subsequently extended to other settings, e.g. colored models \cite{Z22}, and half space models (similar but distinct from the ones considered here) \cite{BBCW18}. Here, we also give the main ideas but do not give a full verification, since it would be very similar to the arguments in the full space setting.

\begin{proposition}
\label{prop: 6vm conv asep}
With the specializations of \eqref{eq: spec for asep lim} and \eqref{eq: nu eq}, we have that
\begin{equation*}
    \Ht_\varepsilon^{\geq i}(N-x,N)\xrightarrow{\varepsilon\to 0} h^{\geq i}(\tau,x)
\end{equation*}
jointly in finite dimensional distributions (over all $i$ and $x$).
\end{proposition}
\begin{proof}
We note that after censoring the effects of transitions outside of an interval $\pm [1,L]$ for the ASEP, and outside of the strip $|x-y|\leq L$ for the six vertex model (in the sense that if a color tries to leave this strip, the vertex will immediately force it to turn back), the convergence is immediate from standard convergence of Bernoulli to Poisson processes (as both the ASEP and the output of the six vertex model can be viewed as a random product in $B_n$, the former with factors appearing through independent Poisson processes, and the latter with factors appearing through independent Bernoulli processes, where technically the Bernoulli processes are updated in a certain order, but since the probability of two successes at the same time goes to $0$, this will not affect the limit). 

Thus, what remains is to show that if we pick a large enough interval $\pm [1,L]$, then this censoring will have essentially no effect, giving the convergence in distribution even without the censoring. Since we only care about finitely many height functions, we can forget about the differences between colors outside of $\pm [1,L]$. Then the only possible issue would be if a color from outside $\pm [1,L]$ traveled far enough into the interval to affect the values of the height function, or a color we track travels outside this interval. But this is a question of the six vertex model and the ASEP away from the boundary (we can simply condition on what happens at the boundary and establish uniform bounds), and in particular can be handled the same way as for the full space models, see \cite{A17}.
\end{proof}

With this, shift invariance for the ASEP (Theorem \ref{thm: shift inv asep}) follows immediately from shift invariance for the six-vertex model.

\section{Absorption times for the type \texorpdfstring{$B$}{B} oriented swap process}
\label{sec: osp}
In this section, we prove a type $B$ analogue of a result of Bufetov, Gorin, and Romik on the asymptotics for the absorption time of the oriented swap process \cite{BGR20}. Our notation and conventions will follow the Coxeter group literature, and so we will refer to the model defined below as a type $B$ model rather than half-space model.

\subsection{Definition of models}
\subsubsection{Type \texorpdfstring{$B$}{B} oriented swap process}
The \emph{type $B$ oriented swap process} is a continuous time absorbing Markov chain on $B_n$. To each Coxeter generator, we associate an exponential clock of rate $1$, except for the generator $s_0$, which we associate with a rate $\mu$ clock. When a clock associated to a generator rings, we move from $\pi$ to $\pi s$ if $l(\pi s)>l(\pi)$.

This Markov chain can alternatively be viewed as a Markov chain on the set of arrangements of the numbers $-n,\dotsc, -1,1,\dotsc,n$. We start with the numbers in increasing order, and associate rate $1$ clocks to edges $(i,i+1)$ and a rate $\mu$ clock to $(-1,1)$. When a clock rings, we swap the numbers at positions $i$ and $i+1$, as well as $-i$ and $-i-1$, if the numbers in those positions are in increasing order. This is equivalent to the fully colored TASEP on a half-open interval, which can be obtained from the ASEP by setting the parameters $q=\alpha=0$. This is a type $B$ analogue of the oriented swap process, which is the same Markov chain but on the numbers $1,\dotsc, n$, and no symmetry constraints (or alternatively fully colored TASEP on a closed interval).

This Markov chain is absorbed at $\pi_0$ the longest element in $B_n$, or the configuration of numbers $n,\dotsc, 1,-1,\dotsc, -n$ in reverse order. Let $T_i$ denote the last time that a swap using $s_i$ occurs, and let $T=\max(T_0,\dotsc, T_{n-1})$ denote the time of absorption. 

More generally, one can define a partially colored model, which corresponds to arrangements of the numbers $-n,\dotsc, -k-1,0,\dotsc, 0,k+1,\dotsc, n$, where there are $2k$ $0$'s in the middle. In the language of Coxeter groups, this corresponds to taking the parabolic quotient with respect to $B_{k}$. In this case, the absorbing state is achieved at time $T^{(k)}=\max(T_{n-1},\dotsc, T_{k})$.

The goal of this section is to relate the distribution of $T$ to certain last passage times in a half space geometry, and study its fluctuations.

\subsubsection{Half space last passage percolation}
Consider a collection of exponential clocks $X_{i,j}$ associated to the vertices in the half quadrant $\{(i,j)\in \N^2\mid i\leq j\}$. When $i<j$, the clocks are rate $1$, and when $i=j$, the clocks are rate $\mu$. We define the \emph{point-to-point last passage time}
\begin{equation*}
    L(x,y)=\max_{\gamma:x\to y}\wt(\gamma),
\end{equation*}
where the maximum is over all up-right paths from $x$ to $y$, and $\wt(\gamma)=\sum_k X_{\gamma_k}$, where $\gamma_k$ are the vertices in the path $\gamma$. We define the \emph{restricted point-to-half-line last passage time} as
\begin{equation*}
    L(x,n)=\max\{L(x,(i,j))\mid i+j=n, i\leq j\}.
\end{equation*}
More generally, we will let $L(x,n,m)=\max\{L(x,(i,j))\mid i+j=n, i\leq j, j\leq m\}$ denote the last passage time between $x$ and the segment of the line $i+j=n$ below $m$. Note that it does not matter whether we restrict our paths to lie below the diagonal or not, because the symmetry constraint ensures that any path going above the diagonal has a weight equal to a path going below the diagonal.

We now recall the following result on the limiting distribution for these point to line LPP times.
\begin{theorem}
\label{thm: LPP asymptotics}
Let $F_{GUE}$ and $F_{GOE}$ denote the Tracy--Widom distributions for the GUE and GOE matrix ensembles respectively. Let $\Phi$ denote the distribution function for a standard normal random variable. Let $m/n$ converge compactly to $\gamma^2\in (0,\infty)$. Then if $\mu>\frac{\gamma}{1+\gamma}$,
\begin{equation*}
    \PP\left(\frac{L((0,0),n+m-2,m-1)-m\alpha}{m^{\frac{1}{3}}\sigma}\leq x\right)\to F_{GUE}(x),
\end{equation*}
if $\mu=\frac{\gamma}{1+\gamma}$,
\begin{equation*}
    \PP\left(\frac{L((0,0),n+m-2,m-1)-m\alpha}{m^{\frac{1}{3}}\sigma}\leq x\right)\to F_{GOE}^2(x),
\end{equation*}
and if $\mu<\frac{\gamma}{1+\gamma}$,
\begin{equation*}
    \PP\left(\frac{L((0,0),n+m-2,m-1)-m\alpha'}{m^{\frac{1}{2}}\sigma'}\leq x\right)\to \Phi(x).
\end{equation*}
Here, the constants $\alpha,\alpha',\sigma,\sigma'$ can be given explicitly as
\begin{alignat*}{4}
    &\alpha &&=(1+\gamma^{-1})^2 &\sigma&=\gamma^{-1}(1+\gamma)^{4/3}
    \\&\alpha'&&=\mu^{-1}+\frac{\gamma^{-2}}{1-\mu}\qquad &\sigma'&=\left(\mu^{-2}+\frac{\gamma^{-2}}{(1-\mu)^2}\right)^{1/2}
\end{alignat*}
\end{theorem}
This result follows immediately from a distributional identity between $L((0,0),n,m)$ and a point-to-point last passage time in a full space LPP model (see \cite[Eqs. 7.59, 7.60]{BR01a} for the case of geometric LPP, for which a limit to the exponential model can be taken), whose asymptotics are given in \cite{BBP05, O08} (see also Proposition 2.1 of \cite{BC11}). A similar result in the positive-temperature setting for the log-gamma polymer was recently proven in \cite{BW21}. For simplicity, we keep the boundary parameter $\mu$ fixed. If one varies $\mu$, the BBP phase transition near the critical point can also be observed.

\subsection{Distributional identities and conjectures}
The main result of this section is the following distributional identity for $T$ in terms of point to line passage times in a half space LPP model.
\begin{theorem}
\label{thm: osp lpp same}
We have
\begin{equation*}
    T\deq L((0,0),2n-2).
\end{equation*}
More generally,
\begin{equation*}
    T^{(k)}=\max(T_{n-1},\dots, T_k)\deq L((0,0),2n-2,n-1-k).
\end{equation*}
\end{theorem}

An immediate corollary of Theorem \ref{thm: osp lpp same} and the limit results given in Theorem \ref{thm: LPP asymptotics} are the following limit results for the fluctuations of the absorbing times in the type $B$ oriented swap process. As noted above, for simplicity we fix the boundary parameter $\mu$, but varying $\mu$ would allow us to obtain the BBP phase transition near the critical point as well.

\begin{corollary}
\label{cor: osp asymp}
Let $k$ vary with $n$ such that $\frac{n-k}{n+k}\to \gamma^2\in (0,1]$ compactly away from $0$. Then if $\mu>\frac{\gamma}{1+\gamma}$,
\begin{equation*}
    \PP\left(\frac{T^{(k)}-m\alpha}{m^{\frac{1}{3}}\sigma}\leq x\right)\to F_{GUE}(x),
\end{equation*}
if $\mu=\frac{\gamma}{1+\gamma}$,
\begin{equation*}
    \PP\left(\frac{T^{(k)}-m\alpha}{m^{\frac{1}{3}}\sigma}\leq x\right)\to F_{GOE}^2(x),
\end{equation*}
and if $\mu<\frac{\gamma}{1+\gamma}$,
\begin{equation*}
    \PP\left(\frac{T^{(k)}-m\alpha'}{m^{\frac{1}{2}}\sigma'}\leq x\right)\to \Phi(x).
\end{equation*}
The constants $\alpha,\alpha',\sigma,\sigma'$ are the same as in Theorem \ref{thm: LPP asymptotics}.
\end{corollary}

In analogy with what occurs in the type $A$ setting, we make the following conjecture. The type $A$ analogue was originally conjectured in \cite{BCGR21}, and recently proved by Zhang \cite{Z22}.
\begin{conjecture}
\label{conj: osp}
We have the distributional identity
\begin{equation*}
    (T_0,\dotsc, T_{n-1})\deq (L((0,0),(n-1,n-1)),\dotsc, L((0,0),(2n-2,0))).
\end{equation*}
\end{conjecture}
\begin{remark}
Using the standard coupling between TASEP and LPP, as well as shift invariance for a single LPP time, it is not hard to show that Conjecture \ref{conj: osp} holds for the marginal distributions (see proof of Theorem \ref{thm: osp lpp same}). The non-trivial statement is that it holds jointly. Theorem \ref{thm: osp lpp same} is a simple corollary of Conjecture \ref{conj: osp}, and provides some more evidence for the conjecture.

In the type $A$ case, this was established in \cite{Z22}, using shift invariance as a crucial tool. It is unclear whether the techniques developed in \cite{Z22} could also be used to attack Conjecture \ref{conj: osp}, and we leave this as an open question.
\end{remark}

\subsection{Proof of Theorem \ref{thm: osp lpp same}}
Our proof follows the strategy of \cite{BGR20} where the type $A$ analogue was established.
We first note that we have a coupling between the finite and infinite TASEPs, given by using the same clocks within the interval $\pm [1,n]$. A crucial fact is that certain events are equivalent under this coupling. The following result is an analogue of a similar result in the full space setting with a similar proof, see Lemma 3.3 of \cite{AHR09}.

\begin{proposition}
\label{prop: fin inf couple}
Couple the exponential clocks of the finite TASEP on $\pm [1,n]$ and the infinite TASEP on $\pm \N$, by using the same clocks within $\pm [1,n]$. Let $\widehat{h}$ and $h$ denote the height functions for the two models. Then
\begin{equation*}
\widehat{h}^{\geq i+1}(\tau,-i-1)=\min(h^{\geq i+1}(\tau,-i-1),n-i)
\end{equation*}
for all $i=0,\dotsc, n-1$.
\end{proposition}
\begin{proof}
First, notice that it suffices to study a system with particles at $-n$ to $-i-1$ for the finite system, and particles to the left of $-i-1$ in the infinite system, since for the purposes of the height functions in consideration these colors are all equivalent, and the remaining colors (which we view as holes), are also equivalent. It thus suffices to show that the equality is preserved at each jump (and since there are finitely many colors, there will be only finitely many jumps up until time $\tau$).

Since we can view the finite process as the infinite process with jumps between $n$ and $n+1$ censored, the $j$th rightmost particle in the finite process is always weakly to the left of the $j$th rightmost particle in the infinite process, and the only discrepancies are caused by jumps occurring in the infinite process, but blocked in the finite process.

Both events only care about the $n-i$ rightmost particles, which can never move left. We claim that the $j$th particle from the right cannot be blocked in this way until reaching position $n-j+1$. We prove this by induction. For $j=1$, this is clear. For the $j$th particle, any blocking must occur due to a discrepancy with the $j-1$th particle, which by induction must occur after it reaches $n-j+2$. But $\min(h^{\geq i}(\tau,-i),n-i)$ never sees these discrepancies, because they must occur past the point that we're counting the particles.
\end{proof}

With this, we may now apply shift invariance to obtain the following distributional identity.

\begin{lemma}
\label{lem: osp TASEP}
We have
\begin{equation*}
    \PP(T_i\leq \tau\text{ for }i=k,\dotsc, n-1)=\PP(h^{\geq 1}(\tau,-2i-1)\geq n-i \text{ for }i=k,\dotsc, n-1).
\end{equation*}
\end{lemma}
\begin{proof}
Note that $T_i\leq \tau$ if and only if $\widehat{h}^{\geq i+1}(\tau,-i-1)=n-i$. Then we have
	\begin{equation}
	\begin{split}
	&\PP(T_i\leq \tau\text{ for }i=k,\dotsc, n-1)
	\\=&\PP(\widehat{h}^{\geq i+1}(\tau,-i-1)=n-i\text{ for }i=k,\dotsc, n-1)
	\\=&\PP(h^{\geq i+1}(\tau,-i-1)\geq n-i\text{ for }i=k,\dotsc, n-1)
	\\=&\PP(h^{\geq 1}(\tau,-2i-1)\geq n-i\text{ for }i=k,\dotsc, n-1),
	\end{split}
	\end{equation}
where we use Proposition \ref{prop: fin inf couple} for the second equality and shift invariance (Theorem \ref{thm: shift inv asep}) for the third.
\end{proof}

\begin{proof}[Proof of Theorem \ref{thm: osp lpp same}]
Using Lemma \ref{lem: osp TASEP}, it suffices to express the event on the right hand side in terms of LPP. This is similar to the full space case, and uses a standard coupling between TASEP and LPP which we now explain.

Since we only care about height functions $h^{\geq 1}$, it suffices to consider a system with particles to the left of $-1$, and holes to the right of $1$ (where we recall that we work on $\pm \N$ and not $\Z$). We associate a growth process to this particle system (see Lemma 3.8 of \cite{BBCS18b} for example). We view particles as a segment of slope $-1$, and holes as a segment of slope $1$. The interface this defines at time $\tau$ is the same as the set of vertices $(x,y)$ with $L((0,0),(x,y))\leq\tau$ in the symmetric LPP. The event that $h^{\geq 1}(\tau,-2i-1)\geq n-i$ is then the same as the event that $L((0,0),(n+i-1,n-i-1))\leq \tau$. Together with Lemma \ref{lem: osp TASEP}, this establishes the claim.
\end{proof}

\section{Fusion}
\label{sec: fusion}
We now wish to take a collection of columns and rows within the half space six-vertex model, and view them as one column and row. This is similar to the fusion procedure for the full space six-vertex model, although some care needs to be taken due to the presence of negative colors, and the boundary needs to be studied separately. The goal of this section is to define the fused model and obtain a formula for the fused vertex weights, which we do by solving an explicit recurrence. Since the details are rather involved, we leave them for Appendix \ref{app: pf}.

\subsection{Definition of fused vertices}
A vertex in the fused (or higher spin) model is given by an assignment of incoming and outgoing colored arrows, with the restriction that at most $L$ arrows enter or exit horizontally, and at most $M$ arrows enter or exit vertically. For a bulk vertex, $L$ and $M$ may be different, but for a boundary vertex, $L=M$. For now, we set up notation, and leave the definition of the random model itself for later.

We will introduce a color $0$, so colors now lie in $\Z$, with the usual ordering. We let $(\Ab(x,y),\Bb(x,y),\Cb(x,y),\Db(x,y))$ denote the vector of colors of arrows at a vertex $(x,y)$, $\Ab(x,y)=(\Ab_i(x,y))_{i\in\Z}$, with $\Ab_i(x,y)$ denoting the number of incoming vertical arrows of color $i$, and the rest going clockwise around the vertex. We will call these \emph{color vectors}. Note that the color vectors lose the information of where arrows are, but we will see that the resulting model is still Markovian if certain assumptions are made. We will sometimes drop the coordinates $(x,y)$ when discussing an arbitrary vertex. We let $\Ab^+$ and $\Ab^-$ denote the positive and negative components, and $\Ab^{\geq i}$ and $\Ab^{\leq i}$ denote the components of color at least or at most $i$. We will use the notation $\Ab^\pm$ to refer to both $\Ab^+$ and $\Ab^-$ (and in particular not $\Ab$ with $0$ removed). We let $|\Ab|=\sum \Ab_i$, and so $|\Ab|=|\Ab^+|+|\Ab^-|+\Ab_0$. We will eventually stop keeping track of the color $0$, but for now, all color vectors will include a component counting arrows of color $0$.

We choose integers $L_i$ to denote the number of arrows for the $i$th row/column. We make the following specializations that will be in effect for the rest of the section. We set $t=0$, and we specialize the rapidities so that for each block of $L_i$ arrows, the rapidities are $z_i,qz_i,\dotsc, q^{L_i-1}z_i$. The incoming arrows will be color $i$ in the $i$th block on the left, and $-i$ in the $i$th block on the bottom. Thus, there are exactly $L_i$ arrows of color $\pm i$ for each $i$.

We now outline the motivation for this. With these specializations and boundary conditions, the vertices satisfy $q$-exchangeability, which allows us to forget the exact order of the incoming arrows and only remember the number of each color. With this, we can think of the model as consisting of a single row/column for each $L_i$, but where up to $L_i$ arrows can occupy it. We will then find formulas for the probabilities associated to these fused vertices. This model will ultimately be degenerated, analytically continued, and taken to a continuous limit in Section \ref{sec: higher spin}. In this section, we will focus on finding the formulas for the fused vertex weights.

Before working with the full model, we will need to find formulas for the vertex weights. For now, we forget about the full model, and instead work with a single vertex. Thus, we drop the coordinates $(x,y)$ in our notation. 

We define an \emph{(M,L)-fused bulk vertex} with \emph{spectral parameter $z$} to be $M$ columns and $L$ rows of bulk vertices, with column and row rapidities given by $\sqrt{z},q\sqrt{z},\dotsc,\allowbreak q^{M-1}\sqrt{z}$ and  $\sqrt{z},q\sqrt{z},\dotsc, q^{L-1}\sqrt{z}$ respectively. We define an \emph{(L,L)-fused boundary vertex} with \emph{spectral parameter $z$} to be $L$ columns and $L$ rows of bulk vertices, except on the diagonal where boundary vertices are placed, and the column and row rapidities are both given by $\sqrt{z},q\sqrt{z},\dotsc,q^{L-1}\sqrt{z}$.

\subsection{\texorpdfstring{$q$}{q}-exchangeability}
Given a distribution $\mu$ on arrangements $(a_1,\dotsc, a_L)$ of a fixed collection of colors, we say that $\mu$ is \emph{$q$-exchangeable} if
\begin{equation*}
    \mu(a_1,\dotsc, a_{i+1},a_{i},\dotsc, a_L)=q\mu(a_1,\dotsc, a_{i},a_{i+1},\dotsc, a_L)
\end{equation*}
if $a_{i+1}>a_i$. Fix color vectors $\Ab$ and $\Bb$. We say that a distribution on the incoming vertical arrows to a vertex, given that the number of each color is given by $\Ab$, is \emph{$q$-exchangeable in the columns} if the distribution of the colors, read from left to right, is $q$-exchangeable. Similarly, we say that a distribution of incoming horizontal arrows, given the color distribution $\Bb$, is \emph{$q$-exchangeable in the rows} if the distribution of colors, read from top to bottom, is $q$-exchangeable. Note that for a boundary vertex, we have a symmetry in the incoming colors, and this definition is consistent in that $q$-exchangeability of the rows and columns are equivalent for boundary vertices.

\begin{proposition}
\label{prop: q-exch bulk}
Consider an $(M,L)$-fused bulk vertex. Suppose that the distributions of incoming arrows are $q$-exchangeable in the rows and columns. Then the distributions of outgoing arrows are also $q$-exchangeable in the rows and columns.
\end{proposition}
Proposition \ref{prop: q-exch bulk} follows directly from what is known in the full space setting (see e.g. \cite[Proposition B.2.2]{BW18}, and note that our spectral parameter is the inverse of theirs), since up to a relabeling of the colors there is no distinction between the the $(M,L)$-fused bulk vertices and the vertices in the fused full space model. The proof is also very similar to the proof of Proposition \ref{prop: q-exch bd} giving $q$-exchangeability for boundary vertices, so we omit it.

\begin{proposition}
\label{prop: q-exch bd}
Consider an $(L,L)$-fused boundary vertex. Suppose that the distribution of incoming arrows is $q$-exchangeable in the rows and columns. Then the distribution of outgoing arrows is also $q$-exchangeable in the rows and columns.
\end{proposition}
\begin{proof}
We use the Yang--Baxter and reflection equations in a similar manner to the full space case, see e.g. \cite[Proposition B.2.2]{BW18}. We first make the observation that for a Yang--Baxter vertex with parameter $q^{-1}$, the outgoing arrows are automatically $q$-exchangeable, no matter the incoming arrows. We note that since the incoming arrow distribution is $q$-exchangeable, we may introduce a Yang--Baxter vertex with parameter $q$ on the left, say at rows $i$ and $i+1$ (and bottom by symmetry), which does not change the distribution entering the fused boundary vertex, and thus does not affect the outgoing distribution. We can then use the Yang--Baxter and reflection equations to move the Yang--Baxter vertex to the right (and top), and note that this implies the distribution of the two outgoing edges $i$ and $i+1$ are $q$-exchangeable. Since this is true for any two adjacent rows, the distribution of outgoing arrows is $q$-exchangeable.
\end{proof}

Propositions \ref{prop: q-exch bulk} and \ref{prop: q-exch bd} allow us to forget the exact positions of the colored arrows, and just remember the number for each color. To see this, simply note that as the incoming arrows at the incoming vertices are of one color for each block, the incoming arrows are $q$-exchangeable in the rows and columns. Then Propositions \ref{prop: q-exch bulk} and \ref{prop: q-exch bd} allow us to conclude that the distributions of all incoming and outgoing arrows at all vertices are $q$-exchangeable. But then given the color vector, we can recover the original distribution by $q$-exchangeability. Thus, keeping track of only the color vectors, we can still recover all information in the original model. From now on, we will forget about the exact positions of colors and keep track of only the color vectors.

We will thus define the \emph{fused} bulk vertex weights $W_{M,L}(z,q;\Ab,\Bb,\Cb,\Db)$ and \emph{fused} boundary vertex weights $W_{L}(z,q,\nu;\Ab,\Bb,\Cb,\Db)$ to be the associated distribution on outgoing color vectors $\Cb$ and $\Db$, given that the incoming colors have color vectors $\Ab$ and $\Bb$ and are $q$-exchangeable. Our next goal is to derive formulas for these weights.

\subsection{Fusion for bulk weights}
We now wish to find formulas for the vertex weights as a function of the color vectors in this case. We begin with the bulk vertices, which is identical to the full space case.

\begin{proposition}
\label{prop: full space fused wt}
Consider an $(M,L)$-fused bulk vertex, and suppose that the distribution of the incoming colors is $q$-exchangeable in both the rows and columns. Then the distribution for $\Cb$ and $\Db$, given the incoming color vectors $\Ab$ and $\Bb$, is given by the fused vertex weight
\begin{equation}
\label{eq: fused wt bulk}
\begin{split}
    &W_{M,L}(z,q;\Ab,\Bb,\Cb,\Db)
    \\=&(q^{M-1}z)^{|\Db|-|\Bb|}q^{|\Ab|L-|\Db|M}
    \\&\qquad\sum_{\mathbf{P}}\Phi_q\left(\Cb-\mathbf{P},\Cb+\Db-\mathbf{P};q^{L-1}z,q^{-1}z\right)\Phi_q\left(\mathbf{P},\Bb;q^{-L-M+1}/z,q^{-L}\right)
\end{split}
\end{equation}
subject to the conditions $\Ab+\Bb=\Cb+\Db$, $|\Ab|=|\Cb|=M$, and $|\Bb|=|\Db|=L$, where the sum is over $0\leq \mathbf{P}_i\leq \max(\Bb_i,\Cb_i)$, and
\begin{equation*}
    \Phi_q(\Ab,\Bb;x,y)=\left(\frac{y}{x}\right)^{|\Ab|}\frac{(x;q)_{|\Ab|}(y/x;q)_{|\Bb|-|\Ab|}}{(y;q)_{|\Bb|}}q^{\sum_{i<j}(\Bb_i-\Ab_i)\Ab_j}\prod _i{\Bb_i\choose \Ab_i}_q.
\end{equation*}
\end{proposition}
\begin{proof}
This is given by Theorem A.5 in \cite{BGW22}, up to some differences in notation and the substitution of $q^{M-1}z$ for $z$. This replacement is needed because of our choice of column/row rapidities.
\end{proof}

\begin{remark}
The weight given in Proposition \ref{prop: full space fused wt} is not quite what we wish to study, because $0$ is included as a color. In particular, in the formula, $|\Ab|=|\Cb|=L$ and $|\Bb|=|\Db|=M$. This will cause issues when we wish to analytically continue the weights in $q^{-L}$ and $q^{-M}$. In the full space setting, it is easy to drop $0$ as a color, and the formula simplifies. In the half space setting, the formulas seem to become much messier, and so we do not fix this issue at this point. It turns out that after we degenerate the vertex weights by setting $z=q^{-L-M+1}$, it will become much easier to remove $0$ as a color.
\end{remark}

\subsection{Fusion for boundary weights}
In this section, we will state the analogue of Proposition \ref{prop: full space fused wt} for the boundary vertices, given by the following theorem. We give a sketch of the proof, and leave the formal verification for Appendix \ref{app: pf}.

\begin{theorem}
\label{thm: fused bd wt}
We have
\begin{equation}
\label{eq: fused wt bd 2}
\begin{split}
    &W_{L}(z,q;\nu;\Ab,\Bb,\Cb,\Db)
    \\=&q^{(1-L)|\Bb^+|}(-\nu)^{|\Bb^+|-|\Cb^+|}\sqrt{z}^{-|\Bb^+|-|\Cb^+|} \Phi_q\left(\Cb^+,\Bb^+;-q^{1-L}\nu/\sqrt{z},-\sqrt{z}\nu\right),
\end{split}
\end{equation}
subject to the constraints $\Ab+\Bb=\Cb+\Db$ and $|\Ab|=|\Bb|=|\Cb|=|\Db|=L$, where
\begin{equation*}
    \Phi_q(\Ab,\Bb;x,y)=\left(\frac{y}{x}\right)^{|\Ab|}\frac{(x;q)_{|\Ab|}(y/x;q)_{|\Bb|-|\Ab|}}{(y;q)_{|\Bb|}}q^{\sum_{i<j}(\Bb_i-\Ab_i)\Ab_j}\prod _i{\Bb_i\choose \Ab_i}_q.
\end{equation*}
\end{theorem}

\begin{remark}
Our fused vertex model is actually equivalent to certain solutions to the Yang--Baxter and reflection equation essentially found by Mangazeev and Lu \cite{ML19} (although they only studied the rank one case, which corresponds to a single color, it is easy to see how to generalize their formulas to the higher rank case). See also the related works \cite{KOY19a,KOY19b}. Mangazeev and Lu obtained the weights by solving a degenerate version of the reflection equation, and in fact obtained solutions even without the degeneration. These should correspond to the fused weights for the half-space six-vertex model without degenerating the parameter $t=0$. Unfortunately, the formulas are quite unwieldy and do not appear suitable for asymptotic analysis, and do not simplify when the bulk vertices are degenerated (as in Section \ref{sec: higher spin}). Additionally, there would be difficulties in extending the formulas to the colored case. Thus, we did not attempt to show that the more general weights can also be obtained with fusion in this work, and leave it as a conjecture.
\end{remark}

Unlike for the bulk vertices, the boundary vertex weights cannot be derived from what is known in the full space setting, and so we must explicitly set up recurrences and show that our formula for the weights solves those recurrences. We now sketch the important ideas in the proof.

To simplify our notation, notice that by symmetry, in the boundary vertices, we can just keep track of $\Bb$ and $\Cb$. For the rest of this section, we will write
\begin{equation*}
    W_{L}(z,q,\nu;\Ab,\Bb,\Cb,\Db)=W_{L}(z,q,\nu;\Bb,\Cb),
\end{equation*}
and simply note that as $\Ab+\Bb=\Cb+\Db$, and $\Ab_i=\Bb_{-i}$ and $\Cb_i=\Db_{-i}$, we lose no information.

Our starting point is to note that a boundary vertex with $L$ rows and columns can be viewed as a combination of a boundary vertex with $L-1$ rows and columns, a boundary vertex with $1$ row and column, and a bulk vertex with $1$ row and $L-1$ columns (see Figure \ref{fig:fusion rec}). We condition on the color of the arrow entering the top row, immediately giving the following lemma.

\begin{figure}
    \centering
    \includegraphics[scale=0.7]{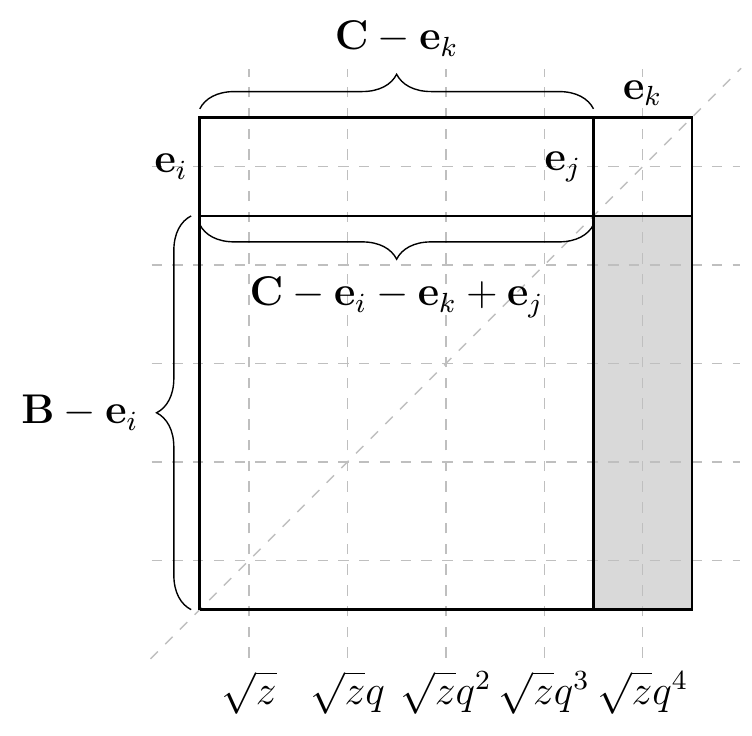}
    \caption{Recursion for boundary vertex weights when $L=4$. The rapidities are indicated, and the color vectors for colors crossing each boundary are noted.}
    \label{fig:fusion rec}
\end{figure}

\begin{lemma}
\label{lem: bd wt rec}
Let $\mathbf{e}_i$ denote the color vector corresponding to a single arrow of color $i$. We have the recurrence
\begin{equation*}
\begin{split}
    &W_{L+1}(z,q,\nu;\Bb,\Cb)
    \\=&\sum_{i,j,k} \PP(a_{L+1}=i)W_{L}(z,q,\nu;\Bb-\mathbf{e}_i,\Cb-\mathbf{e}_i-\mathbf{e}_k+\mathbf{e}_j) 
    \\&\qquad \times W_1(zq^{2L},q,\nu;j,k) W_{L,1}(zq^{L},q;\Cb-\mathbf{e}_i-\mathbf{e}_k+\mathbf{e}_j,i,\Cb-\mathbf{e}_k,j),
\end{split}
\end{equation*}
where $a_{L+1}$ is the color of the topmost incoming edge.
\end{lemma}

The following lemma gives the probability that the topmost incoming arrow in a boundary vertex is of color $i$. This is identical to the full space setting, see e.g. Section 8.6 in \cite{BGW22}, except that we look at the top rather than the bottom, which results in inverting $q$.
\begin{lemma}
\label{prop: q exch prob}
Suppose that a sequence of colors $(a_1,\dotsc, a_L)$ with color vector $\Ab$ has a $q$-exchangeable distribution. Then
\begin{equation*}
    \PP(a_L=i)=\frac{(1-q^{-\Ab_i})q^{-|\Ab^{\geq i+1}|}}{1-q^{-L}}.
\end{equation*}
\end{lemma}

Since the weights should be the higher rank analogue of the formulas obtained in \cite{ML19} (up to some reparametrizations), we do not have to solve the recursion in Lemma \ref{lem: bd wt rec}, but are able to check that the higher rank analogue of the formulas obtained in \cite{ML19} satisfy the recursion. However, the proof is still quite long and tedious. We thus leave the formal proof for Appendix \ref{app: pf}, and note that it proceeds as expected, by directly checking the $L=1$ case, and then plugging in the formula for $W_L$ given by Theorem \ref{thm: fused bd wt} into the recurrence given by Lemma \ref{lem: bd wt rec}, and checking that it holds. Here, we will simply give the explicit formula for $W_{L,1}$, which in principle is enough to complete the proof. The following proposition is given by Theorem 8.2 in \cite{BGW22}. Note that due to differences in convention (both inversion of the spectral parameter and a difference in rapidities), we have the substitution $z=1/zq^{L-1}$.
\begin{proposition}
\label{prop: row vt}
Consider a bulk vertex with $L$ columns and $1$ row. We use $b$ and $d$ to represent the colors entering and exiting horizontally rather than the color vectors $\Bb$ and $\Db$. Then subject to $\Ab+\mathbf{e}_b=\Cb+\mathbf{e}_d$, we have
\begin{equation*}
    W_{L,1}(z,q;\Ab,b,\Cb,d)=(zq^{L-1})^{-I(b>d)}\frac{1-(zq^{L-1})^{-I(b=d)}q^{\Ab_d}}{1-z^{-1}q}q^{|\Ab^{\geq d+1}|},
\end{equation*}
where for a statement $P$, $I(P)=1$ if $P$ is true and $0$ otherwise.
\end{proposition}

\begin{remark}
Fusion for models with a boundary has been previously considered, see for example \cite{MN92, FNR07, KS92}. It has been understood that solutions to the reflection equation have an algebraic origin, and one source are coideal subalgebras of quantum groups (these are the quantum group analogue of symmetric spaces in the classical setting), similar to how solutions to the Yang--Baxter equation originate from $R$-matrices associated to quantum groups, see e.g. \cite{KOY19a, KOY19b, BK19, KS92, CM20}. However, it is unclear to the author how the negative colors that have been introduced fit into this picture. It would be very interesting to see if the models and vertex weights studied in this section could be explained algebraically.
\end{remark}

\subsection{Color projection}
An important property of the fused vertex model is that adjacent colors can be combined to obtain a new version of the fused vertex model with one fewer color, like for the fused full space model. A new feature is that the colors $-1$ and $1$ can be simultaneously combined with $0$.
\begin{proposition}
\label{prop: col proj}
Consider an $(M,L)$-fused bulk vertex or an $(L,L)$-fused boundary vertex. The vertex weight formula remains valid if we take colors $i$ and $i+1$ (and also $-i$ and $-i-1$), and replace them with the single color $i$ (and $-i$). The same is true if we replace $-1$ and $1$ with $0$. More formally, given incoming color vectors $\Ab$ and $\Bb$, let $\Ab'$ and $\Bb'$ denote the color vectors obtained by taking $\Ab_j'=\Ab_j$ if $j\neq i,i+1$, and $\Ab_i'=\Ab_i+\Ab_{i+1}$ and $\Ab_{i+1}'=0$, and similarly for the negative colors and $\Bb'$. For any outgoing color vectors $\Cb'$ and $\Db'$ with $\Cb'_{i+1}=\Db'_{i+1}=0$, we have
\begin{equation*}
    \sum_{\substack{\Cb_i+\Cb_{i+1}=\Cb_i'\\\Db_i+\Db_{i+1}=\Db_i'}}W_{L,M}(z,q;\Ab,\Bb,\Cb,\Db)=W_{L,M}(z,q;\Ab',\Bb',\Cb',\Db'),
\end{equation*}
and
\begin{equation*}
    \sum_{\substack{\Cb_i+\Cb_{i+1}=\Cb_i'\\\Db_i+\Db_{i+1}=\Db_i'}}W_{L}(z,q,\nu;\Ab,\Bb,\Cb,\Db)=W_{L}(z,q,\nu;\Ab',\Bb',\Cb',\Db').
\end{equation*}
If we let $\Ab'_0=\Ab'_1+\Ab'_{-1}+\Ab_0$, and similarly for $\Bb'$, and take $\Cb'$, $\Db'$ such that $\Cb'_1=\Cb'_{-1}=\Db'_1=\Db'_{-1}=0$, then
\begin{equation*}
    \sum_{\substack{\Cb_1+\Cb_{-1}+\Cb_0=\Cb_0'\\\Db_1+\Db_{-1}+\Db_0=\Db_0'}}W_{L,M}(z,q;\Ab,\Bb,\Cb,\Db)=W_{L,M}(z,q;\Ab',\Bb',\Cb',\Db'),
\end{equation*}
and
\begin{equation*}
    \sum_{\substack{\Cb_1+\Cb_{-1}+\Cb_0=\Cb_0'\\\Db_1+\Db_{-1}+\Db_0=\Db_0'}}W_{L}(z,q,\nu;\Ab,\Bb,\Cb,\Db)=W_{L}(z,q,\nu;\Ab',\Bb',\Cb',\Db').
\end{equation*}
\end{proposition}
\begin{proof}
We first note that this property is true for a single (unfused) vertex. For bulk vertices, this follows from the property for the full space model, see \cite[Remark 6.1]{BGW22} for example. To see this at a boundary vertex, simply note that there is no effect if $i$ and $i+1$ are combined, since if $j$ enters a boundary vertex then so must $-j$, so $i$ and $\pm (i+1)$ can never enter together. If we combine $-1$, $0$ and $1$, then a boundary vertex with $1$ and $-1$ entering would have both outgoing states summed over giving a weight of $1$.

Next, we note that the process of combining colors commutes with fusion. Thus, the fact that the proposition holds for a single vertex immediately implies that the fused vertices also satisfy the proposition.
\end{proof}
\begin{remark}
Proposition \ref{prop: col proj} could also be proved by explicitly computing the sums using $q$-binomial identities.

Once we have combined adjacent colors for the entire vertex model, there is no harm in simply removing the missing color and reindexing the remaining colors, and we will do so implicitly whenever convenient.
\end{remark}

\subsection{Shift invariance}
We now state and prove a version of shift invariance for this fused model. We will let $\mathbf{L}=(L_1,\dotsc)$ denote the number of rows/columns fused, and write $\Ht^{\geq i}(x,\mathbf{L})$ to denote the height function associated to the fused stochastic six-vertex model, with $L_i$ incoming arrows of color $i$ in the $i$th (fused) row.

\begin{theorem}
\label{thm: shift inv for fused}
Let $i_k$, $i'_l$, $(x_k,y_k)$, and $(x_l',y_l')$ be two collections of color cutoffs and positions for height functions such that all height functions cross the diagonal, are ordered, $i_k\leq i_l'$ for all $k,l$, and $(x_k,y_k)$ is weakly northwest of $(x_l',y_l'+1)$ for all $k,l$. Then we have a joint distributional equality
\begin{equation*}
    \big\{\Ht^{\geq i_k}(x_k,y_k;\mathbf{L})\big\}_{k}\cup \big\{\Ht^{\geq i'_l}(x'_l,y'_l;\mathbf{L})\big\}_{l}\deq \big\{\Ht^{\geq i_k}(x_k;\mathbf{L}')\big\}_{k}\cup \big\{\Ht^{\geq i'_l+1}(x'_l,y'_l+1;\mathbf{L}')\big\}_{l},
\end{equation*}
where $\mathbf{L}'=(L_1',\dotsc)$ is obtained from $\mathbf{L}$ by setting $L_i'=L_i$ for $i< \min_l i_l'$ or $i>\max_l y_l'+1$, $L_{i+1}'=L_{i}$ if $\min_l i_l' \leq i\leq \max_l y_l'$, and $L_{\min_l i_l'}=L_{\max_l y_l'+1}$ (in other words, we move $L_{\max_l y_l'+1}$ below $L_{\min_l i_l'}$ and shift what is in between up).
\end{theorem}
\begin{proof}
We consider the cuts $C_k=(\overline{i_k-1}+1,0,\overline{x_k^{(2)}},\overline{x_k^{(1)}})$ and similarly for $C_l'$ in terms of the $j_l$ and $y_l$, where $\overline{x}=\sum_{k=1}^x L_k$. We have $\Ht^{\geq i_k}(x_k;\mathbf{L})=\Ht_A(C_k)$, where $A$ is the symmetric skew domain generated by all these cuts and the point $(1,1)$. The theorem then follows immediately from Theorem \ref{thm: shift inv} applied $L_{\max y_l^{(2)}+1}$ many times.
\end{proof}

\section{Half space higher spin six-vertex model}
\label{sec: higher spin}
In this section, we describe how to take the fused model from Section \ref{sec: fusion}, and analytically continue it after degenerating certain parameters to obtain a higher-spin version of the half-space colored stochastic six-vertex model, with no bounds on the number of arrows on an edge. We then take a limit to obtain a model with a continuous number of arrows. Eventually, we will identify this continuous vertex model with the beta polymer of Section \ref{sec: intro beta}. This basic strategy is the same as in the full space setting, see \cite{BGW22} for example, but additional technical difficulties arise.

\subsection{Degeneration and analytic continuation}
We will reindex the rows to start at $0$ rather than $1$. We now set $z_x=q^{-L_x+1/2}$ for all $x\geq 1$ and $z_0=q^{1/2}z_0$ (we will see why $z_0$ needs to be treated specially). The reason for this specialization is that the bulk weights \eqref{eq: fused wt bulk} degenerate, since $\Phi_q(\mathbf{P},\Bb;1,z)=\delta_{\mathbf{P}=0}$ (and note that $|\mathbf{P}|$ can be less than $L$ or $M$). Our bulk and boundary weights now simplify to
\begin{equation}
\label{eq: fused wt bulk degen}
    W_{M,L}(q^{-L-M+1},q;\Ab,\Bb,\Cb,\Db)=\Phi_q(\Cb,\Ab+\Bb;q^{-M},q^{-M-L})
\end{equation}
(where recall that we're still keeping track of $0$ as a color) and
\begin{equation}
\label{eq: fused wt bd degen}
    W_{L}(q^{-2L+1},\nu;\Ab,\Bb,\Cb,\Db)=\nu^{|\Bb^+-\Cb^+|}q^{-L|\Cb^+|} \Phi_q(\Cb^+,\Bb^+;\nu,\nu q^{-L}),
\end{equation}
where we've absorbed a $-q^{1/2}$ factor into $\nu$ for convenience. Here, we're implicitly assuming that $|\Ab|=|\Cb|=M$ and $|\Bb|=|\Db|=L$, and the constraint $\Ab+\Bb=\Cb+\Db$ holds.

Our first step is to remove $0$ as a color. For the boundary weights, there is nothing to do, since $\Bb_0$ and $\Cb_0$ do not appear in \eqref{eq: fused wt bd degen}. The following lemma gives the weight for a bulk vertex.

\begin{lemma}
We have
\begin{equation}
\label{eq: bulk wt no 0}
\begin{split}
    &W_{M,L}(q^{-L-M+1},q;\Ab,\Bb,\Cb,\Db)
    \\=&q^{-|\Cb^-|L-|\Db^+|M-|\Cb^-||\Db^+|}\frac{(q^{-M};q)_{|\Cb^-+\Cb^+|}(q^{-L};q)_{|\Db^-+\Db^+|}}{(q^{-L-M};q)_{|\Cb^-+\Cb^++\Db^-+\Db^+|}}
    \\&\qquad \qquad\times q^{\sum _{i<j<0}\Db_i \Cb_j+\sum _{0<i<j}\Db_i\Cb_j}\prod _{i\neq 0}{\Cb_i+\Db_i\choose \Cb_i}_q.
\end{split}
\end{equation}
\end{lemma}
\begin{proof}
We start with \eqref{eq: fused wt bulk degen}, and substitute $\Cb_0=M-|\Cb^-+\Cb^+|$ and $\Db_0=L-|\Db^-+\Db^+|$. Since $|\Cb|=M$ and $|\Db|=L$, we have
\begin{equation*}
    \frac{(q^{-M};q)_{|\Cb|}(q^{-L};q)_{|\Db|}}{(q^{-L-M};q)_{|\Cb+\Db|}}{\Cb_0+\Db_0\choose \Cb_0}=q^{\Cb_0\Db_0}\frac{(q^{-M};q)_{|\Cb^-+\Cb^+|}(q^{-L};q)_{|\Db^-+\Db^+|}}{(q^{-L-M};q)_{|\Cb^-+\Cb^++\Db^-+\Db^+|}}.
\end{equation*}
We then take the factors in $q^{\sum_{i<j}\Db_i\Cb_j}$ with either $i=0$ or $j=0$, as well as the factors with $i<0<j$, and collect them with the other powers of $q$. This gives the desired formula.
\end{proof}
From now on, color vectors will not track the number of arrows of color $0$. We now wish to remove the restrictions on the number of arrows in the rows and columns by analytically continuing the model in $q^{-L_i}$. However, the first column and row of the model will be treated differently, because otherwise an infinite number of arrows would be entering our model. For convenience, we will refer to the bulk vertices in the first column and row as \emph{incoming vertices}.

The first step is to set the color of the arrows entering in the $0$th row and column to be $0$. Once this is done, we see that there is no issue in analytically continuing $q^{-L_0}=\mathfrak{m}$ to a complex parameter. We now wish to do the same for $q^{-L_i}$, but some care is needed in the remaining vertices in the $0$th column and row, as a potentially infinite number of arrows would be entering. The treatment is identical to the full space setting, but for convenience we state the needed result. The following follows from Lemma 6.8 of \cite{BGW22} after forgetting the negative colors (which is fine since negative colors will never enter the first column) and setting $z=z_0q^{-L}$.

\begin{lemma}[{\hspace{1sp}\cite[Lemma 6.8]{BGW22}}]
Take the weight for a bulk vertex \eqref{eq: fused wt bulk}, and assume that $\Ab$ has no negative colors and no colors greater than $i$, and that $\Bb_i=L$, which means $\Bb_j=0$ for $j\neq i$. The for $q\in (0,1)$, we have
\begin{equation*}
    \lim_{\mathfrak{m}\to 0}W_{\infty, L}(q^{-L+1}z_0,q;\Ab,\Bb,\Cb,\Db)=\begin{cases}\frac{(z_0;q)_\infty(q^{-L};q)_d}{(z_0q^{-L};q)_{\infty}(q;q)_d}z_0^d&\text{if }\Db=d\mathbf{e}_i, d\geq 0,
    \\0&\text{else},
    \end{cases}
\end{equation*}
where
\begin{equation*}
    W_{\infty,L}(z,q;\Ab,\Bb,\Cb,\Db)
\end{equation*}
is defined by taking \eqref{eq: fused wt bulk}, and replacing $q^{-M}$ with $\mathfrak{m}$.
\end{lemma}

Next, we notice that formulas \eqref{eq: fused wt bulk degen} and \eqref{eq: fused wt bd degen} are independent of the number of arrows of color $0$, and are meromorphic functions in $q^{-M}$ and $q^{-L}$. Thus, after forgetting about the arrows of color $0$, we can analytically continue to parameters $\mathfrak{m}=q^{-M}$ and $\mathfrak{l}=q^{-L}$, and remove the restriction on the number of arrows in each row and column.

We will thus let $\mathfrak{l}_x=q^{-L_x}$ for $x>0$. These will act as the new parameters for our fully fused vertex model. We will not track arrows in the first column and row, and instead only keep track of the number entering the rest of the system from them. We also reindex the rows/columns to remove the first one. We summarize the fully fused and analytically continued vertex model:
\begin{itemize}
    \item We start with parameters $\mathfrak{l}_x$ for $x\in\N$ attached to rows/columns, $\nu$ corresponding to the diagonal, and $z_0$ for the first row/column.
    \item The distribution for the incoming arrows at row $x$ is given by
    \begin{equation*}
        \PP(\Bb(1,x)=d\mathbf{e}_x)=\frac{(z_0;q)_\infty(\mathfrak{l}_x;q)_d}{(z_0\mathfrak{l}_x;q)_{\infty}(q;q)_d}z_0^d
    \end{equation*}
    and $0$ otherwise. The same number of arrows of color $-x$ enter at column $x$.
    \item Given $\Ab(x,y)$ and $\Bb(x,y)$, we sample $\Cb(x,y)$ and $\Db(x,y)$ by either
    \begin{equation}
\label{eq: bulk wt analytic}
\begin{split}
    &W_{\mathfrak{l}_x,\mathfrak{l}_y}(q;\Ab,\Bb,\Cb,\Db)
    \\=&\mathfrak{l}^{|\Db^+|}_x\mathfrak{l}^{|\Cb^-|}_yq^{-|\Cb^-||\Db^+|}\frac{(\mathfrak{l}_x;q)_{|\Cb^-+\Cb^+|}(\mathfrak{l}_y;q)_{|\Db^-+\Db^+|}}{(\mathfrak{l}_x\mathfrak{l}_y;q)_{|\Cb^-+\Cb^++\Db^-+\Db^+|}}
    \\&\qquad\qquad\times q^{\sum _{i<j<0}\Db_i \Cb_j+\sum _{0<i<j}\Db_i\Cb_j}\prod _{i\neq 0}{\Cb_i+\Db_i\choose \Cb_i}_q,
\end{split}
\end{equation}
    if $x<y$, or
    \begin{equation*}
        W_{\mathfrak{l}_x}(q,\nu;\Ab,\Bb,\Cb,\Db)=\nu^{|\Bb^+-\Cb^+|}\mathfrak{l}_x^{|\Cb^+|} \Phi_q(\Cb^+,\Bb^+;\nu,\nu \mathfrak{l}_x),
    \end{equation*}
    if $x=y$. Arrows travel through vertices $(x,y)$ with $x>y$ by symmetry. That is, $\Cb_i(x,y)=\Db_{-i}(y,x)$ for all $(x,y)$ and $i$ (this is automatic if $x=y$ by definition of the boundary weight).
\end{itemize}

We now state and prove a version of shift invariance for this analytically continued model.

\begin{theorem}
\label{thm: shift inv analytic}
Let $i_k$, $i'_l$, $(x_k,y_k)$, and $(x_l',y_l')$ be two collections of color cutoffs and positions for height functions such that all height functions cross the diagonal, are ordered, $i_k\leq i_l'$ for all $k,l$, and $(x_k,y_k)$ is weakly northwest of $(x_l',y_l'+1)$ for all $k,l$. Then we have a joint distributional equality
\begin{equation*}
    \big\{\Ht^{\geq i_k}(x_k,y_k;\mathfrak{l})\big\}_{k}\cup \big\{\Ht^{\geq i'_l}(x'_l,y'_l;\mathfrak{l})\big\}_{l}\deq \big\{\Ht^{\geq i_k}(x_k;\mathfrak{l}')\big\}_{k}\cup \big\{\Ht^{\geq i'_l+1}(x'_l,y'_l+1;\mathfrak{l}')\big\}_{l},
\end{equation*}
where $\mathfrak{l}'=(\mathfrak{l}_1',\dotsc)$ is obtained from $\mathbf{z}$ by setting $\mathfrak{l}_i'=\mathfrak{l}_i$ for $i< \min_l i_l'$ or $i>\max_l y_l'+1$, $\mathfrak{l}_{i+1}'=\mathfrak{l}_{i}$ if $\min_l i_l' \leq i\leq \max_l y_l'$, and $\mathfrak{l}_{\min_l i_l'}=\mathfrak{l}_{\max_l y_l'+1}$ (in other words, we move $\mathfrak{l}_{\max_l y_l'+1}$ below $\mathfrak{l}_{\min_l i_l'}$ and shift what is in between up).
\end{theorem}
\begin{proof}
We note that even if the $\mathfrak{l}_x$ are allowed to take complex values, the distribution of the height functions (no longer required to be non-negative real numbers), viewed as functions of the $\mathfrak{l}_x$, are actually holomorphic in a neighbourhood of $0$. Indeed, only finitely many arrows of a given color ever enter the system, and any given vertex can only have finitely many colors entering it (since only colors of value at or below it can reach it), so the formulas for the vertex weights are holomorphic functions in a neighbourhood of $0$. Moreover, the distribution of the height functions is a finite sum over these vertex weights, and so is also holomorphic. 

Theorem \ref{thm: shift inv analytic} states that these distribution functions are equal when $\mathfrak{l}_x=q^{-L_x}$ for positive integers $L_x$. Since $0$ is an accumulation point for this sequence, and the functions are holomorphic, this means that they must be equal as holomorphic functions (here, it is helpful to note that the distribution of any finite collection of height functions is only a function of finitely many of the $\mathfrak{l}_x$). Thus, the stated identity continues to hold for complex parameters $\mathfrak{l}_x$.
\end{proof}

\subsection{Continuous vertex model}
In this section, we wish to scale the model and take a limit as $q\to 1$. We thus define parameters $\sigma_i,\omega,\rho$, let $\varepsilon>0$, and define
\begin{equation}
\label{eq: eps assumptions}
q=\exp(-\varepsilon),\qquad \mathfrak{l}_i=q^{\sigma_i},\qquad \nu=q^{\omega},\qquad z_0=q^{\rho}.
\end{equation}

We begin with the simplest vertices to analyze, which are the incoming vertices. The following proposition follows immediately from Proposition 6.16 in \cite{BGW22} after a change of variables.
\begin{proposition}
\label{prop: incoming limit}
As $\varepsilon\to 0$, $\exp(-\varepsilon\Bb_x(1,x))$ (where recall that $\Bb_x(1,x)$ is the number of incoming arrows at row $x$) converges weakly to a $\Beta(\rho,\sigma_x)$ random variable.
\end{proposition}

Next, we examine the bulk vertices. The analysis is complicated by the fact that we have colors smaller than $0$, which means we cannot directly apply the known results for the full space model, although ultimately we do reduce to it. We first need the following lemma, which gives a two step process to sample a bulk vertex.

\begin{lemma}
At a bulk vertex, given the incoming color vectors $\Ab$ and $\Bb$, the outgoing color vectors $\Cb$ and $\Db$ can be sampled by first sampling $|\Cb^\pm|$ and $|\Db^\pm|$ according to the distribution
\begin{equation}
    \label{eq: sample step 1}
    \mathfrak{l}^{|\Db^+|}_x\mathfrak{l}^{|\Cb^-|}_yq^{-|\Cb^-||\Db^+|}\frac{(\mathfrak{l}_x;q)_{|\Cb^-+\Cb^+|}(\mathfrak{l}_y;q)_{|\Db^-+\Db^+|}}{(\mathfrak{l}_x\mathfrak{l}_y;q)_{|\Cb^-+\Cb^++\Db^-+\Db^+|}}{|\Cb^-+\Db^-|\choose |\Cb^-|}_q{|\Cb^++\Db^+|\choose |\Cb^+|}_q,
\end{equation}
and then given $|\Cb^\pm|$ and $|\Db^\pm|$, sampling $\Cb$ and $\Db$ according to the distribution
\begin{equation}
    \label{eq: sample step 2}
    {|\Cb^-+\Db^-|\choose |\Cb^-|}^{-1}_q{|\Cb^++\Db^+|\choose |\Cb^+|}^{-1}_q q^{\sum _{i<j<0}\Db_i \Cb_j+\sum _{0<i<j}\Db_i\Cb_j}\prod _{i\neq 0}{\Cb_i+\Db_i\choose \Cb_i}_q.
\end{equation}
In particular, after conditioning on $|\Cb^\pm|$ and $|\Db^\pm|$, $\Cb^+$ and $\Db^+$ are independent of $\Cb^-$ and $\Db^-$.
\end{lemma}
\begin{proof}
The proof is extremely similar to the proof of Lemma 6.17 of \cite{BGW22}, the analogous result in the full space setting. In particular, using the $q$-Chu--Vandermonde identity
\begin{equation*}
    {m+n\choose k}_q=\sum_{j}{m\choose k-j}_q{n\choose j}_qq^{j(m-k+j)}
\end{equation*}
to sum \eqref{eq: bulk wt analytic} over $\Cb_i$ with $|\Cb^\pm|$ (and thus also $|\Db^\pm|$) fixed results in \eqref{eq: sample step 1}, and dividing \eqref{eq: bulk wt analytic} by \eqref{eq: sample step 1} gives \eqref{eq: sample step 2}.
\end{proof}

Next, we prove a technical result needed for the asymptotic analysis. We let $\Li(x)=\sum x^k/k^2$ denote the dilogarithm function.
\begin{lemma}
\label{lem: H lemma}
Let $0<a,b<1$ and consider the function $H_{a,b}(x,y)=H(x,y)$ defined by the expression
\begin{equation}
\label{eq: H}
\begin{split}
    &\Li(x)+\Li(a/x)+\Li(y)+\Li(b/y)+\Li(ab)
    \\&\qquad-\Li(a)-\Li(b)-\Li(xy)-\Li(ab/xy)-\Li(1)-\ln(x)\ln(b/y)
\end{split}
\end{equation}
with $a\leq x\leq 1$ and $b\leq y\leq 1$. Then with the change of variables,
\begin{equation*}
\begin{split}
    x&=(1-\eta')(1-a)+a
    \\y&=\frac{1}{\eta(1-b)+b},
\end{split}
\end{equation*}
$H(\eta',\eta)$ on the domain $[0,1]\times [0,1]$ attains its minimum values on a subset of the diagonal $\eta'=\eta$ together with the boundary of $[0,1]\times [0,1]$.
\end{lemma}
\begin{proof}
It suffice to show that the critical points of $H(x,y)$ lie on the curve parametrized by 
\begin{equation*}
\begin{split}
    x&=(1-\eta)(1-a)+a
    \\y&=\frac{b}{\eta(1-b)+b}
\end{split}
\end{equation*}
for $\eta\in [0,1]$. We can compute
\begin{equation*}
    \frac{d}{dx}H(x,y)=\frac{-\ln\left(1-\frac{ab}{xy}\right)+\ln\left(1-\frac{a}{x}\right)+\ln\left(\frac{b}{y}\right)+\ln(1-xy)-\ln(1-x)}{x}
\end{equation*}
and setting to $0$, we obtain the equation
\begin{equation*}
    b(x-a)(1-xy)=(xy-ab)(1-x).
\end{equation*}
Solving for $y$ gives
\begin{equation*}
    y=\frac{ab-b}{ab-bx+x-1},
\end{equation*}
and we can check that the above expressions for $x$ and $y$ in terms of $\eta$ parametrizes the solutions for $a< x< 1$ and $b< y< 1$.

\end{proof}

\begin{proposition}
\label{prop: bulk total limit}
Consider a bulk vertex with spectral parameters $q^{\sigma}$ and $q^{\sigma'}$. Suppose that as $\varepsilon\to 0$, we have $\varepsilon|\Ab^\pm|$ and $\varepsilon|\Bb^\pm|$ stay within some compact subinterval of $(0,\infty)$, and we have $\lim_{\varepsilon\to 0} \varepsilon|\Ab^\pm|=|\alpha^\pm|$ and $\lim_{\varepsilon\to 0}\varepsilon|\Bb^\pm|=|\beta^\pm|$ (and $|\alpha|=|\alpha^+|+|\alpha^-|$ and so on). Then $(\varepsilon|\Cb^\pm|,\varepsilon|\Db^\pm|)$ converge weakly to continuous random variables $(|\gamma^\pm|,|\delta^\pm|)$ such that
\begin{equation}
\label{eq: degen condition}
\begin{split}
    \exp(-|\delta^+|)&=\exp(-|\alpha^+|-|\beta^+|)+(1-\exp(-|\alpha^+|-|\beta^+|))\eta,
    \\\exp(-|\gamma^-|)&=\exp(-|\alpha^-|-|\beta^-|)+(1-\exp(-|\alpha^-|-|\beta^-|))(1-\eta),
\end{split}
\end{equation}
where $\eta\sim \Beta(\sigma,\sigma')$, and
\begin{equation*}
\begin{split}
    |\alpha^-|+|\beta^-|&=|\gamma^-|+|\delta^-|,
    \\|\alpha^+|+|\beta^+|&=|\gamma^+|+|\delta^+|,
\end{split}
\end{equation*}
\end{proposition}
\begin{proof}
We let $|\gamma^\pm|=\varepsilon|\Cb^\pm|$ and $|\delta^\pm|=\varepsilon|\Db^\pm|$. It suffices to show convergence for $\varepsilon|\Cb^-|$ and $\varepsilon|\Db^+|$. Taking the formula \eqref{eq: sample step 1} for the weight of the analytically continued vertex model, we can rewrite as
\begin{equation}
\label{eq: cont weight poch form}
\begin{split}
    &e^{-\sigma|\delta^+|-\sigma'|\gamma^-|-|\gamma^-||\delta^+|/\varepsilon}\frac{(q^{\sigma};q)_\infty }{(q;q)_\infty}\frac{(q^{\sigma'};q)_\infty }{(q;q)_\infty}\frac{(q;q)_\infty}{(q^{\sigma+\sigma'};q)_\infty}
    \\&\qquad\times \frac{(qe^{-|\gamma^-|};q)_\infty (qe^{-|\delta^-|};q)_\infty (qe^{-|\gamma^+|};q)_\infty (qe^{-|\delta^+|};q)_\infty}{(q;q)_\infty (qe^{-|\gamma^-|-|\delta^-|};q)_\infty (qe^{-|\gamma^+|-|\delta^+|};q)_\infty}
    \frac{(q^{\sigma+\sigma'}e^{-|\gamma|-|\delta|};q)_\infty }{(q^{\sigma} e^{-|\gamma|};q)_\infty (q^{\sigma'} e^{-|\delta|})_\infty}.
\end{split}
\end{equation}

We first show that asymptotically, this distribution concentrates on the curve defined by \eqref{eq: degen condition}. This will then let us study only the random variable $\varepsilon |\Cb^-|$, which reduces the problem to that of the full space setting.

We first note that the first three fractions in \eqref{eq: cont weight poch form} are asymptotically
\begin{equation*}
    \varepsilon \frac{\Gamma(\sigma+\sigma')}{\Gamma(\sigma)\Gamma(\sigma')},
\end{equation*}
and in particular is asymptotic to $\varepsilon$. The factor $e^{-\sigma|\delta^+|-\sigma'|\gamma^-|}$ is also bounded, since $|\delta^+|+|\gamma^-|\leq |\alpha|+|\beta|$.

Using the approximation
\begin{equation*}
    (a;q)_\infty=\exp\left(\frac{1}{\ln q}\Li(a)+O(1)\right),
\end{equation*}
which holds uniformly as $q\to 1$ for $a$ bounded away from $1$ (see e.g. \cite[Corollary 10]{K95}), we obtain that the remaining factors in \eqref{eq: cont weight poch form} are asymptotic to
\begin{equation*}
    \exp\left(-\frac{1}{\varepsilon}H_{a,b}(\exp(-|\gamma^-|),\exp(-|\gamma^+|))\right),
\end{equation*}
with $a=\exp(-|\alpha^-|-|\beta^-|)$ and $b=\exp(-|\alpha^+|-|\beta^+|)$ (and $H$ the function in Lemma \ref{lem: H lemma}). Note that the extra factors of $q$ that show up do not matter since $\Li(q^ca)=\Li(a)+O(\varepsilon)$. Here, we must be a bit careful and assume that $|\gamma^\pm|$ and $|\delta^\pm|$ are bounded away from $0$. 

With the change of variables
\begin{equation}
\label{eq: degen condition 2}
\begin{split}
    \exp(-|\delta^+|)&=\exp(-|\alpha^+|-|\beta^+|)+(1-\exp(-|\alpha^+|-|\beta^+|))\eta,
    \\\exp(-|\gamma^-|)&=\exp(-|\alpha^-|-|\beta^-|)+(1-\exp(-|\alpha^-|-|\beta^-|))(1-\eta'),
\end{split}
\end{equation}
we see that the distribution concentrates on the event that at least one of $\eta=\eta'$, $\eta=0$, $\eta=1$, $\eta'=0$ or $\eta'=1$ occurs (which accounts for the possiblity that $|\gamma^\pm|$ or $|\delta^\pm|$ are close to $0$). It now suffices to show that $|\delta^+|$ and $|\gamma^-|$ separately have limiting distributions given by \eqref{eq: degen condition}, as this implies in particular that $\eta, \eta'=0,1$ occur with probability $0$ in the limit, and so asymptotically $\eta=\eta'$.

By forgetting the negative colors, we obtain the usual full space analytically continued vertex model. The convergence of $\varepsilon|\Db^+|$ to $|\delta^+|$ with distribution given by \eqref{eq: degen condition} is then given by Proposition 6.18 in \cite{BGW22} (up to a change of variables, $|\Db|=|\Db^+|$, $|\Ab|=|\Ab^+|+|\Bb^+|$, $\mathfrak{l}=q^{\sigma'}$, and $\mathfrak{m}_x=q^{\sigma+\sigma'}$). A similar argument shows the convergence of $|\gamma^-|$, except we forget the positive colors, and use the symmetry of the model about the diagonal. Then in the limit, $\eta=\eta'$, and so joint convergence holds.
\end{proof}

\begin{proposition}
\label{prop: bulk weight limit}
The bulk vertex weight given by \eqref{eq: bulk wt analytic} converges as $\varepsilon\to 0$ to the following sampling procedure. 

Given incoming color masses $\alpha$ and $\beta$, we let $\eta\sim \Beta(\sigma,\sigma')$, and inductively define for $i\geq 1$,
\begin{equation*}
\begin{split}
    \exp(-|\delta^{\geq i}|)&=\exp(-|\alpha^{\geq i}|-|\beta^{\geq i}|)+(1-\exp(-|\alpha^{\geq i}|-|\beta^{\geq i}|))\eta,
    \\\exp(-|\gamma^{\leq -i}|)&=\exp(-|\alpha^{\leq -i}|-|\beta^{\leq -i}|)+(1-\exp(-|\alpha^{\leq -i}|-|\beta^{\leq -i}|))(1-\eta),
\end{split}
\end{equation*}
and take
\begin{equation*}
\begin{split}
    \delta_i&=|\delta^{\geq i}|-|\delta^{\geq i+1}|,
    \\\gamma_{-i}&=|\gamma^{\leq -i}|-|\gamma^{\leq -i-1}|,
    \\\gamma_i&=\alpha_i+\beta_i-\delta_i,
    \\\delta_{-i}&=\alpha_{-i}+\beta_{-i}-\gamma_{-i}.
\end{split}
\end{equation*}
\end{proposition}
\begin{proof}
Note that given $|\Cb^\pm|$ and $|\Db^\pm|$, the distribution of $\Cb_i$ for $i<0$ and $\Db_i$ for $i>0$ are independent. Also, the formula \eqref{eq: sample step 2} for sampling $\Cb_i$ for $i<0$ given $|\Cb^-|$ is identical to those in the full space setting, and similarly for $\Db_i$ for $i>0$ given $|\Db^+|$. Thus, the limiting distribution of $\varepsilon \Db_i$ for $i>0$ follows from the analysis of the full space case (up to a change of variables), see Proposition 6.19 and Corollary 6.21 of \cite{BGW22}, and the limiting distribution of the $\varepsilon\Cb_i$ for $i<0$ does as well using the diagonal reflection symmetry of the model.
\end{proof}

The analysis of boundary vertices is much simpler, since $\Ab_i=\Bb_{-i}$ and $\Cb_i=\Db_{-i}$.
\begin{proposition}
\label{prop: bd limit}
Consider a boundary vertex with spectral parameter $q^{\sigma}$ and $\nu=q^{\omega}$. Suppose that as $\varepsilon\to 0$, we have $\varepsilon \Ab^+\to \alpha^+$ and $\varepsilon\Bb^+\to\beta^+$, staying in some compact subinterval of $(0,\infty)$. Then $(\varepsilon\Cb^\pm,\varepsilon\Db^\pm)$ converge weakly to continuous random variables $(\gamma^\pm,\delta^\pm)$ such that
\begin{equation}
\label{eq: bd cont limit}
\exp(-(|\delta^{\geq i}|-|\alpha^{\geq i}|))=\exp(-|\beta^{\geq i}|)+(1-\exp(-|\beta^{\geq i}|))\eta,
\end{equation}
where $\eta\sim\Beta(\omega,\sigma)$, and the $\gamma_i$ and $\delta_i$ are defined inductively by
\begin{equation*}
\begin{split}
    \delta_i&=|\delta^{\geq i}|-|\delta^{\geq i+1}|,
    \\\gamma_i&=\alpha_i+\beta_i-\delta_i,
    \\\gamma_{-i}&=\delta_i,
    \\\delta_{-i}&=\gamma_i.
\end{split}
\end{equation*}
\end{proposition}
\begin{proof}
By symmetry, it suffices to ignore the negative arrows, and just look at positive colors. Because the form of the boundary weight for the positive colors is identical to that of the full space vertex weights, up to a change of variables, the limiting distribution for the $\delta_i$ and $\gamma_i$, $i\geq 1$, follows from Proposition 6.18, Proposition 6.19, and Corollary 6.21 of \cite{BGW22}, with $\Ab=\Bb^+$, $\Db=\Bb^+-\Cb^+$, $\mathfrak{l}=q^{\sigma}$, $\mathfrak{m}_x=q^{\omega}$, $\alpha=\beta^+$ and $\delta=\beta^+-\gamma^+$. By symmetry, we then also obtain the distributions for $\gamma_{-i}$ and $\delta_{-i}$.
\end{proof}

We remark that at the level of the continuous vertex model, there is no correlation between the positive and negative colors given the beta random variables $\eta$. Thus, we will allow ourselves to forget the negative arrows, and instead keep track of the positive arrows below the diagonal. This lets us view the model as a continuous vertex model with symmetric randomness. Let us summarize the continuous vertex model. 
\begin{proposition}
\label{prop: cont model description}
The continuous half-space vertex model is equivalent to the following random process.
\begin{itemize}
    \item We start with positive parameters $\sigma_x$ for $x\in \N$ attached to the rows/columns, $\omega$ corresponding to the diagonal, and $\rho$ corresponding to the incoming vertices.
    \item We sample beta random variables $\eta_{x,y}\sim Beta(\sigma_x,\sigma_y)$ for $1\leq x<y$ and set $\eta_{x,y}=1-\eta_{y,x}$ if $x>y$, $\eta_{x,x}\sim (\omega,\sigma_x)$, and $\eta_{0,y}\sim (\rho,\sigma_y)$.
    \item We let $\alpha_i(x,y)$ denote the color mass from the bottom at vertex $(x,y)$, and let $\alpha(x,y)$ denote the color mass vector. We let $\alpha^{\geq i}(x,y)$ denote $\alpha(x,y)$ with colors under $i$ set to $0$. We similarly define analogues for $\beta$, $\gamma$, and $\delta$, corresponding to color masses incoming from the left, and leaving the top and right respectively.
    \item At the left boundary, we set $\beta_y(1,y)=-\ln(\eta_{0,y})$ for $y\geq 1$, and $\beta_i(1,y)=0$ if $i\neq y$. At the bottom boundary, we set $\alpha_i(x,1)=0$ for all $i$ (since we're not tracking negative arrows).
    \item Given the incoming color masses at a bulk vertex $(x,y)$, $x\neq y$, $\alpha_i(x,y)$ and $\beta_i(x,y)$, we then inductively define $\gamma_i(x,y)$ and $\delta_i(x,y)$ by
    \begin{equation*}
        \exp(-|\delta^{\geq i}(x,y)|)=\exp(-|(\alpha+\beta)^{\geq i}(x,y)|)+(1-\exp(-|(\alpha+\beta)^{\geq i}(x,y)|))\eta_{x,y},
    \end{equation*}
    and set $\gamma_i(x,y)=\alpha_i(x,y)+\beta_i(x,y)-\delta_i(x,y)$.
    \item Given incoming color masses at a boundary vertex $(x,x)$, $\alpha_i(x,x)$ and $\beta_i(x,x)$, we then inductively define $\gamma_i(x,x)$ and $\delta_i(x,x)$ by
    \begin{equation*}
        \exp(-(|\delta^{\geq i}(x,x)|-|\alpha^{\geq i}(x,x)|))=\exp(-|\beta^{\geq i}(x,x)|)+(1-\exp(-|\beta^{\geq i}(x,x)|))\eta_{x,x},
    \end{equation*}
    and set $\gamma_i(x,x)=\alpha_i(x,x)+\beta_i(x,x)-\delta_i(x,x)$.
\end{itemize}
\end{proposition}
\begin{proof}
Everything except the behaviour of the positive arrows below the diagonal follows from Propositions \ref{prop: incoming limit}, \ref{prop: bulk weight limit}, and \ref{prop: bd limit}.

The only thing left to check is that the behaviour of the positive arrows below the diagonal is given by the behaviour of the negative arrows above the diagonal, which is how the original model was sampled. We simply note that if $x>y$, $\gamma_i(x,y)=\delta_{-i}(y,x)$ and $\delta_i(x,y)=\gamma_{-i}(y,x)$, and then the claimed sampling procedure for bulk vertices below the diagonal follows from the sampling procedure for negative colors above the diagonal given by Proposition \ref{prop: bulk weight limit}, and $\alpha_i(x,x)=\beta_{-i}(x,x)$ and $\gamma_i(x,x)=\delta_{-i}(x,x)$ which ensures consistency at the boundary.
\end{proof}

We now consider the height functions for the continuous model. We define $\mathcal{H}^{\geq i}(x,y)$ to be the total mass of arrows of color at least $i$ crossing the vertices at or below $(x,y)$ from below. More formally, it is the collection of random variables satisfying
\begin{equation}
\label{eq: rec for cont height}
\begin{split}
    \mathcal{H}^{\geq i}(x,y)&=\mathcal{H}^{\geq i}(x-1,y)-|\gamma^{\geq i}(x-1,y)|,
    \\\mathcal{H}^{\geq i}(x,y)&=\mathcal{H}^{\geq i}(x,y-1)+|\delta^{\geq i}(x,y)|,
\end{split}
\end{equation}
with boundary conditions
\begin{equation*}
    \mathcal{H}^{\geq i}(0,y)=\sum_{y'\leq y} |\beta^{\geq i}(0,y')|.
\end{equation*}
Propositions \ref{prop: incoming limit}, \ref{prop: bulk weight limit}, and \ref{prop: bd limit} imply that
\begin{equation*}
    \mathcal{H}^{\geq i}(x,y)=\lim_{\varepsilon\to 0}\varepsilon\Ht_\varepsilon^{\geq i}(x,y),
\end{equation*}
where $\Ht_\varepsilon$ denotes the height function for the discrete model with the specializations \eqref{eq: eps assumptions}. With this, the following statement of shift invariance follows immediately by taking a limit of Theorem \ref{thm: shift inv analytic}.

\begin{theorem}
\label{thm: shift inv cont}
Let $i_k$, $i'_l$, $(x_k,y_k)$, and $(x_l',y_l')$ be two collections of color cutoffs and positions for height functions such that all height functions cross the diagonal, are ordered, $1\leq i_k\leq i_l'$ for all $k,l$, and $(x_k,y_k)$ is weakly northwest of $(x_l',y_l'+1)$ for all $k,l$. Then we have a joint distributional equality
\begin{equation*}
    \big\{Z^{\geq i_k}(x_k,y_k;\sigma)\big\}_{k}\cup \big\{Z^{\geq i'_l}(x'_l, y'_l;\sigma)\big\}_{l}\deq \big\{Z^{\geq i_k}(x_k,y_k;\sigma')\big\}_{k}\cup \big\{Z^{\geq i'_l+1}(x_l',y_l'+1;\sigma')\big\}_{l},
\end{equation*}
where $\sigma'=(\sigma_1',\dotsc)$ is obtained from $\sigma$ by setting $\sigma_i'=\sigma_i$ for $i< \min_l i_l'$ or $i>\max_l y_l'+1$, $\sigma_{i+1}'=\sigma_{i}$ if $\min_l i_l'\leq i\leq \max_l y_l'$, and $\sigma_{\min_l i_l'}=\sigma_{\max_l y_l'+1}$ (in other words, we move $\sigma_{\max_l y_l'+1}$ below $\sigma_{\min_l i_l'}$ and shift what is in between up).
\end{theorem}

\section{Half space beta polymer}
\label{sec: beta polymer}
In this section, we show that the continuous vertex model described by Proposition \ref{prop: cont model description} is equivalent to a half space version of the beta polymer, or random walk in beta random environment. This model was actually introduced recently by Barraquand and Rychnovsky \cite{BR22}. However, they only considered the model above the diagonal. We show that their model can be extended to below the diagonal, which is necessary for shift invariance to hold. This model has equivalent formulations as either a polymer or a random walk. The original definition in Section \ref{sec: intro beta} is more natural from the polymer point of view, but here we give an alternative formulation which is closer to the random walk model.

\subsection{Alternative model definition}
\begin{figure}
    \centering
    \includegraphics[scale=0.6]{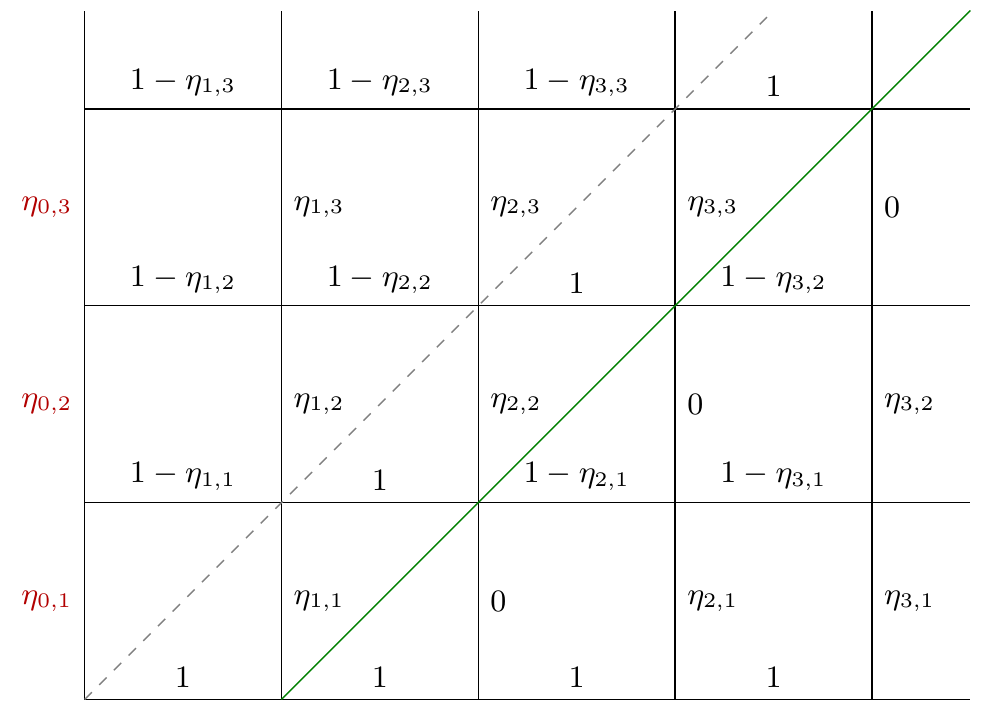}
    \caption{Model on $\HH^2$.}
    \caption{The equivalent formulations of the beta polymer. The model in Figure \ref{fig:beta 2} is obtained from the model in Figure \ref{fig:beta 1} by contracting the edges of weight $1$ between the dashed and green lines, with no effect on the partition functions.}
    \label{fig:beta 1}
\end{figure}
We first give an alternative description of the half space beta polymer that is more closely related to the random walk model. This is equivalent to the original definition in Section \ref{sec: intro beta}, as can be seen by checking the recurrence relations in Lemma \ref{lem: beta rec}. We fix the same parameters $\sigma_x$, $\omega$, and $\rho$ and random variables $\eta_{x,y}$ as in Section \ref{sec: intro beta}. We use a slightly strange coordinate system for $\N_0^2$, but it will become clear why we do so.

If $x\leq y$, then $(x,y)$ will refer to the usual point $(x,y)$ in $\N_0^2$. If $x>y$, then $(x,y)$ will refer to the point $(x+1,y)$. When needed, we will refer to the remaining points, of the form $(x+1,x)$, by $(x+1/2,x)$. To avoid confusion, we will refer to $\N^2$ with these coordinates as $\HH^2$.

To the edges of $\HH^2$, we associate weights as follows:
\begin{itemize}
    \item For all $(x,y)\in \N^2$, we attach a weight of $\eta_{x,y}$ to the edge $(x,y-1)\to (x,y)$, and a weight of $1-\eta_{x,y}$ to the edge $(x-1,y)\to (x,y)$.
    \item We attach a weight of $\eta_{0,y}$ to the edge $(0,y-1)\to (0,y)$.
    \item We attach a weight of $1$ to all edges $(x,0)\to (x+1,0)$.
    \item We attach a weight of $1$ to all edges $(x,x)\to (x+1/2,x)$, and a weight of $0$ to all edges $(x,x-1)\to (x+1/2,x)$.
\end{itemize}
See Figure \ref{fig:beta 1} for a visual representation of this model.

We will be interested in studying the partition function
\begin{equation*}
    Z^{\geq i}(x,y)=\sum_{\gamma:(0,i)\to (x,y)}\wt(\gamma), 
\end{equation*}
where the sum is over all up right paths from $(0,i)$ to $(x,y)$ in $\HH^2$, and $\wt(\gamma)$ is the product of the edge weights over all edges in $\gamma$.

We note the following recurrence for $Z^{\geq i}(x,y)$, which is immediate from the definition.
\begin{lemma}
\label{lem: beta rec}
The partition function $Z^{\geq i}(x,y)$ satisfies
\begin{equation}
\label{eq: beta rec bulk}
    Z^{\geq i}(x,y)=(1-\eta_{x,y})Z^{\geq i}(x-1,y)+\eta_{x,y}Z^{\geq i}(x,y-1)
\end{equation}
if $x\neq y$, and
\begin{equation}
\label{eq: beta rec bd}
    Z^{\geq i}(x,x)=(1-\eta_{x,x})Z^{\geq i}(x-1,x)+\eta_{x,x}Z^{\geq i}(x-1,x-1),
\end{equation}
with boundary conditions
\begin{equation}
\label{eq: beta bc}
Z^{\geq i}(x,0)=1,\qquad Z^{\geq i}(0,y)=\prod_{i\leq u\leq y} \eta_{0,u}.
\end{equation}
\end{lemma}

\begin{remark}
This model is indeed an extension of the one considered in \cite{BR22}. To see this, note that by reversing time (which is the diagonal axis), the partition function can be viewed as the transition probability for a random walk in the random environment defined by the $\eta_{x,y}$. Moreover, the weights below the diagonal do not contribute to $Z(x,y)$ if $x\leq y$, as once a path goes below the diagonal, the edges of weight $0$ prevent it from ever going back above the diagonal. In the random walk formulation, our model is a random walk on the full line, but where the origin is partially reflecting, so walkers on the positive side can never cross, but walkers on the negative side can move to the positive side. It is unclear whether this extension makes the model significantly more interesting, but it is necessary to formulate shift invariance.
\end{remark}

\begin{remark}
\label{rmk: beta other form}
The strange coordinate system for $\HH^2$ could be removed by contracting the edges with weight $1$, which gives our original definition of the model, see Figures \ref{fig:beta 2} and \ref{fig:beta 1}. The alternative description on $\HH^2$ given above seems to be the most natural from the random walk viewpoint, and so we give both descriptions.
\end{remark}

\subsection{Shift invariance for the beta polymer}
In this section, we show that the beta polymer is equivalent to the continuous vertex model, and thus derive shift invariance for the beta polymer.

\begin{lemma}
\label{lem: beta cont same}
Consider the height functions $\mathcal{H}^{\geq i}(x,y)$ for the continuous vertex model as described in Proposition \ref{prop: cont model description}, and the partition functions $Z^{\geq i}(x,y)$ for the beta polymer. Then
\begin{equation*}
    \left(\exp(-\mathcal{H}^{\geq i}(x,y))\right)_{i,x,y}\deq\left(Z^{\geq i}(x,y)\right)_{i,x,y}.
\end{equation*}
\end{lemma}
\begin{proof}
We use the same beta random variables $\eta_{x,y}$ and $\eta_y$ to define both the continuous vertex model and the beta polymer. It then suffices to check that $\exp(-\mathcal{H}^{\geq i}(x,y))$ satisfies the same recurrence and boundary conditions given in Lemma \ref{lem: beta rec}. The boundary conditions \eqref{eq: beta bc} are easy to check, so we will check the recurrence.

We have that at vertex $(x,y)$ for $x\neq y$,
\begin{equation*}
\begin{split}
    |\delta^{\geq i}(x,y)|&=\mathcal{H}^{\geq i}(x,y)-\mathcal{H}^{\geq i}(x,y-1),
    \\|\gamma^{\geq i}(x,y)|&=\mathcal{H}^{\geq i}(x-1,y)-\mathcal{H}^{\geq i}(x,y),
\end{split}
\end{equation*}
and as $|(\gamma^{\geq i}+\delta^{\geq i})(x,y)|=|(\alpha^{\geq i}+\beta^{\geq i})(x,y)|$, this means
\begin{equation*}
    |(\alpha^{\geq i}+\beta^{\geq i})(x,y)|=\mathcal{H}^{\geq i}(x-1,y)-\mathcal{H}^{\geq i}(x,y-1).
\end{equation*}
Thus, by the sampling procedure at $(x,y)$,
\begin{equation*}
\begin{split}
    &\exp(-\mathcal{H}^{\geq i}(x,y))
    \\=&\exp(-\mathcal{H}^{\geq i}(x,y-1)-|\delta^{\geq i}(x,y)|)
    \\=&\exp(-\mathcal{H}^{\geq i}(x,y-1))\left(\exp(-|(\alpha+\beta)^{\geq i}(x,y)|)+(1-\exp(-|(\alpha+\beta)^{\geq i}(x,y)|))\eta_{x,y}\right)
    \\=&\exp(-\mathcal{H}^{\geq i}(x-1,y))+(\exp(-\mathcal{H}^{\geq i}(x,y-1)-\exp(-\mathcal{H}^{\geq i}(x-1,y)))\eta_{x,y}
    \\=&(1-\eta_{x,y})\exp(-\mathcal{H}^{\geq i}(x-1,y))+\eta_{x,y}\exp(-\mathcal{H}^{\geq i}(x,y-1)).
\end{split}
\end{equation*}

Similarly, at a vertex $(x,x)$, we have
\begin{equation*}
    \beta^{\geq i}(x,x)=\mathcal{H}^{\geq i}(x-1,x)-\mathcal{H}^{\geq i}(x-1,x-1),
\end{equation*}
and so
\begin{equation*}
\begin{split}
    &\exp(-\mathcal{H}^{\geq i}(x,x))
    \\=&\exp(-\mathcal{H}^{\geq i}(x-1,x-1)-(|(\delta-\alpha)^{\geq i}(x,x)|))
    \\=&\exp(-\mathcal{H}^{\geq i}(x-1,x-1))\left(\exp(-|\beta^{\geq i}(x,x)|)+(1-\exp(-|\beta^{\geq i}(x,x)|))\eta_{x,x}\right)
    \\=&(1-\eta_{x,x})\exp(-\mathcal{H}^{\geq i}(x-1,x))+\eta_{x,x}\exp(-\mathcal{H}^{\geq i}(x-1,x-1)).
\end{split}
\end{equation*}
Thus, $\exp(-\mathcal{H}^{\geq i}(x,y))=Z^{\geq i}(x,y)$ for all $i$ and $(x,y)$ under this coupling, so they have the same distribution.
\end{proof}

\begin{remark}
It is clear from the fusion procedure to obtain the beta polymer that when we set $t=0$, we introduced an asymmetry into the model responsible for the edges of weight $0$. It would be interesting to see if the restriction that $t=0$ could be lifted. In \cite{ML19}, solutions to the reflection equation are found without this restriction. However, they only study the rank $1$ case, and moreover also do not distinguish negative colors (essentially, they treat the case when the only colors are $\pm 1$). It would be interesting to see if their techniques could be extended to study the boundary vertex weights in the colored case while keeping track of negative colors. We conjecture that such a formula would converge in the $\varepsilon\to 0$ limit to a distribution on the boundary of the half space beta polymer where all three edges associated to vertices $(x,x)$ have non-zero probabilities.
\end{remark}

The Theorem \ref{thm: shift inv beta} on shift invariance for the beta polymer is now an immediate corollary of Theorem \ref{thm: shift inv cont} and Lemma \ref{lem: beta cont same}.

\section*{Acknowledgments}
The author would like to thank Alexei Borodin for his guidance throughout this project. The author also thanks Guillaume Barraquand, Sergei Korotkikh, and Dominik Schmid for helpful discussions.

\bibliography{bibliography}{}
\bibliographystyle{abbrvurl}

\appendix

\section{Proof of Theorem \ref{thm: fused bd wt}}
\label{app: pf}
In this appendix, we provide the proof of Theorem \ref{thm: fused bd wt}. Conceptually, the proof is straightforward. Since Lemma \ref{lem: bd wt rec} gives a recurrence which characterizes the fused boundary vertex weights, we can simply check that the claimed formula also satisfies the same recurrence and has the same value for $L=1$. Unfortunately, the details are quite tedious and involved.

Before we begin, let us remark that the lack of explicit dependence on the negative colors in \eqref{eq: fused wt bd 2} should be expected. This is because we set $t=0$, which means that positive colors, once below the diagonal, can never return. Thus, knowledge of $\Bb$ and $\Cb^+$ is enough to recover $\Cb^-$. We simply take $\Bb^-$ and add in a negative color for each color in $\Bb^+-\Cb^+$.

In particular, this means that we may forget about the negative and $0$ arrows in the recurrence given by Lemma \ref{lem: bd wt rec}, and simply keep track of a system with positive arrows. Thus, for the rest of the proof, we set $\Bb=\Bb^+$ and $\Cb=\Cb^+$, restrict the sums to those over $i,j,k>0$, and treat any negative colors as $0$.

\subsection{Recurrence for fused boundary weights}
We first recall that Lemma \ref{lem: bd wt rec} gives the recurrence
\begin{equation}
\label{eq: fused wt bd rec}
\begin{split}
    &W_{L+1}(z,q,\nu;\Bb,\Cb)
    \\=&\sum_{i,j,k} \frac{(1-q^{-\Bb_i})q^{-|\Bb^{\geq i+1}|}}{1-q^{-L}}W_{L}(z,q,\nu;\Bb-\mathbf{e}_i,\Cb-\mathbf{e}_i-\mathbf{e}_k+\mathbf{e}_j) 
    \\&\qquad \times W_1(zq^{2L},q,\nu;j,k) W_{L,1}(zq^{L},q;\Cb-\mathbf{e}_i-\mathbf{e}_k+\mathbf{e}_j,i,\Cb-\mathbf{e}_k,j)
\end{split}
\end{equation}
for $W_{L}$. We can verify that if $L=1$, $W_L$ evaluates to
\begin{alignat*}{2}
    &W_1(z,q,\nu; b,b,b,b)=1, &&W_1(z,q,\nu;b,-b,-b,b)=1,
    \\&W_1(z,q,\nu;-b,b,-b,b)=\frac{z-1}{z+\sqrt{z}/\nu}, \qquad   &&W_1(z,q,\nu;-b,b,b,-b)=\frac{1+\sqrt{z}/\nu}{z+\sqrt{z}/\nu},
\end{alignat*}
where $b>0$, and is $0$ for all other cases. This indeed matches the vertex weights for the unfused model with $t=0$.

We now turn to checking the recurrence. We note that the triple sum in \eqref{eq: fused wt bd rec} is essentially a double sum, since either $j=k$ or $k=0$. Thus, we split our sum into the following cases:
\begin{enumerate}
    \item $j=k=0$:
    \begin{enumerate}[label=(\theenumi\alph*)]
        \item $i=j=k=0$,
        \item $i>0$, $j=k=0$,
    \end{enumerate}
    \item $k=j>0$:
    \begin{enumerate}[label=(\theenumi\alph*)]
        \item $i=0$, $k=j>0$,
        \item $k=j>i>0$,
        \item $i=j=k>0$,
        \item $i>j=k>0$,
    \end{enumerate}
    \item $k=0$, $j>0$:
    \begin{enumerate}[label=(\theenumi\alph*)]
        \item $i=0$, $j>0$, $k=0$,
        \item $j>i>0$, $k=0$,
        \item $i=j>0$, $k=0$,
        \item $i>j>0$, $k=0$.
    \end{enumerate}
\end{enumerate}
We will let $S_{i,j,k}$ denote the summands on the right hand side of \eqref{eq: fused wt bd rec} divided by $W_{L+1}(z,q,v;\Bb,\Cb)$, and compute $S_{i,j,k}$ in each case. We'll then show
\begin{equation*}
    \sum _{i,j,k}S_{i,j,k}=1,
\end{equation*}
which will prove that the recurrence holds for the formula \eqref{eq: fused wt bd 2}.

\subsection{Useful identities}
Before we attempt to evaluate the $S_{i,j,k}$, we provide some formulas and identities that will be needed. We first recall that
\begin{equation*}
    W_{L,1}(z,q;\Ab,b,\Cb,d)=(zq^{L-1})^{-I(b>d)}\frac{1-(zq^{L-1})^{-I(b=d)}q^{\Ab_d}}{1-z^{-1}q}q^{|\Ab^{\geq d+1}|}
\end{equation*}
by Proposition \ref{prop: row vt}. Next, we note the following recurrence relations for $\Phi_q$, all following easily from the definition
\begin{equation*}
    \Phi_q(\Ab,\Bb;x,y)=\left(\frac{y}{x}\right)^{|\Ab|}\frac{(x;q)_{|\Ab|}(y/x;q)_{|\Bb|-|\Ab|}}{(y;q)_{|\Bb|}}q^{\sum_{i<j}(\Bb_i-\Ab_i)\Ab_j}\prod _i{\Bb_i\choose \Ab_i}_q.
\end{equation*}
We have
\begin{equation*}
    \Phi_q(\Ab-\mathbf{e}_i,\Bb-\mathbf{e}_i;x,y)=\frac{x}{y}q^{-\sum_{j<i}\Bb_j-\Ab_j}\left(\frac{1-yq^{|\Bb|-1}}{1-xq^{|\Ab|-1}}\right)\left(\frac{1-q^{\Ab_i}}{1-q^{\Bb_i}}\right)\Phi_q(\Ab,\Bb;x,y),
\end{equation*}

\begin{equation*}
    \Phi_q(\Ab,\Bb;qx,y)=q^{-|\Ab|}\left(\frac{1-q^{|\Ab|}x}{1-x}\right)\left(\frac{1-y/qx}{1-q^{|\Bb|-|\Ab|-1}y/x}\right)\Phi_q(\Ab,\Bb;x,y),
\end{equation*}
and
\begin{equation*}
    \Phi_q(\Ab+\mathbf{e}_j,\Bb;x,y)=q^{|\Bb^{\leq j-1}|+\Ab_j-|\Ab|}\left(\frac{1-q^{|\Ab|}x}{1-q^{|\Bb-\Ab|-1}y/x}\right)\left(\frac{1-q^{\Bb_j-\Ab_j}}{1-q^{\Ab_j+1}}\right)\Phi_q(\Ab,\Bb;x,y).
\end{equation*}

Let
\begin{equation*}
    \Lambda_{1}=q^{|\Bb^+-\Cb^+|}\left(\frac{1-zq^{L-1}}{1-zq^{|\Bb^+-\Cb^+|-1+L}}\right)\left(\frac{1+\sqrt{z}\nu q^{|\Bb^+|-1}}{1+q^{-L}\nu/\sqrt{z}}\right),
\end{equation*}
and
\begin{equation*}
    \Lambda_2=zq^{L-1}(-\nu \sqrt{z})^{-1}\Lambda_1\left(\frac{1+q^{|\Cb^+|-L}\nu/\sqrt{z}}{1-q^{|\Bb^+-\Cb^+|+L-2}z}\right).
\end{equation*}

When $j=k$ and $i=0$, we have
\begin{equation}
\begin{split}
    \frac{W_L(z,q,\nu;\Bb-\mathbf{e}_0,\Cb-\mathbf{e}_0)}{W_{L+1}(z,q,\nu;\Bb,\Cb)}&=q^{|\Bb^+-\Cb^+|}\left(\frac{1+q^{|\Cb^+|-L}\nu/\sqrt{z}}{1+q^{-L}\nu/\sqrt{z}}\right)\left(\frac{1-zq^{L-1}}{1-zq^{|\Bb^+-\Cb^+|-1+L}}\right)
    \\&=\Lambda_1\left(\frac{1+q^{|\Cb^+|-L}\nu/\sqrt{z}}{1+\sqrt{z}\nu q^{|\Bb^+|-1}}\right).
\end{split}
\end{equation}

When $j=k$ and $i>0$, we have
\begin{equation}
\begin{split}
    \frac{W_L(z,q,\nu;\Bb-\mathbf{e}_i,\Cb-\mathbf{e}_i)}{W_{L+1}(z,q,\nu;\Bb,\Cb)}&=z q^{L+|\Bb^+|-1}\frac{\Phi_q(\Cb-\mathbf{e}_i,\Bb-\mathbf{e}_i;-q^{1-L}\nu/\sqrt{z},-\sqrt{z}\nu)}{\Phi_q(\Cb,\Bb;-q^{-L}\nu/\sqrt{z},-\sqrt{z}\nu)}
    \\&=\Lambda_1 q^{-\sum_{0<l<i}\Bb_l-\Cb_l}\left(\frac{1-q^{\Cb_i}}{1-q^{\Bb_i}}\right),
\end{split}
\end{equation}
when $k=0$ and $i,j> 0$, we have
\begin{equation*}
\begin{split}
    &\frac{W_L(z,q,\nu;\Bb-\mathbf{e}_i,\Cb-\mathbf{e}_i-\mathbf{e}_0+\mathbf{e}_j)}{W_{L+1}(z,q,\nu;\Bb,\Cb)}
    \\=&(-\nu\sqrt{z})^{-1}q^{\sum_{0<l<j}\Bb_l+\Cb_j-|\Cb^+|+I(i>j)}\left(\frac{1+q^{|\Cb^+|-L}\nu/\sqrt{z}}{1-q^{|\Bb^+-\Cb^+|+L-2}z}\right)\left(\frac{1-q^{\Bb_j-\Cb_j}}{1-q^{\Cb_j+1-I(i=j)}}\right)
    \\&\qquad\qquad  \times\frac{W_L(z,q,\nu;\Bb-\mathbf{e}_i,\Cb-\mathbf{e}_i)}{W_{L+1}(z,q,\nu;\Bb,\Cb)}
    \\=&\Lambda_1 (-\nu\sqrt{z})^{-1}q^{\sum_{0<l<j}\Bb_l+\Cb_j-|\Cb^+|+I(i>j)}\left(\frac{1+q^{|\Cb^+|-L}\nu/\sqrt{z}}{1-q^{|\Bb^+-\Cb^+|+L-2}z}\right)
    \\&\qquad \qquad \times\left(\frac{1-q^{\Bb_j-\Cb_j}}{1-q^{\Cb_j+1-I(i=j)}}\right)q^{-\sum_{0<l<i}\Bb_l-\Cb_l}\left(\frac{1-q^{\Cb_i}}{1-q^{\Bb_i}}\right)
    \\=&\Lambda_2  q^{\sum_{0<l<j}\Bb_l+\Cb_j-|\Cb^+|+I(i>j)}\left(\frac{1-q^{\Bb_j-\Cb_j}}{1-q^{\Cb_j+1-I(i=j)}}\right)q^{-\sum_{0<l<i}\Bb_l-\Cb_l}\left(\frac{1-q^{\Cb_i}}{1-q^{\Bb_i}}\right)
    \\=&\Lambda_2  q^{\sum_{0<l<j}\Bb_l-\Cb_l-|\Cb^{\geq j+1}|+I(i>j)}\left(\frac{1-q^{\Bb_j-\Cb_j}}{1-q^{\Cb_j+1-I(i=j)}}\right)q^{\sum_{0<l<i}\Cb_l-|\Bb^+|+|\Bb^{\geq i}|}\left(\frac{1-q^{\Cb_i}}{1-q^{\Bb_i}}\right),
\end{split}
\end{equation*}
and when $k=i=0$ and $j>0$, we have
\begin{equation*}
\begin{split}
    &\frac{W_L(z,q,\nu;\Bb-\mathbf{e}_0,\Cb-\mathbf{e}_0-\mathbf{e}_0+\mathbf{e}_j)}{W_{L+1}(z,q,\nu;\Bb,\Cb)}
    \\=&(-\nu\sqrt{z})^{-1}q^{\sum_{0<l<j}\Bb_l+\Cb_j-|\Cb^+|}\left(\frac{1+q^{|\Cb^+|+1-L}\nu/\sqrt{z}}{1-q^{|\Bb^+-\Cb^+|+L-2}z}\right)\left(\frac{1-q^{\Bb_j-\Cb_j}}{1-q^{\Cb_j+1}}\right)
    \\&\qquad\qquad\times\frac{W_L(z,q,\nu;\Bb-\mathbf{e}_0,\Cb-\mathbf{e}_0)}{W_{L+1}(z,q,\nu;\Bb,\Cb)}
    \\=&\Lambda_1 (-\nu\sqrt{z})^{-1}q^{\sum_{0<l<j}\Bb_l+\Cb_j-|\Cb^+|}\left(\frac{1+q^{|\Cb^+|+1-L}\nu/\sqrt{z}}{1-q^{|\Bb^+-\Cb^+|+L-2}z}\right)\left(\frac{1-q^{\Bb_j-\Cb_j}}{1-q^{\Cb_j+1}}\right)
    \\&\qquad\qquad\times\left(\frac{1+q^{|\Cb^+|-L}\nu/\sqrt{z}}{1+\sqrt{z}\nu q^{|\Bb^+|-1}}\right)
    \\=&\Lambda_2 q^{\sum_{0<l<j}\Bb_l+\Cb_j-|\Cb^+|}\left(\frac{1+q^{|\Cb^+|+1-L}\nu/\sqrt{z}}{1+q^{|\Bb^+|-1}\sqrt{z}\nu}\right)\left(\frac{1-q^{\Bb_j-\Cb_j}}{1-q^{\Cb_j+1}}\right).
\end{split}
\end{equation*}

\subsection{Evaluation of summands}
We will now evaluate $S_{i,j,k}$ in each of the ten cases. We will always implicitly assume that the color vectors have non-negative entries (i.e. that $i,j,k$ are specialized to colors for which the weight is non-zero).

We begin with the case $j=k=0$. Then $W_1(z,q,\nu;j,k)=1$, and we have:

{\bf Case (1a):} $i=j=k=0$
\begin{equation}
\label{eq: 1a}
    S_{0,0,0}=\frac{1-q^{L+1-|\Bb^+|}}{1-q^{L+1}}\Lambda_1\left(\frac{1+q^{|\Cb^+|-L}\nu/\sqrt{z}}{1+\sqrt{z}\nu q^{|\Bb^+|-1}}\right)\left(\frac{q^{|\Cb^+|}-q^{1-L}/z}{1-q^{1-L}/z}\right).
\end{equation}

{\bf Case (1b):} $i>0$, $j=k=0$
\begin{equation}
\label{eq: 1b}
    S_{i,0,0}=\frac{1-q^{\Cb_i}}{1-q^{L+1}}\Lambda_1 q^{L+1-|\Bb^{\geq i}|-\sum_{0<l<i}\Bb_l-\Cb_l}\left(\frac{1-q^{L-|\Cb^+|}}{1-q^{1-L}/z}\right)q^{|\Cb^+|-2L}/z
\end{equation}
We can sum this over $i$ to obtain
\begin{equation}
\begin{split}
    \frac{1-q^{|\Cb^+|}}{1-q^{L+1}}\Lambda_1 q^{1-L+|\Cb^+-\Bb^+|}\left(\frac{1-q^{L+1-|\Cb^+|}}{1-q^{1-L}/z}\right)/z.
\end{split}
\end{equation}

We let
\begin{equation*}
    \Lambda_3=\left(\frac{1+q^L\sqrt{z}/\nu}{q^{2L}z+q^L\sqrt{z}/\nu}\right)\Lambda_1(1-q^{L+1})^{-1}.
\end{equation*}

Next, we consider the case $k=j>0$. Then we have:

{\bf Case (2a):} $i=0$, $k=j>0$
\begin{equation}
\label{eq: 2a}
    S_{0,j,j}=(1-q^{L+1-|\Bb^+|})\Lambda_3\left(\frac{1+q^{|\Cb^+|-L}\nu/\sqrt{z}}{1+\sqrt{z}\nu q^{|\Bb^+|-1}}\right)\left(\frac{1-q^{\Cb_j}}{1-q^{1-L}/z}\right)q^{|\Cb^{\geq j+1}|}.
\end{equation}
Summing this over $j$ gives
\begin{equation}
    (1-q^{L+1-|\Bb^+|})\Lambda_3\left(\frac{1+q^{|\Cb^+|-L}\nu/\sqrt{z}}{1+\sqrt{z}\nu q^{|\Bb^+|-1}}\right)\left(\frac{1-q^{|\Cb^+|}}{1-q^{1-L}/z}\right).
\end{equation}

{\bf Case (2b):} $k=j>i>0$
\begin{equation}
\label{eq: 2b}
    S_{i,j,j}=\Lambda_3 q^{L+1-|\Bb^{\geq i}|} q^{-\sum_{0<l<i}\Bb_l-\Cb_l}\left(1-q^{\Cb_i}\right)\left(\frac{1-q^{\Cb_j}}{1-q^{1-L}/z}\right)q^{|\Cb^{\geq j+1}|}.
\end{equation}

{\bf Case (2c):} $i=j=k>0$
\begin{equation}
\label{eq: 2c}
    S_{i,i,i}=\Lambda_3 q^{L+1-|\Bb^{\geq i}|} q^{-\sum_{0<l<i}\Bb_l-\Cb_l}\left(1-q^{\Cb_i}\right)\left(\frac{1-q^{\Cb_j-2L}/z}{1-q^{1-L}/z}\right)q^{|\Cb^{\geq j+1}|}.
\end{equation}

{\bf Case (2d):} $i>j=k>0$
\begin{equation}
\label{eq: 2d}
    S_{i,j,j}=\Lambda_3 q^{L+1-|\Bb^{\geq i}|} q^{-\sum_{0<l<i}\Bb_l-\Cb_l}\left(1-q^{\Cb_i}\right)\left(\frac{(zq^{2L})^{-1}-q^{\Cb_j-2L}/z}{1-q^{1-L}/z}\right)q^{|\Cb^{\geq j+1}|}.
\end{equation}

We can first sum over $i$ and $j$ in cases (2b), (2c), and (2d), noting that the sums factor. The sum over $j$ telescopes, giving
\begin{equation*}
    \frac{1-q^{|\Cb^+|-2L}/z}{1-z^{-1}q^{1-L}},
\end{equation*}
and the sum over $i$ telescopes, giving $1-q^{|\Cb^+|}$. Thus, the sum over cases (2b), (2c), and (2d) gives
\begin{equation}
\label{eq: 2b-2d}
    \Lambda_3 q^{L+1-|\Bb^+|}(1-q^{|\Cb^+|})\frac{1-q^{|\Cb^+|-2L}/z}{1-z^{-1}q^{1-L}}.
\end{equation}
We can now sum case (2a) with \eqref{eq: 2b-2d}, giving
\begin{equation}
\label{eq: case 2}
\begin{split}
    &\frac{\Lambda_3(1-q^{|\Cb^+|})}{(1-q^{1-L}/z)(1+\sqrt{z}\nu q^{|\Bb^+|-1})}\Bigg((q^{L+1-|\Bb^+|}-q^{|\Cb^+-\Bb^+|+1-L}/z)(1+\sqrt{z}\nu q^{|\Bb^+|-1})
    \\&\hspace{6.5cm}+(1-q^{L+1-|\Bb^+|})(1+q^{|\Cb^+|-L}\nu/\sqrt{z})\Bigg)
    \\=&\frac{\Lambda_3(1-q^{|\Cb^+|})(1-q^{|\Cb^+-\Bb^+|+1-L}/z)(1+\sqrt{z}\nu q^L)}{(1-q^{1-L}/z)(1+\sqrt{z}\nu q^{|\Bb^+|-1})}
    \\=&\Lambda_1\frac{(1-q^{|\Cb^+-\Bb^+|+1-L}/z)(1+q^{-L}\nu/\sqrt{z})}{(1-q^{L+1})(1-q^{1-L}/z)(1+\sqrt{z}\nu q^{|\Bb^+|-1})}(1-q^{|\Cb^+|})
\end{split}
\end{equation}

Let
\begin{equation}
    \Lambda_4=\Lambda_2(1-q^{L+1})^{-1}\left(\frac{q^{2L}z-1}{q^{2L}z+q^L\sqrt{z}/\nu}\right).
\end{equation}

Finally, we consider the case $k=0$ and $j>0$.

{\bf Case (3a):} $i=0$, $j>0$, $k=0$
\begin{equation}
\label{eq: 3a}
    S_{0,j,0}=\Lambda_4 (1-q^{L+1-|\Bb^+|}) q^{\sum_{0<l<j}\Bb_l-\Cb_l}\left(\frac{1+q^{|\Cb^+|-L+1}\nu/\sqrt{z}}{1+q^{|\Bb^+|-1}\sqrt{z}\nu}\right)\left(\frac{1-q^{\Bb_j-\Cb_j}}{1-q^{1-L}/z}\right).
\end{equation}
We can sum \eqref{eq: 3a} over $j>0$ which telescopes, giving
\begin{equation}
    \Lambda_4 (1-q^{L+1-|\Bb^+|}) \left(\frac{1+q^{|\Cb^+|+1-L}\nu/\sqrt{z}}{1+q^{|\Bb^+|-1}\sqrt{z}\nu}\right)\left(\frac{1-q^{|\Bb^+-\Cb^+|}}{1-q^{1-L}/z}\right).
\end{equation}

{\bf Case (3b):} $j>i>0$, $k=0$
\begin{equation}
\label{eq: 3b}
    S_{i,j,0}=\Lambda_4  q^{L+1-|\Bb^+|+\sum_{0<l<j}\Bb_l-\Cb_l+\sum_{0<l<i}\Cb_l}\left(1-q^{\Cb_i} \right)\left(\frac{1-q^{\Bb_j-\Cb_j}}{1-q^{1-L}/z}\right).
\end{equation}

{\bf Case (3c):} $i=j>0$, $k=0$
\begin{equation}
\label{eq: 3c}
    S_{i,i,0}=\Lambda_4  q^{L+1-|\Bb^+|+\sum_{0<l<i}\Bb_l}\left(\frac{1-q^{\Bb_i-\Cb_i}}{1-q^{1-L}/z}\right)(1-q^{\Cb_i}q^{-2L+1}/z).
\end{equation}

{\bf Case (3d):} $i>j>0$, $k=0$
\begin{equation}
\label{eq: 3d}
    S_{i,j,0}=\Lambda_4  q^{L+1-|\Bb^+|+\sum_{0<l<j}\Bb_l-\Cb_l+\sum_{0<l<i}\Cb_l}\left(1-q^{\Cb_i} \right)\left(\frac{1-q^{\Bb_j-\Cb_j}}{1-q^{1-L}/z}\right)q^{-2L+1}/z.
\end{equation}
Summing cases (3b), (3c) and (3d) over $i$ gives
\begin{equation*}
    \Lambda_4  q^{L-|\Bb^+|+\sum_{0<l<j}\Bb_l-\Cb_l}\left(1-q^{|\Cb^+|}q^{-2L+1}/z \right)\left(\frac{1-q^{\Bb_j-\Cb_j}}{1-q^{1-L}/z}\right),
\end{equation*}
after which we can sum over $j$ and obtain
\begin{equation*}
    \Lambda_4 q^{L+1-|\Bb^+|}\left(1-q^{|\Cb^+|-2L+1}/z \right)\left(\frac{1-q^{|\Bb^+-\Cb^+|}}{1-q^{1-L}/z}\right),
\end{equation*}
and summing with case (3a) gives
\begin{equation}
\label{eq: case 3}
\begin{split}
    &\frac{\Lambda_4(1-q^{|\Bb^+-\Cb^+|})(1-q^{|\Cb^+-\Bb^+|+2-L}/z)(1+\sqrt{z}\nu q^L)}{(1-q^{1-L}/z)(1+\sqrt{z}\nu q^{|\Bb^+|-1})}
    \\=&\frac{\Lambda_1(1+q^{|\Cb^+|-L}\nu/\sqrt{z})}{(1-q^{L+1})(1-q^{1-L}/z)(1+\sqrt{z}\nu q^{|\Bb^+|-1})}(q^{2L}z-1)(1-q^{|\Bb^+-\Cb^+|})z^{-1}q^{|\Cb^+-\Bb^+|+1-L}.
\end{split}
\end{equation}
Adding \eqref{eq: case 3} to \eqref{eq: 1a}, the cancellation comes from
\begin{equation*}
\begin{split}
    &(q^{2L}z-1)(1-q^{|\Bb^+-\Cb^+|})z^{-1}q^{|\Cb^+-\Bb^+|+1-L}+(1-q^{L+1-|\Bb^+|})(q^{|\Cb^+|}-q^{1-L}/z)
    \\=&q^{|\Cb^+|}(1-q^{1-L-|\Bb^+|}/z)(1-q^{L+1-|\Cb^+|}),
\end{split}
\end{equation*}
giving a total of
\begin{equation*}
    \frac{\Lambda_1(1-q^{L+1-|\Cb^+|})q^{|\Cb^+|}}{(1-q^{L+1})(1-q^{1-L}/z)}\frac{(1-q^{1-L-|\Bb^+|}/z)(1+q^{|\Cb^+|-L}\nu/\sqrt{z})}{(1+\sqrt{z}\nu q^{|\Bb^+|-1})}.
\end{equation*}
We then add this to \eqref{eq: 1b}, using the cancellation
\begin{equation*}
\begin{split}
    &\frac{(1-q^{1-L-|\Bb^+|}/z)(1+q^{|\Cb^+|-L}\nu/\sqrt{z})}{(1+\sqrt{z}\nu q^{|\Bb^+|-1})}+(1-q^{|\Cb^+|})q^{1-2L-|\Bb^+|}/z
    \\=&\frac{(1-q^{1-L+|\Cb^+-\Bb^+|}/z)(1+q^{-L}\nu/\sqrt{z})}{(1+\sqrt{z}\nu q^{|\Bb^+|-1})}
\end{split}
\end{equation*}
to obtain
\begin{equation*}
    \frac{\Lambda_1(q^{|\Cb^+|}-q^{L+1})}{(1-q^{L+1})(1-q^{1-L}/z)}\frac{(1-q^{1-L+|\Cb^+-\Bb^+|}/z)(1+q^{-L}\nu/\sqrt{z})}{(1+\sqrt{z}\nu q^{|\Bb^+|-1})}.
\end{equation*}
Finally, we add this to \eqref{eq: case 2}, and obtain $1$ as desired.

\end{document}